\documentclass[11pt,a4paper]{article}
\usepackage[utf8]{inputenc}
\makeatletter

\usepackage[english]{babel}
\usepackage{csquotes}

\usepackage[T1]{fontenc}
\usepackage{microtype}
\usepackage{soul}
\g@addto@macro\bfseries{\boldmath} 

\usepackage[style=alphabetic,backend=biber]{biblatex}
\usepackage{fullpage}
\usepackage{multicol}
\usepackage[symbol]{footmisc}
\usepackage{subcaption}
\usepackage{caption}

\usepackage{pgfplots}
\usepackage{tikz}
\usetikzlibrary{patterns}

\tikzstyle{point}=[fill=black, draw=none, shape=circle]

\tikzstyle{grid}=[-, line width=0.02mm]
\tikzstyle{axis}=[->, line width=0.5mm]
\tikzstyle{graduation}=[-, line width=.5mm]
\tikzstyle{diagonal}=[-, line width=.7mm, dash pattern=on 3mm off 1mm]
\tikzstyle{lorenz_curve}=[-, line width=.7mm]
\tikzstyle{function_bg}=[-, line width=.3mm]
\tikzstyle{function_fg}=[-, line width=.9mm, dash pattern=on 3mm off 1mm]
\tikzstyle{function_fg2}=[-, line width=.9mm, dotted]
\tikzstyle{tangent line}=[-, line width=.3mm]
\tikzstyle{S_mu}=[-, draw=none, pattern=dots]
\tikzstyle{doublearrow}=[<->, line width=.5mm]
\tikzstyle{forbidden}=[-, line width=0.02mm, pattern=north east lines]
\tikzstyle{doublearrow_anchor}=[-, line width=0.02mm, dash pattern=on 1.5mm off .5mm]
\tikzstyle{projection_grad}=[-, line width=0.1mm, dash pattern=on 1mm off 1mm]

\usepackage{amsmath}
\usepackage{amssymb}
\usepackage{mathtools}
\usepackage{textcomp}
\usepackage{mathrsfs}
\allowdisplaybreaks

\newcommand{\R}{\mathbf{R}}
\newcommand{\N}{\mathbf{N}}
\newcommand{\prob}{\mathbf{P}}
\newcommand{\Exp}{\mathbf{E}}
\newcommand{\1}{\mathbf{1}}
\newcommand{\Rbarplus}{\overline{\R_+}}
\newcommand{\lb}{\text{\textnormal{\textbf[}}}
\newcommand{\rb}{\text{\textnormal{\textbf]}}}
\newcommand{\Borel}{\mathfrak B}
\newcommand{\Leb}{\mathscr L}
\newcommand{\Lorspace}{\mathfrak L}
\newcommand{\quantilespace}{\mathfrak Q}
\newcommand{\eps}{\varepsilon}
\renewcommand{\le}{\leqslant}
\renewcommand{\ge}{\geqslant}
\newcommand{\dd}{\;\mathrm{d}}
\newcommand{\id}{\mathrm{id}}
\newcommand{\M}{\mathbf{M}}
\newcommand{\KDE}{\mathrm{KDE}}
\newcommand{\indep}{\perp \!\!\! \perp}
\newcommand{\Lone}{\mathrm{L}^1}
\newcommand{\Wone}{{\mathrm{W}_1}}
\newcommand{\weak}{{\mathscr{W}}}
\newcommand{\eqdef}{\,{}{:=}{}\,}
\newcommand{\longrightarrowsim}{\overset{\sim\,}{\longrightarrow}}

\newcommand{\subalign}[1]{%
	\vcenter{%
		\Let@ \restore@math@cr \default@tag
		\baselineskip\fontdimen10 \scriptfont\tw@
		\advance\baselineskip\fontdimen12 \scriptfont\tw@
		\lineskip\thr@@\fontdimen8 \scriptfont\thr@@
		\lineskiplimit\lineskip
		\ialign{\hfil$\m@th\scriptstyle##$&$\m@th\scriptstyle{}##$\hfil\crcr
			#1\crcr
		}%
	}%
}

\newcommand\nobreakpar{\par\nobreak\@afterheading}

\newcommand{\centerstar}{\begin{center}*\end{center}}

\usepackage{enumitem}
\usepackage[colorlinks=true, allcolors=magenta]{hyperref}
\usepackage{amsthm}
\usepackage{theoremref}

\newenvironment{enumhypos}{\begin{enumerate}[label=(\textit{\roman*})]}{\end{enumerate}}

\usepackage{mathtools}
\newtagform{bold}[\bfseries]{\bfseries(}{\bfseries)}
\usetagform{bold}
\renewcommand{\eqref}[1]{\textup{(\ref{#1})}}
\numberwithin{equation}{section}

\newtheoremstyle{bfnote}%
{}{}%
{\itshape}{}%
{\bfseries}{.}%
{ }%
{\thmname{#1}\thmnumber{ #2}\thmnote{ (#3)}}

\theoremstyle{bfnote}
\newtheorem{theorem}[equation]{Theorem}
\newtheorem{proposition}[equation]{Proposition}
\newtheorem{lemma}[equation]{Lemma}
\newtheorem{corollary}[equation]{Corollary}
\newtheorem{application}[equation]{Application}

\newtheoremstyle{definition-bfnote}%
{}{}%
{\upshape}{}%
{\bfseries}{.}%
{ }%
{\thmname{#1}\thmnumber{ #2}\thmnote{ (#3)}}
\theoremstyle{definition-bfnote}
\newtheorem{counterex}[equation]{Counter-example}

\newtheorem{definition}[equation]{Definition}

\title{Alternate definitions of Gini, Hoover and Lorenz measures of inequalities and convergence with respect to the Wasserstein $\Wone$ metric}
\author{Valentin Melot* \\\\ {\footnotesize *~Inspection générale des finances~(IGF) / General Inspectorate of Finance,} \\{\footnotesize French Ministry of Economics and Finance, Paris.}}
\date{Sept. 2024}

\makeatother

\begin{document}
\maketitle
\begin{abstract}
This article focuses on some properties of three tools used to measure economic inequalities with respect to a distribution of wealth $\mu$: Gini coefficient $G$, Hoover coefficient or Robin Hood coefficient $H$, and the Lorenz concentration curve $L$. To express the distributions of resources, we use the framework of random variables and abstract Borel measures, rather than discrete samples or probability densities. This allows us to consider arbitrary distributions of wealth, e.g. mixtures between discrete and continuous distributions.

In the first part (sections~\ref{section_altdefs_lorenz}--\ref{section_altdefs_gini_hoover_application}), we discuss alternate definitions of $G$, $H$ and $L$ that can be found in economics literature. The Lorenz curve is defined as the normalized integral of the quantile function (\cite{gastwirth1971}), which is not the same as saying “$L(p)$ is the share of wealth owned by the $100p$ first centiles of the population” (proposition~\ref{lorenz_definition_match_condition}). The Gini and Hoover coefficients are introduced in terms of expectation of random variables. In section~\ref{section_altdefs_gini_hoover}, we interpret Gini and Hoover as geometrical properties of the Lorenz curve (theorem \ref{theorem_alternate_def_gini} and corollary \ref{theorem_alternate_def_hoover_max}). In particular, we give a more general and straightforward proof of the main result of \cite{dorfman1979}. Section~\ref{section_altdefs_gini_hoover_application} gives two direct applications. We \emph{en route} prove the (not trivial) fact that the Lorenz curve fully characterizes a distribution, up to a rescaling (proposition~\ref{bijection_M_Lorspace_R}).

The second part of the article (section~\ref{section_convergence_results}--\ref{section_weaker_asumptions_cv}) focuses on the consistency of $G(\mu)$, $H(\mu)$ and $L_\mu$ as $\mu$ is approximated or perturbated. The relevant tool to use is the Wasserstein metric $\Wone$, i.e. the $\Lone$ metric between quantile functions. $\Wone(\mu_n, \mu_\infty) \to 0$ if and only if underlying random variables converge in distribution and the total amount of wealth converges. In theorem~\ref{characterization_lorenz_convergence} and proposition~\ref{convergence_indicators}, we show that if $\Wone(\mu_n, \mu_\infty) \to 0$, then $G(\mu_n) \to G(\mu_\infty)$, $H(\mu_n) \to H(\mu_\infty)$ and $L_{\mu_n} \to L_{\mu_\infty}$ uniformly. Subsection~\ref{subsection_topo} discusses topological implications of this fact. Thus, applications \ref{convergence_corollary_noise}, \ref{convergence_corollary_sampling}, \ref{convergence_corollary_quantiles}, \ref{convergence_corollary_quantiles_and_sampling} and \ref{convergence_corollary_sampling_kernel} justify that the empirical Gini, Hoover indexes and Lorenz curves computed on a sample or rebuilt with partial information converge to the real Gini, Hoover indexes and Lorenz curve as information increases. Eventually, in section~\ref{section_weaker_asumptions_cv}, we discuss the situations where the $\Wone$ convergence is not ensured, but weaker asumptions can be made (convergence in distribution in~\ref{subsection_weak_convergence}, convergence of means in~\ref{subsection_convergence_means} or uniform integrability in~\ref{subsection_ui})
\end{abstract}

\begin{multicols}{2}[]
\paragraph{Introduction.} We discuss inequality indexes about a given resource, for instance income $x$, capital $k$, hours of work $w$ or economic utility $U(x, k, w)$.

The repartition of the resource can be modelized by a random variable $X$ on a probability space $(\Omega, \mathscr F, \prob)$.  Most articles either modelize $X$ as a discrete random variable, or as a density with respect to Lebesgue measure. (See for instance \cite{chotikapanich2008} on different ways to modelize incomes.) However, it is sensible to chose more complex modelizations, for instance a mixture of a Dirac mass in 0 (people with no gross income) and a random variable with density such as a gamma or a lognormal for the people with a nonzero income.

What interests us in order to study inequalities is not the variable $X$ itself, which depends on the underlying probability space, but its distribution $\mu$, which is a Borel measure on $\R_+$ --- we restrict to nonnegative variables. In the following article, we will use the formalism of measure theory, as fully introduced in \cite{legall2022}. The details on the notations we use can be found in appendix~\ref{appendix_notations}. In particular, we note $\Borel$ the Borel $\sigma$-algebra on a topological space and $\Leb$ the Lebesgue measure on a measurable subset of $\R$. $\mathscr M_1(\mathcal X, \mathscr F)$ is the set of measures on the measurable space $(\mathcal X, \mathscr F)$ with mass $1$, i.e. of probability distributions on $\mathcal X$.

If $\mu$ is a distribution, $F_\mu$ is its cumulative distribution function (c.d.f.), $m_\mu$ its mean (i.e. $\Exp[X]$ where $X \sim \mu$) and $Q_\mu$ its quantile function. The definition and basic properties of quantile functions, which will be used very often in this article, are summed up in appendix~\ref{appendix_quantile_functions}. In particular, we freely use the “Galois inequalities”, i.e. the fact that for all $p \in [0,1)$ and $q \in \R_+$, $Q_\mu(p) \le q \iff p \le F_\mu(q)$.

For any $\alpha > 0$ we note $\M_\alpha$ the set of probability distributions on $(\R_+, \Borel)$ such that $m_\mu = \alpha$ and $\M \eqdef \bigcup_{\alpha > 0} \M_\alpha$ the set of probability distributions with nonzero, finite means.

In most cases in socioeconomics, we want inequality indicators not be affected by uniform rescaling. (E.g. the indicators need not change if the incomes are expressed in cents rather than in dollars. To take into considelation the fact that the marginal benefit of $\$1$ is not constant, we can apply a concave transformation $U$ to the variable $X$). We say that $\mu$ and $\nu$ are equal up to a scale factor of $\alpha$ if and only if there exist two random variables $X$ and $Y$ on some probability space such that $X = \alpha Y$ almost surely (a.s.), $X \sim \mu$ and $Y \sim \nu$. We write $\mu \equiv \nu$ iff $\mu$ and $\nu$ are equal up to any scale factor $\alpha > 0$.

\end{multicols}

\part{Alternate definitions of Lorenz curve, Gini index and Hoover index}

\begin{multicols}{2}[\section{Alternate definitions of the Lorenz curve}\label{section_altdefs_lorenz}]
We introduce the definition of the Lorenz curve formalized by \cite{gastwirth1971}. Then we discuss the relations that exist between this formalization and a more intuitive definition.

\subsection{Integral of quantile definition and elementary properties}\label{section_introduction_lorenz_and_first_properties}
\begin{definition}\label{def_lorenz}
Let $\mu \in \M$ and $Q_\mu$ its quantile function. The \emph{Lorenz function} of the measure $\mu$ is:
\[ \begin{array}{rcrcl}
L_\mu&:&[0,1]&\longrightarrow&[0,1]\\
&&p&\longmapsto&\frac{\displaystyle\int_0^p Q_\mu(t) \dd t}{\displaystyle\int_0^1 Q_\mu(t)\dd t}.
\end{array} \]

The \emph{Lorenz concentration curve} of $\mu$ is the representative curve of the function $L_\mu$ in an orthonormal frame.
\end{definition}

If $X$ is a random variable with $X \sim \mu$, we allow to write $L_X \eqdef L_\mu$.

The facts that $Q_\mu$ is not defined in $1$ and that $Q_\mu(1^-)$ need not be finite are not an issue in this definition. Indeed, the denominator integral is equal to $m_\mu$ (see lemma~\ref{lemma_mean_quantile} in appendix). Thus, it is nonzero and finite. It follows that the numerator integral is also finite.

\paragraph{Basic properties.} We give a few elementary properties of the Lorenz function (see for instance \cite{thompson1976}). Let $\mu \in \M$, $m$ its mean. We have:\nobreakpar
\begin{itemize}
\item $L_\mu$ is nondecreasing.
\item $L_\mu(0) = 0$ and $L_\mu(1) = 1$.
\item $L_\mu$ is continuous, and even absolutely continuous (see for instance \cite[chapter~VII, theorem~1, p.~99]{hartman1961}).
\item For all $t \in (0,1]$, $L_\mu$ has a left derivative $\partial_- L_\mu(t) = \frac{Q_\mu(t)}{m}$. Likewise, for all $t \in [0,1)$, it has a right derivative $\partial_+ L_\mu(t) = \frac{Q_\mu(t^+)}{m}$. These results are inferred from the fact that $Q_\mu$ is left-continuous (proposition~\ref{quantile_continuity} in appendix) and has a right-limit everywhere.
\item As a corollary, $Q_\mu$ being nondecreasing, $L_\mu$ is a convex function.
\item Thus, for all $p \in [0,1]$, $L_\mu(p) \le p$.
\item The transformation ${\mu \longmapsto L_\mu}$ is scale-invariant, i.e. $L_\mu = L_\nu$ as soon as $\mu \equiv \nu$.
\end{itemize}

\paragraph{Finite case.} Let $n \in \N^*$. For any vector $\mathbf x \eqdef (x_1, \dots, x_n) \in \R_+^n$, let $\mathbf x^\uparrow$ the vector with same components but reordered increasingly. Let $\hat \mu$ the empirical measure associated with sample $\mathbf x$. Then for every $k \in \lb 0,n \rb$, we have $\int_0^{k/n} Q_\mu(p) \dd p = \frac 1 n \sum_{i=1}^k x^\uparrow_i$. Hence:
\[ L_{\hat \mu}\left({\textstyle\frac k n}\right) = \frac{\sum_{i=1}^k x^\uparrow_i}{\sum_{i=1}^n x_i} \]
and $L_{\hat \mu}$ is affine of the intervals of form $\left[\frac i n, \frac{i+1} n\right]$.

\paragraph{Majorization and the Lorenz order.} A vector $\mathbf x \eqdef (x_1, \dots, x_n)$ is said to \emph{majorize} another vector $\mathbf y \eqdef (y_1, \dots, y_n)$ if ${\sum_{i=1}^n x_i = \sum_{i=1}^n y_i}$ and for every $k \in \lb1,n\rb$, ${\sum_{i=1}^k x_i^\uparrow \le \sum_{i=1}^k y_i^\uparrow}$. If so, a distribution of incomes $x_1, \dots, x_n$ is more unequal than $y_1, \dots, y_n$ and has same mean. See \cite{marshall2011} for a complete course on majorization theory.

Let $\mu$ and $\nu$ the empirical measures associated with $\mathbf x$ and $\mathbf y$. for every $k \in \lb 0,n\rb$, $L_\mu\left(\frac k n\right) \le L_\nu\left(\frac k n\right)$. Hence, for all $p \in [0,1]$, $L_\mu(p) \le L_\nu(p)$.

We say that a measure $\mu \in \M$ \emph{Lorenz-dominates} $\nu \in \M$ if $L_\mu(p) \le L_\nu(p)$ for all $p \in [0,1]$, i.e. $L_\mu \le L_\nu$. From what preceeds, follows that Lorenz domination naturally extends majorization. In particular, one can say that $\mu$ is more inequal than $\nu$ if $L_\mu \le L_\nu$. See \cite{arnold1987} and \cite[chapter 17.C]{marshall2011} for a brief introduction to Lorenz order.

\subsection{Some other intuitive definitions and how to deal with atoms}\label{section_alt_lorenz}
\subsubsection{The pseudo-Lorenz function}

The Lorenz function is often defined intuitively by economists as “\emph{the function that maps $p$ to the proportion of the total resource owned by the bottom $100p\%$-share of the total population}”. However, this definition is not appropriate.
\begin{definition}\label{def_pseudo_lorenz}
	Let $\mu \in \M$, $Q_\mu$ its quantile function. We call \emph{pseudo-Lorenz function of $\mu$} the function $\Lambda_\mu: [0,1] \longrightarrow [0,1]$ such that for all $p \in [0,1)$,
	\[ \Lambda_\mu(p) \eqdef \frac{\displaystyle \int_0^{Q_\mu(p)} u \dd\mu(u)}{\displaystyle \int_0^{\infty} u \dd\mu(u)}. \]
	and $\Lambda_\mu(1) \eqdef 1$.
\end{definition}

$\Lambda_\mu$ formalises the intuitive definition presented above, using the quantiles formalism. Notice that the denominator is exactly $m_\mu$.

Do $\Lambda$ and $L$ coincide? It is clear that the answer is no as soon as $\mu$ has nonzero atoms. In this case, most properties of $L$ proven before do not hold. For instance, if $\mu$ is a Dirac mass $\mu \eqdef \delta_x$ (with $x > 0$), then $Q_\mu(0) = 0$ and $Q_\mu(p) = x$ for all $p \in (0,1]$, hence $\Lambda_\mu(0) = 0$ and $\Lambda_\mu(t) = 1$ for all $t > 0$. Thus, $\Lambda_\mu$ is not continuous at 0.

Actually, there is a relation between $\Lambda_\mu$ and $L_\mu$, as stated in the following lemma:
\begin{lemma}\label{relation_lambda_l}
	Let $\mu \in \M$. Then, for all ${p \in [0,1)}$,
	\[ \Lambda_\mu(p) - L_\mu(p) = \frac{Q_\mu(p)}{m_\mu}\cdot[F_\mu(Q_\mu(p))-p]. \]
\end{lemma}

In order to ease the notation, we write $Q$, $F$, $m$, $L$ and $\Lambda$ as an abbreviation for $Q_\mu$, $F_\mu$, etc. if no ambiguity.

\begin{proof}
	We perform a pushforward change-of-variable through $Q$. Let ${p \in[0,1)}$. First notice that:
	\[\int_0^{Q(p)} u\dd\mu(u) = \int_{0^+}^{Q(p)} u\dd\mu(u).\]
	
	Then, applying the law of unconscious statistician (lemma \ref{pushforward_quantile} in appendix) with the function
	\[\begin{array}{rcrcl}f &:& \R_+ &\longrightarrow& \R_+\\ && x & \longmapsto & x \cdot \1_{(0,Q(p)]}(x),\end{array}\] we have:
	\begin{align*}
		\int_{0^+}^{Q(p)} x\dd\mu(x) &= \int_0^\infty f(x) \dd\mu(x)\\
		&= \int_0^1 f(Q(u)) \dd u\\
		&= \int_0^1 Q(u) \cdot  \1_{(0, Q(p)]}(Q(u)) \dd u\\
		\int_{0^+}^{Q(p)} x\dd\mu(x) &= \int_{Q^{-1}\langle(0, Q(p)]\rangle} Q(u) \dd u.
	\end{align*}
	
	Yet by Galois inequalities,
	\[Q^{-1}\langle(0, Q(p)]\rangle = (F(0), F(Q(p))].\]
	
	Furthermore, $Q(F(0)) = 0$ so $Q(t) = 0$ for all $t \in [0, F(0)]$. Eventually, we get:
	\begin{align*}
		\int_0^{Q(p)} u\dd\mu(u) = \int_0^{F(Q(p))} Q(t) \dd t.
	\end{align*}
	
	As $p \le F(Q(p))$, we deduce that for all $p \in [0,1)$,
	\[\Lambda(p) - L(p) = \frac{1}{m}\int_{p}^{F(Q(p))} Q(u)\dd u \ge 0.\]
	
	Yet $Q$ is a constant function over $[p, F(Q(p))]$, since it is nondecreasing and takes same values at the endpoints of the interval. This concludes the proof.
\end{proof}

From this result and the immediate properties of the quantile function, we deduce the conditions under which $\Lambda$ and $L$ match:

\begin{proposition}\label{lorenz_definition_match_condition}
	Let $\mu \in \M$.
	
	For all $p \in [0,1]$, $L_\mu(p) \le \Lambda_\mu(p)$. Furthermore, ${L_\mu(p) = \Lambda_\mu(p)}$ iff one of the following assumptions holds:
	\begin{enumhypos}
		\item $\exists x \in \R_+, p = F(x)$.
		\item $Q(p) = 0$.
		\item $p=1$.
	\end{enumhypos}
\end{proposition}

Thus, the “intuitive” definition of Lorenz curve is valid as soon as $\mu$ has no atom (except possibly in 0), but it fails in every other case.

\subsubsection{Interpolling atomic Lorenz functions with nonatomic ones?}
Since the intuitive definition is valid as soon as $\mu$ is nonatomic, it might be tempting to deal with nondiffuse measures by approximating them with nonatomic ones, e.g. by allocating the mass of atoms on an interval, so the intuitive definition would hold.

Such a hack is, for instance, sometimes used to define the median of a measure with atoms. To illustrate this, let $\mu = \frac12 \left(\delta_{\{0\}} + \delta_{\{1\}}\right)$. Some authors define the median in such a way that the median of $\mu$ would be $\frac{1}{2}$. (Whereas $Q_\mu(0.5)$ is equal to zero, not $\frac{1}{2}$). This, however, does not lead to any satisfying result.

Note that we land on our feet with the integral-of-quantile-function definition. Indeed, one can imagine taking the individuals corresponding to an atom, splitting them into infinitely thin parts, and arbitrarily ordering them. For instance, if $\mu = \frac13 \delta_{\{1\}} + \frac23\delta_{\{2\}}$, we would say that the bottom half of the total population is the set of individuals for which $X = 1$ plus one quarter of those for which $X=2$.

What looks like this approach of “joining the points” best is in fact given by the following approach.

\subsubsection{Kendall’s parametric curve}

A last way to formalize this definition, proposed by \cite[section 2.23]{kendall1945}, is to see the Lorenz curve as a parametric curve.
\begin{definition}[Kendall curve]\label{def_kendall} Let $\mu \in \M$. The \emph{Kendall curve} of $\mu$ is the parametrized curve $\mathscr K_\mu$ of equation:
\[
\mathscr K_\mu : \left\{\begin{array}{rcl}
	x &=& \displaystyle \int_0^t \dd \mu(u) = F_\mu(t), \\
	y &=& \frac{\displaystyle \int_0^t u \dd \mu(u)}{\displaystyle \int_0^\infty u \dd \mu(u)},
\end{array}\right.( t \in \R_+).
\]
\end{definition}

Do $\mathscr K$ coincide with what we defined as the Lorenz curve, i.e. the graph of $L$? In fact, if $\mu$ has atoms then $F$ has discontinuities, the curve jumps along the $x$-axis, so it is not the graph of a function. For instance, if $\mu = \delta_x$ (with $x > 0$), the Lorenz “curve” according to Kendall’s definition is the reunion of two points $(0;0)$ and $(1;1)$. So let’s describe the relations between $\mathscr K$ and $L$.

\paragraph{$\mathscr K$ is included in the graph of $L$.} This point is easy to check. Let $t \in \R_+$. It suffices to prove that $m \cdot L(F(t)) = \int_0^t u \dd \mu(u)$. Using Galois inequalities and the pushforward formula (lemma~\ref{pushforward_quantile}), we have:
\begin{align*}
	m \cdot L(F(t)) &= \int_0^{F(t)} Q(u) \dd u\\
	&= \int_0^1 \1_{(u \le F(t))} Q(u) \dd u\\
	&= \int_0^1 \1_{(Q(u) \le t)} Q(u) \dd u\\
	&= \int_0^\infty \1_{(x \le t)} x \dd \mu(x)\\
	m \cdot L(F(t)) &= \int_0^t x \dd\mu(x).
\end{align*}

\paragraph{What happens outside of $F\langle\R_+\rangle$?} So the only difference between $\mathscr K$ and the graph of $L$ lies in the values of $x$ that are not reached by $F(t), t \in \R_+$. These points belong to two categories:\nobreakpar
\begin{itemize}
	\item $x = 1$ (if the support of $\mu$ is not bounded).
	\item The $x$ such that $F$ is discontinuous at $x$. In other words, if $x \in F\langle\R_+\rangle$ and $x < 1$, then there exists $t$ such that $F(t^-) \le x < F(t)$. The first inequality can be either srict or an equality.
\end{itemize}

We successively deal with the three cases.

\paragraph{Case $x = 1$.} As $t \to \infty$, $F(t) \to 1$ and $\int_0^t u \dd\mu(u) \to \int_0^\infty u \dd\mu(u) = m$. Hence, $(1; L(1)) = (1;1)$ is a limit point of the curve $\mathscr K$.

\paragraph{Case $F(t^-) = x < F(t)$.} As in previous caise, if $u \to t$, then $F(u) \to F(t^-) = x$. Furthermore, by Galois inequalities (inverse versions, see appendix~\ref{inverse_galois}):
\begin{align*}
	\frac{1}{m} \int_0^u v \dd\mu(v) \xrightarrow[\subalign{u & \to t\\ u &< t}]{} ~ & \frac{1}{m} \int_0^{t^-} v \dd \mu(v)\\*
	= ~ & \frac{1}{m} \int_0^\infty \1_{(v < t)} ~ v \dd \mu(v) \\
	= ~ & \frac{1}{m} \int_0^1 \1_{(Q(z) < t)} ~ Q(z) \dd z \\
	= ~ & \frac{1}{m} \int_0^1 \1_{(z < F(t^-))} ~ Q(z) \dd z\\
	\frac 1 m \int_0^u v \dd\mu(v) \xrightarrow[\subalign{u & \to t\\ u &< t}]{} ~ & L(F(t^-)) = L(x).
\end{align*}

Hence, $(x; L(x))$ is a limit point of $\mathscr K$.

\paragraph{Case $F(t^-) < x < F(t)$.}  By the Galois inequalities, $x \le F(t) \iff Q(x) \le t$.
Furthermore, $x > F(t^-) \iff Q(x) \ge t$.

Thus, $Q$ is constant over the interval $(F(t^-), F(t)]$. Hence, $L$ is affine over the same interval; since it is continuous, it is affine over $[F(t^-), F(t)]$. It follows that $(x, L(x))$ belongs to a line joining $(F(t^-); L(F(t^-))$  to $(F(t); L(F(t)))$. $(F(t^-); L(F(t^-))$ is either a point of $\mathscr K$ or a limit point thereof (see case~2), while $(F(t); L(F(t))) \in \mathscr K$.

These results can be summarized as follow:
\begin{proposition}
	Let $\mu \in \M$.
	
	The Lorenz curve of $\mu$, i.e. the graph of the function $L_\mu$, is the union of the following sets:\nobreakpar
	\begin{itemize}
		\item The Kendall curve $\mathscr K_\mu$;
		\item The set of limit points of $\mathscr K_\mu$;
		\item For every $t \in \R_+$ such that $F_\mu$ is not continuous at $t$, the segment joining the points of coordinates $(F_\mu(t^-); L_\mu(F_\mu(t^-)))$  and $(F_\mu(t); L_\mu(F_\mu(t)))$.
	\end{itemize}
\end{proposition}

\bigskip

We show the differences between $L_\mu$, $\Lambda_\mu$ and $\mathscr K_\mu$ for $\mu \eqdef \frac{1}{2}\left(\Leb_{[0,1]} + \delta_{0.5}\right)$ in fig.~\ref{fig_comparing_defs_lorenz}, p.~\pageref{fig_comparing_defs_lorenz}.
\end{multicols}

\begin{figure}[!p]
	{  \centering
		\begin{tabular}{cc}
			\subcaptionbox{Cumulative distribution function of $\mu$. $F$ has a discontinuity in $0.5$ and is right-continuous.  \label{comparing_defs_lorenz:A}}[0.45\textwidth]{\resizebox{!}{.39\textwidth}{\begin{tikzpicture}
\node [label={below left:O}] (origin) at (0, 0) {};
\draw [fill=black] (origin.center) circle (.02) ;
\draw [style=axis] (origin.center) to (0, 5.5);
\draw [style=axis] (origin.center) to (5.5, 0);

\node [label={left:1}] at (-0.125, 5) {};
\node [label={below:1}] at (5, -0.125) {};
\draw [style=graduation] (-0.125, 5) to (0.125, 5);
\draw [style=graduation] (5,-0.125) to (5, 0.125);

\draw [style=grid] (5, 0) to (5,5.5) {};
\draw [style=grid] (0,5) to (5.5, 5) {};
	
	\draw [style=graduation] (-0.125, 1.25) to (0.125, 1.25);
	\node [label={left:0.25}] at (-0.125, 1.25) {};
	\draw [style=graduation] (-0.125, 3.75) to (0.125, 3.75);
	\node [label={left:0.75}] at (-0.125, 3.75) {};
	\draw [style=graduation] (2.5,-0.125) to (2.5, 0.125);
	\node [label={below:0.5}] at (2.5, -0.125) {};
	
	\draw [style=projection_grad] (0, 3.75) -- (2.5, 3.75) -- (2.5, 0);
	\draw [style=projection_grad] (0, 1.25) -- (2.5, 1.25);

	\draw [line width=.7mm] (0,0) -- (2.5, 1.25);
	\draw [line width=.7mm] (2.5,3.75) -- (5,5) -- (5.5,5);
	\node [label={below:$F$}] at (4,4.5) {};
	\draw [fill=white, line width=.7mm] (2.5, 1.25) circle (.1); 
	\draw [fill=black, line width=.7mm] (2.5, 3.75) circle (.1); 
\end{tikzpicture}}} &
			
			\subcaptionbox{Quantile function of $\mu$. \label{comparing_defs_lorenz:B}}[0.45\textwidth]{\resizebox{!}{.39\textwidth}{\begin{tikzpicture}
\node [label={below left:O}] (origin) at (0, 0) {};
\draw [fill=black] (origin.center) circle (.02) ;
\draw [style=axis] (origin.center) to (0, 5.5);
\draw [style=axis] (origin.center) to (5.5, 0);

\node [label={left:1}] at (-0.125, 5) {};
\node [label={below:1}] at (5, -0.125) {};
\draw [style=graduation] (-0.125, 5) to (0.125, 5);
\draw [style=graduation] (5,-0.125) to (5, 0.125);

\draw [style=grid] (5, 0) to (5,5.5) {};
\draw [style=grid] (0,5) to (5.5, 5) {};
	
	\draw [style=graduation] (1.25, -0.125) to (1.25, 0.125);
	\node [label={below:0.25}] at (1.25, -0.125) {};
	\draw [style=graduation] (3.75, -0.125) to (3.75, 0.125);
	\node [label={below:0.75}] at (3.75, -0.125) {};
	\draw [style=graduation] (-0.125, 2.5) to (0.125, 2.5);
	\node [label={left:0.5}] at (-0.125, 2.5) {};
	
	\draw [style=projection_grad] (3.75, 0) -- (3.75, 2.5) -- (0, 2.5);
	\draw [style=projection_grad] (1.25, 0) -- (1.25, 2.5);
	
	\draw [line width=.7mm] (0,0) -- (1.25, 2.5) --(3.75, 2.5) -- (5,5);
	\node [label={below:$Q$}] at (2.5,2.5) {};
	
	\draw [fill=white, line width=.7mm, scale=5] (1, 1) circle (.02); 
\end{tikzpicture}}} \\
			
			\vspace{2em}&~\\
			
			\subcaptionbox{The thin line represents the Lorenz function of $\mu$ as we defined it in~\ref{def_lorenz}. The thick, dashed line is the graph of the pseudo-Lorenz $\Lambda$, which has a discontinuity in $0.25$. These functions coincide at all $p$ such that there exists an $ {x \in \R_+}$, with $p=F(x)$.
				\label{comparing_defs_lorenz:C}}[0.45\textwidth]{\resizebox{!}{.39\textwidth}{\begin{tikzpicture}
\node [label={below left:O}] (origin) at (0, 0) {};
\draw [fill=black] (origin.center) circle (.02) ;
\draw [style=axis] (origin.center) to (0, 5.5);
\draw [style=axis] (origin.center) to (5.5, 0);

\node [label={left:1}] at (-0.125, 5) {};
\node [label={below:1}] at (5, -0.125) {};
\draw [style=graduation] (-0.125, 5) to (0.125, 5);
\draw [style=graduation] (5,-0.125) to (5, 0.125);

\draw [style=grid] (5, 0) to (5,5.5) {};
\draw [style=grid] (0,5) to (5.5, 5) {};
	
	\draw [style=graduation] (-0.125, 0.625) to (0.125, 0.625);
	\node [label={left:0.125}] at (-0.125, 0.625) {};
	\draw [style=graduation] (1.25, -0.125) to (1.25, 0.125);
	\node [label={below:0.25}] at (1.25, -0.125) {};
	\draw [style=graduation] (3.75, -0.125) to (3.75, 0.125);
	\node [label={below:0.75}] at (3.75, -0.125) {};
	\draw [style=graduation] (-0.125, 3.125) to (0.125, 3.125);
	\node [label={left:0.625}] at (-0.125, 3.125) {};
	
	\draw [style=projection_grad] (0, 0.625) -- (1.25, 0.625);
	\draw [style=projection_grad] (0, 3.125) -- (1.25, 3.125) -- (1.25, 0);
	\draw [style=projection_grad] (3.75, 3.125) -- (3.75, 0);
	
	\draw [smooth, samples=100, domain=0:.25, scale=5, style=function_bg] plot(\x, {2*\x*\x});
	\draw [scale=5, style=function_bg] (.25, 0.125) -- (.75, .625);
	\draw [smooth, samples=100, domain=.75:1, scale=5, style=function_bg] plot(\x, {2*\x*\x - 2*\x + 1});
	\node [label={below:$L$}] at (2.5, 1.875) {};
	
	\draw [smooth, samples=100, domain=0:.25, scale=5, style=function_fg] plot(\x, {2*\x*\x});
	\draw [scale=5, style=function_fg] (.25, 0.625) -- (.75, .625);
	\draw [smooth, samples=100, domain=.75:1, scale=5, style=function_fg] plot(\x, {2*\x*\x - 2*\x + 1});
	\draw [fill=white, line width=.7mm, scale=5] (.25, .625) circle (.02); 
	\draw [fill=black, line width=.7mm, scale=5] (.25, .125) circle (.02); 
	\node [label={above:$\Lambda$}] at (2.5, 3) {};
\end{tikzpicture}}} &
			
			\subcaptionbox{Here again, the thin line is the graph of the Lorenz function $L$. The thick dotted line is the Kendall curve $\mathscr K$. This one is divided in two connected components $\mathscr K_1$ and $\mathscr K_2$. $\mathscr K_2$ is closed, while $\mathscr K_1$ does not include its end $(0.25; 0.125)$. These two components are included in the graph of $L$, the rest being a segment joining the ends of $\mathscr K_1$ and $\mathscr K_2$.
				\label{comparing_defs_lorenz:D}}[0.45\textwidth]{\resizebox{!}{.39\textwidth}{\begin{tikzpicture}
\node [label={below left:O}] (origin) at (0, 0) {};
\draw [fill=black] (origin.center) circle (.02) ;
\draw [style=axis] (origin.center) to (0, 5.5);
\draw [style=axis] (origin.center) to (5.5, 0);

\node [label={left:1}] at (-0.125, 5) {};
\node [label={below:1}] at (5, -0.125) {};
\draw [style=graduation] (-0.125, 5) to (0.125, 5);
\draw [style=graduation] (5,-0.125) to (5, 0.125);

\draw [style=grid] (5, 0) to (5,5.5) {};
\draw [style=grid] (0,5) to (5.5, 5) {};
	
	\draw [style=graduation] (-0.125, 0.625) to (0.125, 0.625);
	\node [label={left:0.125}] at (-0.125, 0.625) {};
	\draw [style=graduation] (1.25, -0.125) to (1.25, 0.125);
	\node [label={below:0.25}] at (1.25, -0.125) {};
	\draw [style=graduation] (3.75, -0.125) to (3.75, 0.125);
	\node [label={below:0.75}] at (3.75, -0.125) {};
	\draw [style=graduation] (-0.125, 3.125) to (0.125, 3.125);
	\node [label={left:0.625}] at (-0.125, 3.125) {};

	\draw [style=projection_grad] (0, 0.625) -- (1.25, 0.625) -- (1.25, 0);
	\draw [style=projection_grad] (0, 3.125) -- (3.75, 3.125) -- (3.75, 0);
	
	\draw [smooth, samples=100, domain=0:.25, scale=5, style=function_bg] plot(\x, {2*\x*\x});
	\draw [scale=5, style=function_bg] (.25, 0.125) -- (.75, .625);
	\draw [smooth, samples=100, domain=.75:1, scale=5, style=function_bg] plot(\x, {2*\x*\x - 2*\x + 1});
	\node [label={below:$L$}] at (2.5, 1.875) {};
	
	\draw [smooth, samples=100, domain=0:.25, scale=5, style=function_fg2] plot(\x, {2*\x*\x});
	\draw [smooth, samples=100, domain=.75:1, scale=5, style=function_fg2] plot(\x, {2*\x*\x - 2*\x + 1});
	\draw [fill=white, line width=.7mm, scale=5] (.25, .125) circle (.02); 
	\draw [fill=black, line width=.7mm, scale=5] (0,0) circle (.02); 
	\draw [fill=black, line width=.7mm, scale=5] (1,1) circle (.02); 
	\draw [fill=black, line width=.7mm, scale=5] (.75, .625) circle (.02); 
	\node [label={center:$\mathscr K_1$}] at (.5, .4) {};
	\node [label={above:$\mathscr K_2$}] at (4, 3.75) {};
\end{tikzpicture}}}\\
			\vspace{2em}&~\\
		\end{tabular}
		\caption{\label{fig_comparing_defs_lorenz} Illustration of the differences between the Lorenz function $L$, the Kendall curve $\mathscr K$ and the pseudo-Lorenz function $\Lambda$ of a measure $\mu = \frac 12 (\Leb_{[0,1]} + \delta_{0.5})$.
		\\[1em]
		Consider any probability space $(\Omega, \mathscr F, \prob)$. Let $U$ a random variable uniformly distributed in $[0,1]$ and $X$ Bernoulli with parameter $0.5$ with $X \indep U$. Let $Y = 0.5$ if $X = 0$ and $Y = U$ if $X = 1$. $\mu$ is the distribution of $Y$.}
	}
\end{figure}

\begin{multicols}{2}[\section{The Lorenz curve and the mean fully characterize a distribution}]
\label{section_lorenz_characterizes}
The goal of this section is to formally prove the (not so trivial) fact that $L_\mu = L_\nu$ only if $\mu \equiv \nu$.

\subsection{Bijection between $\M$ and $\quantilespace$}

In appendix~\ref{appendix_quantile_functions}, we recall that quantile functions of measures taken in $\mathscr M_1(\R_+, \Borel) $ are nondecreasing, left-continuous and take the value 0 in 0. In fact, these properties fully characterize the set of quantile functions.

Let $\quantilespace$ the set of nondecreasing, left-continuous functions $[0,1) \longrightarrow \R_+$ taking value $0$ in $0$.

If $f$ is a function and $\mu$ a measure, we note $f_\sharp(\mu)$ the pushforward of $\mu$ through $f$ (see appendix~\ref{appendix_notations} for details).

\begin{proposition}\label{converse_quantile}
	The following mappings are inverse bijections:
	\[ \begin{array}{rcl}
		\mathscr M_1(\R_+, \Borel) & \overset{\sim}{\longleftrightarrow} & \quantilespace\\
		\mu & \longmapsto & Q_\mu \\
		q_\sharp\left( \Leb_{[0,1)} \right) & \reflectbox{\ensuremath{\longmapsto}} & q.
	\end{array} \]
\end{proposition}

\begin{proof}
	It is a well known fact that for all $\mu \in \mathscr M_1(\R_+, \Borel)$, $\left(Q_\mu\right)_\sharp(\Leb_{[0,1]}) = \mu$ (see proposition~\ref{characterization_by_quantile} in appendix). Hence, it suffices to prove that for every $q \in \mathfrak Q$, $q$ is the quantile function of the measure $\mu = q_\sharp\left(\Leb_{[0,1)}\right)$.
	
	First, notice that for all $x \in \R_+$,
	\begin{align*}
		F_\mu(x) &= \mu([0,x])\\
		&= \Leb \left( q^{-1}\langle [0,x] \rangle \right)\\
		&= \Leb \left( \{ u \in [0,1) : 0 \le q(u) \le x \} \right)\\
		F_\mu(x) &= \Leb \left( \{q \le x\} \right),
	\end{align*}
	where we note:
	\[ \{q \le x\} \eqdef \{ u \in [0,1)~: q(u) \le x \}.\]
	
	As $q$ is nondecreasing, the set ${\{q \le x\}}$ is an interval of form $[0, s)$ or $[0,s]$, so ${F_\mu(x) = \sup\{q \le x\}}$. As we assume that $q$ is left-continuous, $s$ belongs to this interval, so the c.d.f. of $\mu$ is given by:
	\begin{align}\label{proof_converse_quantile:closed_form_F}
		F_\mu(x) &= \max \{q \le x\}.
	\end{align}
	
	Now, fix $p \in [0,1)$; we need to prove that $q(p) = Q_\mu(p)$. We distinguish two cases:\nobreakpar
	\begin{itemize}
		\item This is immediate for $p = 0$.
		\item Now assume $p > 0$. By definition of $Q_\mu(p)$, having $Q_\mu(p) = q(p)$ is equivalent to having both the following assertions true:
		\begin{enumhypos}
			\item $F_\mu(q(p)) \ge p$;\label{proof_converse_quantile:min_in_set}
			\item $\forall r \in \R_+, r < q(p) \implies F_\mu(r) < p$.\label{proof_converse_quantile:min_is_best}
		\end{enumhypos}
		
		By \eqref{proof_converse_quantile:closed_form_F}, \ref{proof_converse_quantile:min_in_set} is equivalent to having ${p \in \{q \le q(p)\}}$, which is true.
		
		Furthermore, by \eqref{proof_converse_quantile:closed_form_F}, \ref{proof_converse_quantile:min_is_best} is equivalent to
		$$ \forall r \in \R_+, r < q(p)	\implies p \notin \{q \le r\} $$
		which is also true.\qedhere
	\end{itemize}
\end{proof}

Notice that this latest equivalence fails if we do not have a \ul{$\max$} but only a \ul{$\sup$} in \eqref{proof_converse_quantile:closed_form_F}, i.e. is $q$ is not left-continuous.

\subsection{Bijection between $\M$ and $\Lorspace \times \R_+^*$}
\begin{definition}
	$\Lorspace$ is the subspace of functions ${\ell \in [0,1]^{[0,1]}}$ that satisfy all four following conditions:
	\begin{multicols}{2}
		\begin{enumhypos}
			\item $\ell$ is continuous,
			\item $\ell$ is convex,
			\item $\ell(0)=0$,
			\item $\ell(1)=1$.
		\end{enumhypos}
	\end{multicols}
\end{definition}

Notice that \emph{(i)} can be replaced by: \emph{(i') $\ell$ is continuous at $1$}. Indeed, if $\ell$ is convex, then it is continuous over $(0,1)$. Hypotheses \emph{(ii)}, \emph{(iii)} and \emph{(iv)} imply that, $\ell(p) \le p$ for all $p \in (0,1)$; hence $\ell(0^+) = 0$ and $\ell$ is continuous at 0.

\begin{proposition}\label{bijection_M_Lorspace_R} The following mapping is a bijection $\M \approx \Lorspace \times \R_+^*$:
	\[\begin{array}{rcrcl}
		\Phi &:& \M & \longrightarrowsim & \Lorspace \times \R_+^* \\
		&& \mu & \longmapsto & (L_\mu, m_\mu).
	\end{array}\]
	Its inverse is:
	\[\begin{array}{rcrcl}
		\Phi^{-1} &:& \Lorspace \times \R_+^* & \longrightarrowsim & \M\\
		&& (\ell, n) & \longmapsto &	(n \cdot \partial_- \ell)_\sharp (\Leb_{[0,1)}).\\
	\end{array}
	\]
\end{proposition}

This means, in particular, that the space of Lorenz functions is \emph{exactly} $\Lorspace$.

To prove this proposition, we first need a technical lemma:
\begin{lemma}\label{left_derivate_convex_is_left_continuous}
	Let $(a,b]$ be an interval of $\R$ and $f$ a convex function defined over $(a,b]$. The left derivative function $x \longmapsto \partial_-f(x)$ is left-continuous.
\end{lemma}
\begin{proof}
	For any $a < x < y \le b$, let $S(x, y) \eqdef \frac{f(y) - f(x)}{y - x}$. As $f$ is convex, the “chordal slope lemma” ensures that $S$ is nondecreasing with respect to each variable. Hence, for all $x \in (a,b]$, $S(x^-, x)$ exists (it may be equal to $+\infty$); it is by definition equal to $\partial_-f(x)$. Thus, $x \longmapsto \partial_-f(x)$ is nondecreasing. Thus, it admits a left-limit everywhere.
	
	Fix $z \in (a,b]$. We immediately have $\partial_-f(z^-) \le \partial_-f(z)$. Let us prove the opposite inequality.
	
	For any $x < y < z$, we have by convexity: $S(x, y) \le \partial_-f(y)$. Hence taking the left limit $y \to z^-$, we get ${S(x, z) \le \partial_-f(z^-)}$. But then taking the left limit $x \to z^-$, we get ${S(z^-, z) \le \partial_-f(z^-)}$; in other words ${\partial_-f(z) \le \partial_-f(z^-)}$.
	
	Then, $\partial_-f$ is left-continuous at $z$.
\end{proof}

\begin{proof}[Proof of the proposition~\ref{bijection_M_Lorspace_R}]
	The proof lies in two parts: (i)~proving $\Phi$ is an injection, (ii)~proving for all $(L, m) \in \Lorspace \times \R_+^*$, ${\Phi(\mu) = (L, m)}$ where $\mu = (m \cdot \partial_- L)_\sharp (\Leb_{[0,1)})$.
	
	\paragraph{$\Phi$ is an injection.}
	
	Assume $\mu$ and $\nu$ are measures such that $L_\mu = L_\nu$ and $m_\mu = m_\nu$.
	
	By taking the left derivative at any ${u \in (0,1)}$: $\frac{Q_\mu(u)}{m_\mu} = \frac{Q_\nu(u)}{m_\nu}$. The equality also stands for $u=0$. Thus, $Q_\mu = Q_\nu$. By proposition~\ref{converse_quantile}, $\mu = \nu$.
	
	\paragraph{$\Phi$ is a surjection.} Let $\ell \in \Lorspace$ and $m \in \R_+^*$.
	
	By convexity, $\ell$ admits a finite left-derivative everywhere on $(0,1]$ and a finite right derivative everywhere on $[0,1)$. Furthermore, $\partial_-\ell$ and $\partial_+\ell$ are nondecreasing, and we have for all $x \in (0,1)$,
	\[0 \le \partial_+\ell(0) \le  \partial_-\ell(x) \le \partial_+\ell(x).\]
	
	We define $q : [0,1) \longrightarrow \R_+$ by setting $q(0) = 0$ and $q(p) = \partial_- \ell(p)$ for $p > 0$. $q$ is a nondecreasing function.
	
	By lemma~\ref{left_derivate_convex_is_left_continuous}, $q$ is left-continuous. So is the function $n \cdot q$. Hence, by proposition~\ref{converse_quantile}, the measure $\mu = (n\cdot q)_\sharp(\Leb_{[0,1)})$ fullfills $q = \frac{Q_\mu}{n}$.
	
	Now we need to prove that $m_\mu = n$ and that $\ell$ is the integral of function $q$. Contrary to what may seem at first glance, this is non-trivial. (See for instance the discussion in the preliminary section of \cite[chapter~VI, p.~83]{hartman1961}.) This requires a few steps:\nobreakpar
	\begin{itemize}
		\item First, as $q$ is nondecreasing, it is $\Borel$-measurable.
		\item $\ell$ is derivable almost everywhere, and we have $\ell'(x) = q(x)$  at every $x \in (0,1)$ where $\ell$ is derivable.
		\item Take any $0 < x < 1$. The function $q$ is bounded by $q(x)$ over the interval $[0,x]$. Thus, $\ell$ is $q(x)$-Lipschitz over this interval, hence absolutely continuous over it.
		\item Hence, we infer from \cite[chapter VII, theorem~3, p.~100]{hartman1961} that $q$ is integrable over $[0,x]$ and:
		\begin{align} 
			\int_0^x q(t) \dd t = \ell(x). \label{proof_mappings L*R_M:int_q} 
		\end{align}
		\item Now we deal with $x = 1$. We set for every $n \in \N$: $g_n = \1_{[0, 1-2^{-n}]}$. Then the sequence of functions $(q\cdot g_n)_{n \in \N}$ is nondecreasing and converges pointwise to $q$ over $[0,1)$. Hence, by monotone convergence theorem:
		\begin{align*}
			\int_0^{1^-} q(t) \dd t
			&= \int_0^{1^-} \lim_{n \to \infty} q(t)~ g_n(t) \dd  t\\
			&= \lim_{n\to \infty} \int_0^{1^-} q(t)~ g_n(t) \dd t\\
			\int_{[0,1)} q(t) \dd t &= \lim_{n \to \infty} \ell(1 - 2^{-n}).
		\end{align*}
		By continuity of $\ell$, this limit is equal to $\ell(1)$, i.e. to $1$. Then, $q$ is integrable over $[0,1)$. From proposition~\ref{lemma_mean_quantile}, we get:
		\[ m_\mu = \int_0^{1^-} n \cdot q(t) \dd t = n.  \]
		
		\item Eventually, for all $x \in [0,1)$, \eqref{proof_mappings L*R_M:int_q} can be rewritten as 
		\[\frac{1}{m_\mu}\int_0^\infty Q_\mu(x) \dd x =  \ell(x)\]
		i.e. $L_\mu(x) = \ell(x)$.\qedhere
	\end{itemize}
\end{proof}

\begin{corollary}\label{corollary_rescale_lorenz}
	Let $\mu$, $\nu \in \M$. $L_\mu = L_\nu$ if and only if $\mu$ is a rescaling of $\nu$.
\end{corollary}

\subsection{The quotient set $\M / \R_+^*$}

The relation $\equiv$ is an equivalence relation. Let $\mu \in \M$ and $[\mu]$ the equivalence class of $\mu$. The mapping $\nu \in [\mu] \longmapsto m_\nu \in \R_+^*$ is one-to-one. Thus, we note $\M / \R_+^*$ the quotient set $\M /\!\!\equiv$.

It is immediate to check that the mapping:
\[ \begin{array}{rcl}
	\M & \longrightarrowsim & (\M / \R_+^*) \times \R_+^*\\
	\mu  &\longmapsto & ([\mu], m_\mu)
\end{array}
\]
is a bijection.

Furthermore, for all $\alpha > 0$, the mapping ${\mu \in \M_\alpha \longmapsto [\mu] \in \M / \R_+^*} $ is a bijection.

Finally, an immediate corollary of proposition~\ref{bijection_M_Lorspace_R} is:
\begin{corollary}
	The mapping
	\[\begin{array}{rcrcl}
		\Psi &:& \M / \R_+^* & \longrightarrowsim & \Lorspace\\
		&& [\mu] & \longmapsto & L_\mu
	\end{array}\]
	is a bijection.
\end{corollary}
In section~\ref{section_convergence_results}, we will show that all these bijections are in fact homeomorphisms for some natural topologies.
\end{multicols}

\begin{multicols}{2}[\section{Alternate definitions of the Gini and Hoover indicators}\label{section_altdefs_gini_hoover}]
	In the following section, we introduce two commonly used inequality indexes: the Gini index and the Hoover index.
	
	These indexes are directly related to geometric properties of the Lorenz curve. Regarding Gini index, this relation is known since \cite{gini1914} and commonly used by economists. As far as we know, only \cite{dorfman1979} gives a proof which is valid in almost all generality, but some restrictions remain and the redaction of the proof is somehow surprising. Hence, we give a full proof, valid in all generality as soon as $\mu \in \M$. A more detailed motivation of this section is given in appendix~\ref{appendix_need_gini_hoover}.
	
	\subsection{Gini and Hoover as mean differences and mean deviation}
	\begin{definition}\label{def_gini_hoover}
		Let $\mu \in \M$, $X, X' \sim \mu$ independent and identically distributed (i.i.d.), and $m = \Exp[X]$ the mean of $\mu$.
		
		The \emph{Gini coefficient} of $\mu$ is the ratio:
		\[G(\mu) \eqdef \frac{\Exp[|X-X'|]}{2m}.\]
		
		The \emph{Hoover coefficient} of $\mu$ is the ratio:
		\[H(\mu) \eqdef \frac{\Exp\left[\left|X-m \right|\right]}{2m}.\]
	\end{definition}

	If $X$ is a random variable on a probability space $(\Omega, \mathscr F)$ with distribution $\mu$, we allow to write $G(X) = G(\mu)$ and $H(X) = H(\mu)$.
	
	The Hoover coefficient is sometimes called the Pietra index.
	
	\paragraph{Elementary properties.} Let $\mu \in \M$ and ${(X, X') \sim \mu \otimes \mu}$. Then:\nobreakpar
	\begin{itemize}
		\item $G$ and $H$ are scale-invariant.
		\item $G(\mu), H(\mu) \in [0,1]$, by triangle inequality.
		\item $H(\mu) \le G(\mu)$. Indeed, Jensen’s inequality implies that \begin{align}
			\label{jensen_H_le_G}
			\left|\Exp[X  - X' |X]\right| \le {\Exp[|X-X'|~\mid~X]}
		\end{align}
		where $\Exp[~\bullet~\mid~X~]$ denotes the expectation conditional to $X$.
		
		Since ${\Exp[X|X] = X}$ and ${\Exp[X'|X] = m_\mu}$, the left-hand side equal to ${|X - m_\mu|}$. Taking the expectation of the inequality concludes.
		
		\item The following assertions are equivalent:
		\begin{enumhypos}
			\item $G(\mu) = 0$.
			\item $H(\mu) = 0$.
			\item $\mu$ is a Dirac mass. 
		\end{enumhypos}
		This can be deduced from the equality case of the triangle inequality in \eqref{jensen_H_le_G}.
		
		\item The upper bound $1$ is strict, i.e. $G(\mu) < 1$. Otherwise, we would have \[\Exp[|X-X'|] = \Exp[|X|] + \Exp[|-X'|],\] implying that $X$ and $-X'$ have same sign almost surely. This would imply that $X = X' = 0$ a.s., which is absurd.
		\item But $1$ is the best upper bound for both indicators, i.e. $\sup_{\mu \in \M} G(\mu) = 1$. For instance, for all $\alpha \in [0,1)$, let:
		\[\mu_\alpha = \alpha\cdot\delta_0 + (1-\alpha)\cdot\delta_1, \]
		then $G(\mu_\alpha) = H(\mu_\alpha) = \alpha$.
	\end{itemize}
	
	\centerstar 
	
	Until the end of the subsection, we fix a distribution $\mu \in \M$ and let $m$, $F$, $Q$ and $L$ respectively the mean, the c.d.f., the quantile function and the Lorenz function of $\mu$.
	
\subsection{Gini coefficient as area between the Lorenz curve and the diagonal}\label{section_alternate_def_gini}

\begin{theorem}\label{theorem_alternate_def_gini}
Let $\mu$ be a distribution in $\M$, $G(\mu)$ its Gini coefficient, and $L_\mu$ its Lorenz function. We have:
\[G(\mu) = 1 - 2 \displaystyle\int_0^1 L_\mu(p) \dd p.\]
\end{theorem}

In other words, the Gini coefficient is (up to a factor 2) the area between, on the one hand, the Lorenz curve of the considered definition, and on the other hand the diagonal. Notice that the diagonal is the Lorenz curve associated to a perfectly equal distribution.

We propose two proofs of this result. The first one is directly inspired from \cite{dorfman1979}, but more straightforward and with no hypothese on $\mu$ except $0 < m_\mu < \infty$. The second one is a new proof, which general idea is courtesy of David~Leturcq. We present it under the asumption that $\mu$ has no atom. The general idea of the proof is still valid if $\mu$ has atoms but needs a few refinements, that are fully detailed in appendix~\ref{appendix_proof_gini_lorenz_atomic_complete}.

\subsubsection{First proof (via $\int_0^1 (1 - F(t))^2 \dd t$)}
\begin{lemma}[Dorfman, 1979]\label{lemma_dorfman1979}
	Let $\mu \in \M$, $F$ its cumulative distribution function, $G$ its Gini coefficient. Then:
	\[ G = 1 - \frac{\displaystyle \int_0^\infty (1-F(t))^2 \dd t}{\displaystyle \int_0^\infty (1-F(t)) \dd t}. \]
\end{lemma}

\begin{proof}
	Notice that for any real numbers $a$ and $b$, we have~: $|a-b| = a + b - 2\min(a,b)$.
	
	Let $X$ and $Y$ be independent random variables of same distribution $\mu$ on $(\Omega, \mathscr F, \prob)$. We have:
	\[ \Exp[|X - Y|] = 2\Exp[X] - 2\Exp[\min(X,Y)].  \]
	
	Let $Z = \min(X,Y)$. We thus have:
	\[G = \frac{\Exp[|X-Y|]}{2\Exp[X]} = 1 - \frac{\Exp[Z]}{\Exp[X]}. \]
	
	Notice that:
	\[\int_{0}^\infty (1-F(t)) \dd t = \Exp[X]\]
	(see lemma~\ref{lemma_mean_survival} in appendix).
	
	Now let’s focus on the numerator. Let $F_Z$ the c.d.f. of $Z$. We notice that for all $t \in \R_+$,
	\begin{align*}
		1 - F_Z(t) &= \prob(Z > t)\\
		& = \prob((X > t) \cap (Y > t) )\\
		1 - F_Z(t) &=(1 - F(t))^2
	\end{align*}
	since $X \indep Y$.

	Lemma~\ref{lemma_mean_survival} now gives us that:
	\[\Exp[Z] = \int_0^\infty (1-F(t))^2 \dd t.\qedhere\]
\end{proof}

\begin{proof}[Proof of theorem~\ref{theorem_alternate_def_gini}] Let
	\begin{align*}
		N \eqdef& \int_{0}^\infty (1-F(t))^2 \dd t\\
		N &= \int_{t=0}^\infty \int_{u=F(t)}^1 2(1-u) \dd u \dd t.
	\end{align*}
	
	Since the Lebesgue measure is nonatomic, we can drop the lower bound of the inner integral, and then use Galois inequalities and Fubini’s theorem:
	\begin{align*}
		N &= 2\int_{t=0}^\infty \int_{u=F(x)^+}^1 (1-u) \dd u \dd t\\
		&= 2\int_{t=0}^\infty \int_{u=0}^1 (1-u) \1_{u > F(t)} \dd u \dd t\\
		&= 2\int_{t=0}^\infty \int_{u=0}^1 (1-u) \1_{Q(u) > t} \dd u \dd t\\
		&= 2\int_{u=0}^1 \int_{t=0}^{Q(u)^-} \dd t ~ (1-u)\dd u\\
		N &= 2\int_{0}^1 Q(u) (1-u) \dd u.
	\end{align*}
	
	Given the definition of $L(u) = \frac{1}{m} \int_0^u Q(t) \dd t$, an integration by parts shows that:
	\[ N = 2m \int_0^1 L(u) \dd u. \]
	(For a general proof of the integration by part formula in case where the “primitives” are not differentiable \emph{everywhere} but only absolutely continuous, see for instance \cite[chapter~VII, theorem~4, p.~103]{hartman1961}.)

	Injecting this latter expression of $N$ in the numerator of the formula in lemma \ref{lemma_dorfman1979} concludes the proof.
\end{proof}

\subsubsection{Second proof (direct computation) for nonatomic measures}\label{second_proof_equivalence_gini_simplified}
\begin{proof}
	The proof lies in switching between integrals with respect to $\mu$ and with respect to Lebesgue measure. As $\mu$ is assumed to be diffuse, one can include of exclude the bounds of the integrals without having to be cautious.
	
	Let $I \eqdef m\int_0^1 L(p)\dd p$. As $\mu$ is diffuse, thanks to proposition~\ref{lorenz_definition_match_condition}, one has
	\[I = \int_0^{1} m\cdot \Lambda(p)\dd p = \int_{p=0}^{1} \int_{u=0}^{Q(p)} u \dd \mu(u) \dd p. \]
	
	Using Galois inequalities and Fubini’s theorem, the integral can be rewritten as:
	\begin{align}
		I &= \int_{p=0}^{1} \int_{u=0}^\infty u \1_{u < Q(p)} \dd \mu(u) \dd p \nonumber \\
		&= \int_{p=0}^{1} \int_{u=0}^\infty u \1_{F(u) < p} \dd\mu(u) \dd p\nonumber\\
		&= \int_0^\infty \int_{p=F(u)}^{1} u \dd p \dd\mu(u)\nonumber\\
		I &= \int_0^\infty u\cdot(1-F(u)) \dd\mu(u) \label{proof_eq_gini_simplified:eq1}\\
		&= \int_0^\infty u \int_{s=u}^\infty \dd\mu(s) \dd\mu(u) \nonumber\\
		I &= \iint_{0 \le u \le s} u \dd\mu(u)\dd\mu(s) \label{proof_eq_gini_simplified:eq2}.
	\end{align}
	
	But from \eqref{proof_eq_gini_simplified:eq1}, we also deduce that:
	\begin{align*}
		I &= \int_0^\infty u\dd\mu(u) - \int_0^\infty \int_{s=0}^u u\dd\mu(s) \dd\mu(u)\\
		I &= m - \iint_{0 \le s \le u} u \dd\mu(s)\dd\mu(u). 
	\end{align*}

	Then, performing the permutation of variables $u \leftrightarrow s$, we write:
	\begin{align}
		I = \int_{0 \le u \le s} s \dd\mu(u) \dd\mu(s). \label{proof_eq_gini_simplified:eq3}
	\end{align}
	
	Summing up the equations \eqref{proof_eq_gini_simplified:eq2} and \eqref{proof_eq_gini_simplified:eq3} gives us:
	\[ I = \frac{m}{2} - \frac12 \iint_{0 \le u < s} (s-u) \dd\mu(s)\dd\mu(u). \]
	
	By symmetry, we have:
	\begin{align*}
		m\cdot G(\mu) &= \frac12 \iint_{\R_+^2} |s-u| \dd\mu(u)\dd\mu(s)\\
		m\cdot G(\mu) &= \iint_{0\le u<s} (s-u)\dd\mu(s) \dd\mu(u).
	\end{align*}
	
	Thus,
	\[I = \frac m 2 - \frac m 2 \cdot G(\mu)\]
	which concludes the proof.
\end{proof}

\subsection{Interpretations of the Hoover coefficient}
\subsubsection{Preliminary computations}
Let:
\begin{align*}
	P ~& \eqdef \int_0^{m^-} (m - x) \dd\mu(x)\\
	R ~& \eqdef \int_{m^+}^0 (x - m) \dd\mu(x).
\end{align*}

$P$ (resp. $R$) is the share of income that the people poorer (resp. richer) than the average own above (resp. below) the average. We have:
\[ H(\mu) = \frac{R(\mu) + P(\mu)}{2m} \]

Furthermore:
\begin{align*}
	R - P =&  \int_{m^+}^\infty x \dd\mu(x) - (1-F(m))\,m
		\\* & \qquad+ \int_{0}^{m^-} x \dd\mu(x) - F(m^-)\,m\\
	=& \int_0^\infty x \dd\mu(x) - \mu(\{m\})\,m
		\\*	&\qquad- m + m\,F(m) - m\,F(m^-)\\
	R - P =&~ 0.
\end{align*}

This implies in particular that:
\begin{align} \label{hoover_R/m}
	H(\mu) = \frac{R(\mu)}{m} = \frac{P(\mu)}{m}.
\end{align}

\subsubsection{Hoover coefficient as Robin Hood index}
\paragraph{Share redistributed.} Another classical way to see the Hoover coefficient is that it is the share of wealth to take to the people richer than the average and to redistribute to the poors in order to reach a perfect equality.

Indeed, the previous result simply means that the Hoover index is the relative share of wealth that the rich have above the average. This one is equal to the relative share of wealth that that the poors need to reach the average. Hence, taking a share $H(\mu)$ of wealth (to the people above the average) and redistributing it (to the people below the average), we reach perfect equality.

This is why the Hoover coefficient is sometimes called “Robin Hood index”. (Even though Robin Hood would not steal from the riches to give tho the poors, but from greedy government that was over-taxing their people.)

\paragraph{Most efficient way to redistribute.} Is there a more efficient way to redistribute the money in order to reach the perfect equality? Our intuition doubts it.

The answer lies in optimal transportation theory. Searching for the most efficient way to redistribute money means solving the optimal transformation problem from distribution $\mu$ to distribution $\delta_m$ with nonsymetric cost function $c(x,y) = x-y$ if $x>y$, $c(x,y) = 0$ otherwize. (One only counts the taxes \emph{taken}, not the aids and subsidies \emph{paid}.) Adopting Kantorovitch’s optimal transport problem framework (see \cite[introduction]{villani2003}), we are faced to the minimization problem:
\[ I = \min_{\pi \in \Pi(\mu, \delta_m)} \iint_{{\R_+}^2} c(x,y) \dd\pi \]
where $\Pi(\mu, \delta_m)$ is the set of measures on $(\R_+, \Borel)$ with marginals $\mu$ and $\delta_m$.

However, as $\delta_m$ is deterministic, there is only one measure in the set $\Pi(\mu, \delta_m)$, namely the tensor product $\mu \otimes \delta_m$. Hence, the minimum we are looking can be written:
\begin{align*}
	I ~&= \int_{x=0}^\infty \left( \int_{y=0}^\infty \1_{x \ge y}(x-y) \dd\delta_m(y) \right) \dd\mu(x)\\
	&= \int_{x=0}^\infty (x-m) \1_{x \ge m} \dd\mu(x)\\
	&= \int_m^\infty (x-m)\dd\mu(x)\\
	I ~&= R = m\cdot H(\mu).
\end{align*}

Thus, $H(\mu)$ is the optimal share of incomes to be redistributed.

\paragraph{Open question.} Can the Gini coefficient be interpreted in a similar, natural way?

At this point, the answers seems to be negative. This is a usual criticism against the Gini index: for instance, Piketty writes --- among other criticisms --- that the Gini coefficient gives “an abstract and sterile view of inequality” \cite[p.~408]{piketty2014}.

\subsubsection{Hoover coefficient as maximum vertical distance between the Lorenz curve and the diagonal}

Another classic interpretation of the Hoover coefficient involves the Lorenz curve.

\begin{proposition}\label{theorem_alternate_def_hoover_mean}
	Let $\mu \in \M$, $m$ its mean, $L$ its Lorenz function and $F$ its cumulative distribution function. Then:
	\[ H(\mu) = F(m) - L(F(m)). \]
\end{proposition}

\begin{proof}
	From \eqref{hoover_R/m}, we can write:
	\begin{align*}
		H(\mu) &= \frac{1}{m} \int_0^{m^-} (m - x) \dd\mu(x)\\
		& = \frac{1}{m} \int_0^m (m-x) \dd\mu(x)\\
		H(\mu) &= F(x) - \frac{1}{m} \int_0^m x \dd\mu(x).
	\end{align*}
	
	Now, we apply the pushforward theorem (lemma~\ref{pushforward_quantile} in appendix) with function
	\[\begin{array}{rcrcl}
		f&:&\R_+&\longrightarrow&\R_+\\
		&&x&\longmapsto&x \cdot \1_{x \le m}.
	\end{array}\]
	
	Hence,
	\begin{align*}
		\int_0^m x \dd \mu(x) &= \int_0^1 Q(p) \1_{Q(p) \le m} \dd p\\
		&= \int_0^1 Q(p) \1_{p \le F(m)} \dd p\\
		&= \int_0^{F(m)} Q(p) \dd p\\
		\int_0^m x \dd \mu(x) &= m\, L(F(m)).\qedhere
	\end{align*}
\end{proof}

\begin{theorem}\label{theorem_alternate_def_hoover_max}
	Let $\mu$ be a distribution in $\M$ and $L$ its Lorenz function. We have:
	\[H(\mu) = \displaystyle\max_{p \in [0,1]} \left(p - L_(p)\right).\]
\end{theorem}

\begin{proof}
	Let $\Phi: p \longmapsto p - L(p)$. Using proposition~ \ref{theorem_alternate_def_hoover_mean}, it is enough to prove that the concave function $\Phi$ reaches its maximum in $F(m)$.
	
	As per $L$, $\Phi$ has a left-derivative $\partial_- \Phi$ and a right-derivative $\partial_+ \Phi$ in every point of $(0,1)$. We have for $t \in (0,1)$: $\partial_-\Phi(t) = 1 - \frac{Q(t)}{m}$ and $\partial_+\Phi(t) = 1 - \frac{Q\left(t^+\right)}{m}$. $\Phi$ reaches its maximum in $p_\star$ if and only if $\partial_-\Phi\left(p_\star\right) \ge 0$ and $\partial_+\Phi\left(p_\star\right) \le 0$. This is equivalent to having both $Q\left(p_\star\right) \le m$ and $Q\left(p_\star^+\right) \ge m$. 
	
	On the one hand, the Galois inequalities imply that $Q(F(m)) \le m$. On the other hand, they ensure that if $x > F(m)$, then $Q(x) > m$. Taking the infimum, $Q(F(m)^+) \ge m$.
\end{proof}
\end{multicols}

\begin{multicols}{2}[\section{Some direct applications of these results}]
\label{section_altdefs_gini_hoover_application}
\subsection{$G$ and $H$ are nondecreasing with respect to Lorenz-domination}

\begin{proposition}
	Let $\mu, \nu \in \M$. Assume that for all $p \in [0,1]$, $L_\mu(p) \le L_\nu(p)$. Then:\nobreakpar
	\begin{enumerate}
		\item $G(\nu) \le G(\mu)$ and $H(\nu) \le H(\mu)$;
		\item Furthermore, $G(\mu) = G(\nu)$ if and only if $\mu \equiv \nu$ ($\mu = \nu$ up to some scale factor).
	\end{enumerate}
\end{proposition}
\begin{proof}
	The first part is an immediate consequence of theorems~\ref{theorem_alternate_def_gini} and~\ref{theorem_alternate_def_hoover_max}.
	
	Now assume $G(\mu) = G(\nu)$. Then, by theorem~\ref{theorem_alternate_def_gini}, we have $\int_0^1 L_\nu(t) - L_\mu(t) \dd t = 0$. $L_\nu - L_\mu$ is nonnegative, continuous, so $L_\mu = L_\nu$ on $[0,1]$. Hence, by corollary~\ref{corollary_rescale_lorenz}, $\mu \equiv \nu$.
\end{proof}

From part~2 of the proposition, we can say that $G$ is \emph{strictly increasing} with respect to Lorenz-majorization. However, $H$ is insensitive to redistribution either among the group of people with incomes higher (resp. lower) than the average, and thus not increasing. For instance, let
\[
	\mu \eqdef \frac 1 4 \left(2 \delta_0 + \delta_1 + \delta_3\right), \quad
	\nu \eqdef \frac 1 4 \left(2 \delta_0  + 2 \delta_2\right)
\]
then $L_\mu \le L_\nu$, $H(\mu) = H(\nu) = \frac 1 2$, but $\mu \not\equiv \nu$.

\subsection{Extreme values of Gini index under constraint on Hoover index}
\label{appendix_extreme_values_under_constraint}
Theorem~\ref{theorem_alternate_def_gini} states that the Gini index of a measure $\mu \in \M$ is two times the area of the surface
\[S_\mu \eqdef \{(x,y) : x \in [0,1], L_\mu(x) \le y \le x\}.\]

We have ${S_\mu = T \cap \Gamma_\mu}$, where $T$ is the (full) triangle
\[T \eqdef \{(x,y) : 0 \le y \le x \le 1\}\]
and $\Gamma_\mu$
is the epigraph of $L_\mu$, i.e.
\[\Gamma_\mu \eqdef \{ (x,y) : x \in [0,1], L_\mu(x) \le y \le 1\}.\]
As $L_\mu$ is a convex function, $\Gamma_\mu$ is a convex surface; so is $T$. Hence, the surface $S_\mu$ is convex.

Theorem~\ref{theorem_alternate_def_hoover_max} ensures that a measure $\mu$ has a Hoover index of $h$ if, and only if, its Lorenz curve touches the line of equation $y = x-h$ but never goes below (see fig.~\ref{fig_gini_lorenz:A}, p.~\pageref{fig_gini_lorenz:A}). We can describe the measures reaching the extreme values.

\begin{proposition}\label{extreme_value_under_constraint}
	Fix $h \in (0,1)$.
	\begin{enumerate}
		\item We have:
		\[\{G(\mu) : \mu \in \M, H(\mu) = h\} = [h, 2h - h^2).\]
	
		\item The measures $\mu$ for which \[{G(\mu) = H(\mu) = h}\] are exactly the bimodal measures of form:
		\[ \mu = \alpha \cdot \delta_{m\cdot \left(1 - \frac h \alpha\right)} + (1-\alpha) \cdot \delta_{m \cdot \left(1 + \frac h {1 - \alpha}\right)}\] 
		where $m > 0$ and $\alpha \in [h, 1)$.
	\end{enumerate}
\end{proposition}

$m$ is the mean of the distribution, and $\alpha$ a form parameter that is the relative share of the group of poors. In this configuration, the total group of poors owns a share $\alpha - h$ of the total resource. At the same time, the group of rich (share $1-\alpha$ of the population) owns a share $1-\alpha + h$ of the resource.

\begin{proof} Consider an orthonormal frame of origin $O$. Let $I$ the point of coordinates $(1;1)$. For any $h$, the line of equation $y = x - h$ is the one containing the points $A(h;0)$ and $B(1;1-h)$. The surface $S_\mu$ is delimited by the segment $[OI]$ and the curve of $L_\mu$.
	\paragraph{Proof of the lower bound \& the lower bound is reached (fig.~\ref{fig_gini_lorenz:B}).} Minimizing $G(\mu)$ under the constraint $H(\mu) = h$ is equivalent to minimizing the surface of the epigraph of $L_\mu$ under the constraint that $L_\mu$ touches the line $[AB]$.
	
	Assume $L_\mu$ fillfills the constraint. There exists $\alpha \in [h, 1)$ such that $L$ passes through the point $C_\alpha$ with coordinates $(\alpha; \alpha -h)$. Hence, the surface $S_\mu$ must be convex and contain the points $O$, $I$ and $C_\alpha$. Yet, the surface with minimal area containing these three points is their convex envelope, i.e. the triangle $OC_\alpha I$. The surface of this triangle is $\frac{h}{2}$ whatever the chosen value of $\alpha$.
	
	Hence, $G(\mu)$ is minimized by any $\mu$ such that $L_\mu$ is piecewise affine, which graph is made of the segments $[OC_\alpha]$ and $[C_\alpha I]$. This corresponds to the measures of form:
	\[\mu_{\alpha,m} = \alpha\cdot \delta_{m\cdot\left(1 - \frac h \alpha\right)} + (1-\alpha)\cdot \delta_{m\cdot \left(1 + \frac h {1 - \alpha}\right)}\]
	for some $m > 0$. The proposition~\ref{bijection_M_Lorspace_R} ensures that there is no other measures with such a graph.
	
	Eventually, the set \[\{\mu_{\alpha, m} : \alpha \in [h, 1), m > 0\}\] is the set of functions minimizing $G(\mu)$ under the constraint $H(\mu) = h$; for these functions, we have $G(\mu_{\alpha, m}) = h$.
	
	\paragraph{Proof of the upper bound \& proof that the upper bound is never reached (fig.~\ref{fig_gini_lorenz:C}).} If $H(\mu) = h$, then $L_\mu$ must be contained within the trapezoid $OABI$. Thus, the area of $S_\mu$ must be at most the area of $OABI$, i.e. $h - \frac{h^2}{2}$. Hence, $G(\mu) \le 2h - h^2$.
	
	We now prove that this value cannot be reached. As $L_\mu$ is continuous at $1$, there exist $\eta > 0$ such that for all $x \ge 1-\eta$, $L_\mu(x) > 1 - \frac{h}{2}$. Let $R$ be the rectangle:
	\[ R = [1-\eta,1] \times \left[1- h , 1-\frac h 2\right]. \]
	
	We have that:\nobreakpar
	\begin{itemize}
		\item The epigraph $\Gamma_\mu$ cannot intersect $R$, so $S_\mu \cap R = \varnothing$.
		\item But $R$ is contained in the (full) trapezoid $OABI$. 
	\end{itemize}
	
	This implies that the surface of $S_\mu$ is at most the surface of $OABI$, minus the surface of $R$, i.e.:
	\[ G(\mu) \le 2h - h^2 - \eta h. \]
	
	That is, the upper bound is never reached.
	
	\paragraph{Proof that every value between the upper bound and the lower bound is reached (fig.~\ref{fig_gini_lorenz:D}).} In order to achieve any value in the interval $(h, 2h - h^2)$, we just propose a three-group distribution such that the Lorenz curve is composed of three segments $[OA]$, $[AC_\alpha]$ and $[C_\alpha I]$ where $C_\alpha$ has coordinates $(\alpha; \alpha-h)$, with $h \le \alpha < 1$ is to be determined (note that for $h=\alpha$, $C_h$ and $A$ are combined). Such a $\mu$ exists by proposition~\ref{bijection_M_Lorspace_R}.
	
	Some elementary geometry ensures that the area of the quadrilateral $OAC_\alpha I$ is:
	\[ \mathscr A_\alpha = \frac{1}{2} (h + \alpha h - h^2).\]
	
	The Gini index of the distribution is $2\mathscr A_\alpha$. Thus, $G(\mu)$ reaches every value of $h + \alpha h - h^2$ for $h \le \alpha < 1$, i.e. $G(\mu)$ reaches every value of the interval $[h, 2h-h^2]$.
\end{proof}
\end{multicols}

\begin{figure}[!p]
	{  \centering
		\begin{tabular}{cc}
			\subcaptionbox{To match the condition $H(\mu) = h$, the curve $L_\mu$ must touch the line $[AB]$ without crossing it. The dotted surface $S_\mu$, delimited by the line $[OI]$ and the Lorenz curve $L_\mu$, has area $\frac12 G(\mu)$. \label{fig_gini_lorenz:A}}[0.45\textwidth]{\resizebox{!}{.39\textwidth}{\begin{tikzpicture}
	\input{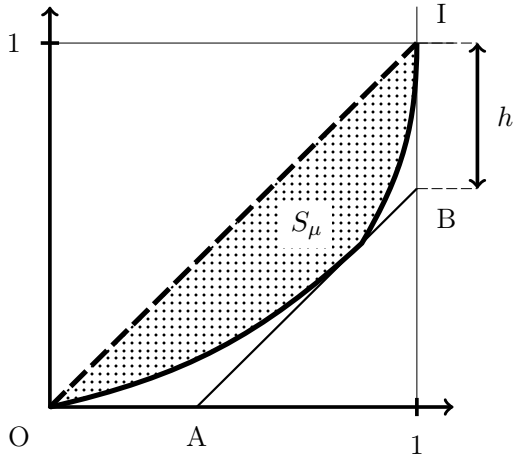}
	
	\node (C_alpha) at (4.25, 2.25) {};
	\draw [style={lorenz_curve}] (origin.center)
		to [bend right=15] (C_alpha.center)
		to [bend right=15] (I.center);
	\draw [style={S_mu}] (origin.center)
		to [bend right=15] (C_alpha.center)
		to [bend right=15] (I.center)
		to cycle;
	
	\node [fill=white] () at (3.5, 2.5) {$S_\mu$};
	
	\draw [style=doublearrow] (5.825, 5) to (5.825, 3);
	\node [label={right:$h$}] at (5.825, 4) {};
	\draw [style=doublearrow_anchor] (B.center) to (5.825, 3);
	\draw [style=doublearrow_anchor] (I.center) to (5.825, 5);
\end{tikzpicture}}} &
			
			\subcaptionbox{To minimize the Gini index, the Lorenz curve must minimize the surface $S_\mu$ while still touching the line $[AB]$. This condition is checked if and only if $L_\mu$ touches the line in a single point $C_\alpha$ and if surface $S_\mu$ is the convex envelope of the points $O$, $C_\alpha$ and $I$, i.e. the dotted triangle. \label{fig_gini_lorenz:B}}[0.45\textwidth]{\resizebox{!}{.39\textwidth}{\begin{tikzpicture}
	\input{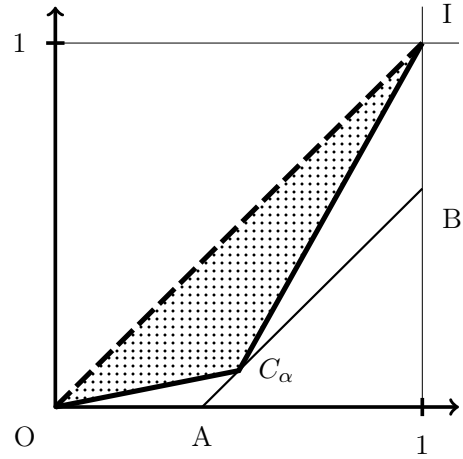}
	
	\node [label={right: $C_\alpha$}] (C_alpha) at (2.5, 0.5) {};

	\draw [style={lorenz_curve}] (origin.center)
	 to (C_alpha.center)
	 to (I.center);
	\draw [style={S_mu}] (origin.center)
	 to (C_alpha.center)
	 to (I.center)
	 to cycle;
\end{tikzpicture}}} \\
			
			\vspace{2em}&~\\
			
			\subcaptionbox{The dotted surface $S_\mu$ needs be contained within the trapezoid $OABI$, which area is $h - \frac{h^2}{2}$. However, it need not include the hatched rectangle of dimensions $\eta \times \frac h 2$. \label{fig_gini_lorenz:C}}[0.45\textwidth]{\resizebox{!}{.39\textwidth}{\begin{tikzpicture}
	\input{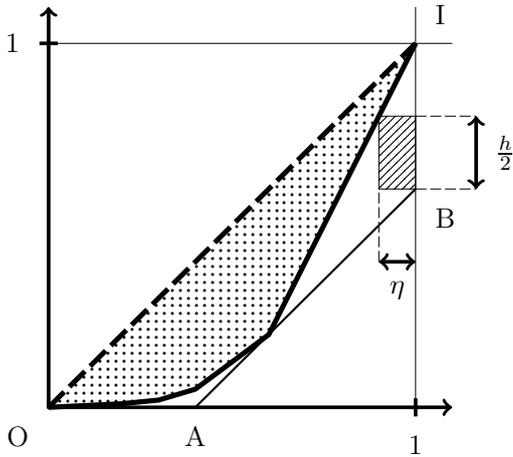}
	
	\node (LC1) at (1, 0.05) {};
	\node (LC2) at (1.5, 0.10) {};
	\node (LC3) at (2, 0.25) {};
	\node (LC4) at (3, 1) {};
	
	\draw [style=forbidden] (4.5,4)
		to (5,4)
		to (5,3)
		to (4.5,3)
		to cycle;
	
	\draw [style={lorenz_curve}] (origin.center)
		to (LC1.center)
		to (LC2.center)
		to (LC3.center)
		to (LC4.center)
		to (I.center);
	\draw [style={S_mu}] (origin.center)
		to (LC1.center)
		to (LC2.center)
		to (LC3.center)
		to (LC4.center)
		to (I.center)
		to cycle;

	\draw [style=doublearrow] (4.5,2) to (5,2);
	\node [label={below:$\eta$}] at (4.75,2) {};
	\draw [style=doublearrow_anchor] (4.5,2) to (4.5,3);
	\draw [style=doublearrow] (5.825,3) to (5.825,4);
	\node [label={right:$\frac h 2$}] at (5.825,3.5) {};
	\draw [style=doublearrow_anchor] (5,3) to (5.825,3);
	\draw [style=doublearrow_anchor] (5,4) to (5.825,4);
\end{tikzpicture}}} &
			
			\subcaptionbox{If $\alpha$ moves between $A$ (included) and $B$ (excluded), then the dotted surface takes every possible area between $\frac h 2$ (included) and $h - \frac{h^2}{2}$ (excluded). \label{fig_gini_lorenz:D}}[0.45\textwidth]{\resizebox{!}{.39\textwidth}{\begin{tikzpicture}
	\input{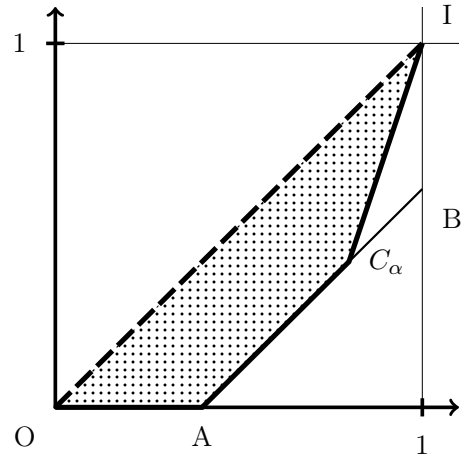}
	
	\node [label={right: $C_\alpha$}] (C_alpha) at (4, 2) {};

	\draw [style={lorenz_curve}] (origin.center)
	 to (A.center)
	 to (C_alpha.center)
	 to (I.center);
	\draw [style={S_mu}] (origin.center)
	 to (A.center)
	 to (C_alpha.center)
	 to (I.center)
	 to cycle;
\end{tikzpicture}}}\\
			\vspace{2em}&~\\
		\end{tabular}
		\caption{Illustration of the geometric arguments used in the proof of corollary~\ref{extreme_value_under_constraint}. }
	}
	
	~
	
	In all figures, we fix some $h \in (0,1)$. The points $A$ and $B$ have coordinates $(h;0)$ and $(1;1-h)$. The line $(AB)$ has equation $y=x-h$. The point $C_\alpha$ is the point of this line having absciss $\alpha \in [h, 1)$.
\end{figure}


\part{Convergence properties of Lorenz curve, Gini index and Hoover index}
\nobreakpar
\begin{multicols}{2}[\section{Convergence of $L$, $G$ and $H$ when the measures converge in $\Wone$}\label{section_convergence_results}]
	It has been known for an long time that if $\mu$ is a random variable and $X_1, \dots, X_n$ is a random sample of $\mu$ (i.i.d. variables of distribution $\mu$), then the \emph{empirical Gini index} $G((X_i)_{i})$ converges to the \emph{theoretical Gini index} $G(\mu)$. Under some good hypotheses of integrability, we do even get the asymptotic normality of $\sqrt{n}(G((X_i)_{1\le i\le n}) - G(\mu))$ \cite[section 9.b]{hoeffding1948}.
	
	We do not intend to explore the asymptotic distributions of empirical indexes. However, we do not restrict on convergence of random samples: we will determine necessary and sufficient conditions for having uniform convergence of the Lorenz curves as soon as the underlying distributions converge.
	
	\subsection{$\Wone$ distance, weak convergence and uniform integrability}
	\label{subsection_notations_convergence}
	Recall that a collection $(U_i)_{i \in I}$ of real random variables defined on a space $(\Omega, \mathscr F, \prob)$ is said \emph{uniformly integrable} (\emph{u.i.}) if
	\[ \sup_{i \in I} \Exp\bigl[|U_i|~\1_{|U_i| > \alpha}\bigr] \xrightarrow[\alpha \to +\infty]{} 0. \]
	It is known (see \cite[theorem~16.14, corollary]{billingsley1995}) that if $U$ is a random variable, ${\Exp[|U_n-U|] \xrightarrow[n \to \infty]{} 0}$ if and only if $U_n \xrightarrow[n \to \infty]{} U$ in distribution and $(U_n)_{n \in \infty}$ is u.i.

	Define the \emph{Wasserstein-1 metric} $\Wone$ on $\mathscr M_1(\R_+)$ as:
	\[ \Wone(\mu, \nu) \eqdef \int_0^1 \left|Q_\mu(t) - Q_\nu(t)\right| \dd t. \]
	Convergence with respect to metric $\Wone$ is equivalent to $\Lone$ convergence of quantile functions.
	
	We say that $(\mu_i)_{i \in \N}$ is \emph{uniformly integrable} if the collection $(Q_{\mu_i})_{i \in \N}$ is uniformly integrable in the measure space $\bigl([0,1), \Borel, \Leb\bigr)$. Keep in mind that this notion of uniform integrability of \ul{measures} is nonstandard; most authors do only define uniform integrability for measurable functions or for random variables.
	
	Finally, we say $(\mu_n)_{n \in \N}$ \emph{converges weakly} to $\mu_\infty$ if $F_{\mu_n}(x) \xrightarrow[n \to \infty]{} F_{\mu_\infty}(x)$ for all $x$ at which $F_{\mu_\infty}$ is continuous. We note $\mu_n \xrightarrow[n \to \infty]{\weak} \mu_\infty$. This is equivalent to having $Q_{\mu_n} \xrightarrow[n \to \infty]{} Q_{\mu_\infty}$, $\Leb$-almost surely.
	
	The pertinence of these concepts and the elementary properties are detailed in appendices~\ref{appendix_weak_cv}, \ref{appendix_W1_cv} and~\ref{appendix_unif_int}.
	
	\begin{theorem}[Scheffé--Lebesgue]\label{scheffe_lebesgue}~\\
		Let $(\mu_n)_{n \in \N} \in \M^\N$ and $\mu_\infty \in \M$. The following are equivalent:
		\begin{enumhypos}
			\item $\mu_n \xrightarrow[n \to \infty]{\Wone} \mu_\infty$;
			\item $\mu_n \xrightarrow[n \to \infty]{\weak} \mu_\infty$ \ul{and} $m_{\mu_n} \xrightarrow[n \to \infty]{} m_\infty$;
			\item $\mu_n \xrightarrow[n \to \infty]{\weak} \mu_\infty$ \ul{and} $(\mu_n)_{n \in \N}$ is uniformly integrable.
		\end{enumhypos}
	\end{theorem}
	For a detailed proof, see propositions~\ref{scheffe_measures} and~\ref{weak_vitali} in appendix.

	\subsection{Necessary and sufficient conditions for convergence of Lorenz curves}

	\begin{proposition}\label{characterization_lorenz_convergence}
		Let $(\mu_n)_{n \in \N}$ and $\mu_\infty$ measures in $\mathbf{M}$. We denote $L_\bullet$ their Lorenz functions and and $m_\bullet$ their means. The following assertions are equivalent:
		\begin{enumhypos}
			\item $\mu_n \xrightarrow[n \to \infty]{\Wone} \mu_\infty$; \label{characterization_lorenz_convergence:CV_W^1}
			\item $L_n \xrightarrow[n \to \infty]{} L_\infty$ pointwise over $[0,1]$ \ul{and} $m_{\mu_n} \xrightarrow[n \to \infty]{} m_{\mu_\infty}$;\label{characterization_lorenz_convergence:CV_Lor_pw+means}
			\item $L_n \xrightarrow[n \to \infty]{} L_\infty$ uniformly over $[0,1]$ \ul{and} $m_{\mu_n} \xrightarrow[n \to \infty]{} m_{\mu_\infty}$.\label{characterization_lorenz_convergence:CV_Lor_unif+means}
		\end{enumhypos}
	\end{proposition}
	
	We split the proof in two part.

	\begin{proof}[Proof of the easy implications]
		$\ref{characterization_lorenz_convergence:CV_W^1} \implies \ref{characterization_lorenz_convergence:CV_Lor_pw+means}$ is trivial using definition \ref{def_lorenz}, triangle inequality and the definition of the $\Wone$ metric.

		Implication $\ref{characterization_lorenz_convergence:CV_Lor_pw+means} \implies \ref{characterization_lorenz_convergence:CV_Lor_unif+means}$ is a direct application of Dini’s lemma, that we hereby recall.
	\end{proof}
	
	\begin{lemma}[Dini]\label{dini}
		Let $[a,b]$ be a segment.  Let for $n \in \N$, $f_n$ and $f_\infty: K \longrightarrow \R$. Assume:
		\begin{enumhypos}
			\item $f_n \xrightarrow[n \to \infty]{} f_\infty$ pointwise;
			\item $f_n$’s are nondecreasing functions;
			\item $f_\infty$ is continuous on $[a,b]$.
		\end{enumhypos}
		
		Then $f_n \xrightarrow[n \to \infty]{\|.\|_\infty} f_\infty$ (uniform convergence on $[a,b]$).
	\end{lemma}
	(See \cite[problem~127]{polya2004} for a proof.)
			
	Now the implication $\ref{characterization_lorenz_convergence:CV_Lor_unif+means} \implies \ref{characterization_lorenz_convergence:CV_W^1}$ remains to be proven. To do so, we introduce the following lemma.
	
	\begin{lemma} \label{lemma_unif_convergence_of_convex_implies_convergence_of_derivatives}
		Let $I$ an open interval of $\R$. Let $F_n : I \longrightarrow \R$ convex functions that uniformly converge to $F_\infty$ over $I$. Let $\delta_- \bullet$ the left-derivative of a function.
		
		$\partial_- F_n(x) \xrightarrow[n \to \infty]{} \partial_- F_\infty(x)$ pointwise for every $x$ at which $\partial_- F_\infty$ is continuous.
	\end{lemma}
	
	\begin{proof}
		First, for all $\eps > 0$, there exists $N_\eps \in \N$ such that for every $n \ge N_\eps$, ${\|F_n - F_\infty\|_\infty < \eps^2}$.
		
		Let, for every $n \in \N \cup \{\infty\}$ and ${a, b, p \in I}$ and $p \in I$,
		\[ S_n(a,b) \eqdef \frac{F_n(b) - F_n(a)}{b-a}. \]
		
		Notice that $F_\infty$, as a uniform limit of convex functions, is itself convex.
		
		To ease the notations, let $f_n \eqdef \partial_- F_n$ for every $n \in \N \cup \{\infty\}$.
		
		\paragraph{Minoring $\liminf_n f_n(p)$.} First, fix $\eps > 0$. for every $n \ge N_\eps$, we have:
		\begin{align*}
			&\left|S_n(p-\eps, p) - S_\infty(p-\eps, p)\right|\\
			=~& \frac{\left|F_n(p-\eps) - F_n(p) - F_\infty(p-\eps) + F_\infty(p)\right|}{\eps}\\
			\le~ & \frac{\left|F_n(p-\eps) - F_\infty(p-\eps)\right|}{\eps} + \frac{\left|F_n(p) - F_\infty(p)\right|}{\eps}\\
			\le~& \frac{\eps^2}{\eps} + \frac{\eps^2}{\eps} = 2\eps.
		\end{align*}
		
		Hence, by convexity of $F_n$,
		\[ S_\infty(p-\eps,p) - 2\eps \le S_n(p-\eps, p) \le f_n(p). \]
		
		Furthermore, we have by definition of $f_\infty$:
		\[ S_\infty(p-h, p) - 2 h \xrightarrow[\subalign{h &\to 0\\ h& > 0}]{} f_\infty(p) - 0. \]
		
		Hence, for every $\delta > 0$, there exists $\eps_\delta > 0$ such that
		\[ S_\infty\left(p - \eps_\delta\right) - 2\eps_\delta > f_\infty(p) - \delta. \]
		
		So for every $\delta > 0$, there exists $N \eqdef N_{\eps_\delta}$ such that $\forall n \ge N_{\eps_\delta}$,
		\[ f_n(p) > f_\infty(p) - \delta.\]
		
		\paragraph{Majoring $\limsup_n f_n(p)$.}
		Similarly, for all $\eps > 0$, for every $n \ge N_\eps$,
		\[ |S_n(p, p+\eps) - S_\infty(p, p+\eps)| \le 2\eps. \]
		By convexity of $F_n$,
		\[f_n(p) = \partial_-F_n(p) \le S_n(p, p+\eps). \]
		Hence,
		\[ f_n(p) \le S_\infty(p, p+\eps) + 2\eps. \]
		
		By convexity of $F_\infty$, for all $\eta > 0$:
		\[ \partial_-F_\infty(p) \le \partial_+F_\infty(p) \le \partial_-F_\infty(p+\eta). \]
		
		Thus, since $f_\infty$ is continuous at $p$, \[\partial_+F_\infty(p) = \partial_-F_\infty(p^+) = \partial_-F_\infty(p) =  f_\infty(p).\]
		
		Hence,
		\[ S_\infty(p, p+h) + 2h \xrightarrow[\subalign{h &\to 0\\ h&> 0}]{} f_\infty(p). \]
		
		Hence, for all $\delta > 0$, there exists $\eps'_\delta$ such that $S_\infty(p, p +\eps'_\delta) + 2\eps'_\delta < f_\infty(p) + \delta$. Finally, for every $n \ge N_{\eps'_\delta}$,
		\[ f_n(p) < f_\infty + \delta.\qedhere \]
	\end{proof}
	
	\begin{proof}[Proof of the remaining part of theorem~\ref{characterization_lorenz_convergence}.]~
		
		Assume hypothesis~\ref{characterization_lorenz_convergence:CV_Lor_unif+means}.
		The functions $L_n$ are convex. As stated in section \ref{section_introduction_lorenz_and_first_properties}, for all $p \in (0,1]$ and for every $n \in \N \cup \{\infty\}$, ${\partial_- L_{\mu_n}(p) = Q_{\mu_n}(p) / m_{\mu_n}}$
		
		As $L_{\mu_n} \xrightarrow[n \to \infty]{} L_{\mu_\infty}$ uniformly, lemma~\ref{lemma_unif_convergence_of_convex_implies_convergence_of_derivatives} ensures that the sequence of functions $(Q_{\mu_n} / m_{\mu_n})_{n \in \N}$ converges to $Q_{\mu_\infty} / m_{\mu_\infty}$ pointwise over $(0,1)$, except at the points of discontinuity of $Q_{\mu_\infty}$. The convergence also stands at $0$.
		
		As we assumed $m_{\mu_n} \xrightarrow[n \to \infty]{} m_{\mu_\infty}$, we get $Q_{\mu_n} \xrightarrow[n \to \infty]{} Q_{\mu_\infty}$ pointwise over $(0,1)$, except in the points of discontinuity of $Q_\infty$. Hence, $\mu_n \xrightarrow[n \to \infty]{\weak} \mu_\infty$. By Scheffé--Lebesgue, \ref{characterization_lorenz_convergence:CV_W^1} stands.
	\end{proof}
	
	\subsection{Sufficient condition for Gini and Lorenz convergence}
	\begin{proposition}\label{convergence_indicators}
		Let $(\mu_n)_{n\in\N}$ and $\mu_\infty$ measures of $\M$. If $\mu_n \xrightarrow[n \to \infty]{\Wone} \mu$, then:\nobreakpar
		\begin{enumerate}
			\item $L_{\mu_n} \xrightarrow[n \to \infty]{} L_{\mu_\infty}$ uniformly over $[0,1]$.
			\item $G(\mu_n) \xrightarrow[n \to \infty]{} G(\mu_\infty)$.
			\item $H(\mu_n) \xrightarrow[n \to \infty]{} H(\mu_\infty)$.
		\end{enumerate}
	\end{proposition}
	
	\begin{proof}
		Theorem \ref{characterization_lorenz_convergence} ensures that $(L_{\mu_n})_{n\in \N}$ converges to $L_{\mu_\infty}$ uniformly on $[0,1]$. So does $(\id - L_{\mu_n})_{n\in \N}$ to $\id - L_{\mu_\infty}$, where ${\id: x \in [0,1] \longmapsto x}$.
		
		Thus,
		\begin{align*}
			\int_0^1 p - L_{\mu_n}(p) \dd p & \xrightarrow[n \to \infty]{} \int_0^1 p - L_{\mu_\infty}(p) \dd p~;\\
			\max_{p \in [0,1]} ~ p - L_{\mu_n}(p) & \xrightarrow[n\to \infty]{} ~ \max_{p\in[0,1]} p - L_{\mu_\infty}(p).
		\end{align*}
		
		We conclude the proof by applying theorem~\ref{theorem_alternate_def_gini} and theorem~\ref{theorem_alternate_def_hoover_max}.
	\end{proof}

	\subsection{Topological formulation: homeomorphism between $(\M, \Wone)$ and $\Lorspace \times \R_+^*$}
	\label{subsection_topo}
	\subsubsection{Basic result}

	Recall that a topological space is called \emph{sequential} if its sequentially closed sets are closed. A mapping from a sequential space to any topological space is continuous iff it is sequentially continuous. Being first countable is a sufficient condition for being sequential; thus, metric spaces are sequential.
	
	Now, we embed:\nobreakpar
	\begin{itemize}
		\item $\R_+^*$ with its standard norm $|\cdot|$;
		\item $\M$ with the metric $\Wone$;
		\item $\Lorspace$ with the uniform norm $\|.\|_\infty$.
	\end{itemize}
	
	$\R_+^*$, $\M$ and $\Lorspace$, embedded with the underlying topologies, are sequential spaces. Hence, propositions~\ref{bijection_M_Lorspace_R} and~\ref{characterization_lorenz_convergence} can be reworded as follow:
	\begin{theorem}\label{characterization_lorenz_convergence_topological}
		The following mapping is a homeomorphism:
		\[ \begin{array}{crcl}
			\Phi:& (\M, \Wone) &\approx& (\Lorspace, \|\cdot\|_\infty) \times (\R_+^*, |\cdot|)\\
			&\mu & \longmapsto & (L_\mu, m_\mu).
		\end{array}	\]
	\end{theorem}
	
	\subsubsection{$\Wone$ topology on $\M / \R_+^*$}
	Now, let’s focus on the quotient space $\M / \R_+^*$ of distributions on $(\R_+, \Borel)$ modulo equality up to a rescaling. Let $\pi : \M \longrightarrow \M / \R_+^*$ the canonical projection. $\M / \R_+^*$ can be endowed with the quotient $\Wone$ topology: $U$ is open in $(\M / \R_+^*, \Wone)$ iff $\pi^{-1}\langle{}U\rangle{}$ is open in $(\M, \Wone)$.
	
	\paragraph{A description of the $\Wone$ topoology on $\M / \R_+^*$.} For $\alpha > 0$ and $\mu \in \M$, call $S_\alpha(\mu)$ the distribution of $\alpha \cdot X$ where $X \sim \mu$. For $\eps > 0$ and $\mu \in \M$, call $B(\mu, \eps)$ the $\Wone$-open ball of center $\mu$ and radius $\eps$.
	
	Let $U$ a $\Wone$-open subset of $\M$. Call ${U' = \pi^{-1}\langle\pi\langle U\rangle\rangle}$ its saturate for $\equiv$. If $\mu \in U'$, there exists $\mu_0 \in U$ and $\alpha > 0$ such that $\mu = S_\alpha(\mu_0)$. Since $U$ is open, there exists $\eps > 0$ such that $B(\mu_0, \eps) \subseteq U$. Now consider the open ball $V' \eqdef B(\mu, \alpha\eps)$. It suffices to prove that $V' \subseteq U'$. For $\nu \in V'$, let $\nu_0 = S_{1/\alpha}(\nu)$. We have (see lemma~\ref{rescale_quantile} in appendix for details):
	\begin{align*}
		\Wone(\nu_0, \mu_0) &= \int_0^1 |Q_{\nu_0}(p) - Q_{\mu_0}(p)| \dd p\\
		&= \int_0^1 \left|\frac1\alpha Q_{\nu}(p) - \frac1\alpha Q_{\mu}(p)\right| \dd p\\
		&= \frac{1}{\alpha} \Wone(\mu, \nu)\\
		\Wone(\nu_0, \mu_0) &< \eps
	\end{align*}
	thus $\nu_0 \in B(\mu_0, \eps) \subseteq U$, so $\nu \in U'$. Hence, $V' \subseteq U'$. Hence, $U'$ is open.
	
	This proves that $\pi : \M \longrightarrow \M/\R_+^*$ is an open mapping. Hence, the $\Wone$ topology on $\M/\R_+^*$ is generated by the elementary opens
	$ \pi\langle B(\mu_0, \eps)\rangle $
	for $\mu_0 \in \M$ and $\eps > 0$.
	
	Notice that the quotient $\Wone$ \emph{metric} need not induce the quotient $\Wone$ \emph{topology}. In fact, it does not; one can check that the quotient $\Wone$ metric is a trivial pseudometric $d([\mu], [\nu])=0$ for all $\mu$ and $\nu$.
	
	\paragraph{Quotient mapping $[\mu] \longmapsto L_\mu$.} The mapping ${\Psi: \mu \in \M \longmapsto L_\mu}$ is continuous and goes to the quotient. Hence, it induces a continuous mapping: ${\tilde \Psi: [\mu] \in \M / \R_+^* \longmapsto L_\mu}$. One can directly check that its reciprocate function is given by:
	\[ \tilde\Psi^{-1}(\ell) = \left[ \Phi^{-1}(\mu, 1) \right] \]
	which is also continuous.
		
	Finally, we have:
	\begin{proposition}\label{homeos_M_Lorspace} The following mappings are homeomorphisms:
	\[  \begin{array}{rrcl}
			\tilde \Psi : & \M / \R_+^* & \approx & \Lorspace\\
			&[\mu] & \longmapsto & L_\mu\\
	\\
			&\M & \approx & \M/\R_+^* \times \R_+^*\\
			&\mu & \longmapsto & ([\mu], m_\mu)\\
	\\
			\forall \alpha \in \R_+^*, & \M_\alpha & \approx & \M / \R_+^*\\
			&\mu & \longmapsto & [\mu]\\
	\\
			\forall \mu \in \M, &[\mu] & \approx & \R_+^*\\
			&\nu & \longmapsto & m_\nu
		\end{array}\]
	where $\M$, $\M_\alpha$, $[\mu]$ and $\M / \R_+^*$ are embedded with the $\Wone$ topology or the topology induced by it; $\Lorspace$ is embedded with the $\|\cdot\|_\infty$ norm and $\R_+^*$ is embedded with its standard topology.
	\end{proposition}  
	
	\subsubsection{Embedding $\Lorspace$ with the pointwise convergence topology}
	What can we say if $\Lorspace$ is not embedded with the topology of \ul{uniform} convergence, but \ul{pointwise} convergence, i.e. the topology induced by the product topology on $\R^{[0,1]}$? As the pointwise convergence topology is not sequential, \emph{a priori} proposition~\ref{characterization_lorenz_convergence} tells nothing.
	
	However, we can adapt Dini’s theorem to prove that for given $a, b \in \R$, the $\|\cdot\|_\infty$ topology and the product topology on $\R^{[a,b]}$ induce the same topology on the subset of continuous nondecreasing functons.
	\begin{proposition}[generalization of Dini]
		Let $[a,b] \subseteq \R$, $E \eqdef \R^{[a,b]}$ and $\mathscr C^\uparrow \subseteq E$ the set of continuous, nondecreasing functions.
		
		Let $\tau_\infty$ the topology on $E$ induced by the norm $\|\cdot\|_\infty$ and $\tau_\times$ the product topology on $E$. Let $\tau'_\infty$ (resp. $\tau'_\times$) the trace topoolgy induced by $\tau_\infty$ (resp. $\tau_\times$) on $\mathscr C^\uparrow$.
		
		We have $\tau'_\infty = \tau'_\times$.
	\end{proposition}
	
	\begin{proof} For $f \in E$ and $\eps > 0$, let $B(f, \eps)$ the $\|\cdot\|_\infty$-ball of center $f$ and radius $\eps$.
		\paragraph{$\tau_\infty$ is thinner than $\tau_\times$.}
		Let $n \in \N^*$ and for $i = 1, \dots, n$: ${x_i \in [a,b]}$, $y_i \in \R$ and $\eps_i > 0$. Let
		\[U \eqdef \{ g \in E : \forall i \in \lb 1, n \rb, |g(x_i) - y_i| < \eps_i \}\]
		an elementary open set of $\tau_\times$. Let $g_0 \in U$. for every $i = 1, \dots, n$, let $a_i = g_0(x_i) - (y_i - \eps_i)$ and $b_i = (y_i + \eps_i) -  g_0(x_i)$. Let:
		\[ m = \min\{a_i, b_i : i = 1, \dots, n\}.\]
		Then $g_0 \in B(g_0, m) \subseteq U$, so $U$ is $\tau_\infty$-open.
		
		It follows that  $\tau_\infty$ is thinner than $\tau_\times$. Hence, $\tau'_\infty$ is also thinner than $\tau'_\times$.
		
		\paragraph{$\tau'_\times$ is thiner than $\tau'_\infty$.}
		
		First, notice that the topology $\tau'_\infty$ admits as a basis the trace balls of form $B(f, \eps) \cap \mathscr C^\uparrow$ with $\eps > 0$ \ul{and $f \in \mathscr C^\uparrow$}.
		
		Now consider a trace ball $B(f, \eps) \cap \mathscr C^\uparrow$ with $f \in \mathscr C^\uparrow$. It suffices to prove that there exists an elementary open $U \in \tau_\times$ such that:
		\[ f \in U \quad \text{and} \quad U \cap \mathscr C^\uparrow \subseteq B(f, \eps). \]
		
		To do so, we use the same argument as for Dini’s theorem. By Heine--Cantor theorem, $f$ is uniformly continuous. Let $\eta > 0$ such that for all $x, y \in [a,b]$, ${|x - y| < \eta \implies |f(x) - f(y)| < \frac{\eps}{5}}$.
		
		Let $a_0 \eqdef a < a_1 < \dots < a_n \eqdef b$ with $a_{i} - a_{i-1} < \eps$ for every $i = 1, \dots, n$. Let $U$ the elementary $\tau_\times$-open:
		\[ \left\{ g \in E : \forall i \in \lb 0, n\rb, |g(a_i) - f(a_i)| < \frac \eps 5 \right\}.  \]
		
		It is immediate that $f \in U$. Now assume $g \in U \cap \mathscr C^\uparrow$. Let $x \in [a,b]$. There exists $i \in \lb0,n-1\rb$ such that $a_i \le x \le a_{i+1}$. We have:
		\begin{align*}
			& |g(x) - f(x)|\\
			\le~& |g(x) - g(a_i)| + |g(a_i) - f(a_i)|
				\\*	&\quad + |f(a_i) - f(x)| \\
			<~& |g(a_{i+1}) - g(a_i)| + \frac{2\eps}{5} \\
			\le~& |g(a_{i+1}) - f(a_{i+1})| + |f(a_{i+1}) - f(a_i)|
				\\*	& \quad + |f(a_i) - g(a_i)| + \frac{2\eps}{5}\\
			<~& \eps.\qedhere
		\end{align*}
	\end{proof}
	\paragraph{Consequence.} Since $\Lorspace \subseteq \mathscr C^\uparrow$, the theorem~\ref{characterization_lorenz_convergence_topological} remains valid if $\Lorspace$ is embedded with pointwise convergence topology.

	\subsubsection{$G$ and $H$ as continuous applications} To prove proposition~\ref{convergence_indicators}, we merely used the fact that the mappings ${L_\mu \longmapsto G(\mu)}$ and ${L_\mu \longrightarrow H(\mu)}$ are continuous. Thus, the mappings $G : \M \longrightarrow [0,1) $ and $ H : \M \longrightarrow [0,1)$ are continuous, where $\M$ is embedded with the $\Wone$ topology. Since $G$ and $H$ go to the quotient, the quotient mappings $\tilde G$ and $\tilde H : \M / \R_+^* \longrightarrow [0,1)$ are continuous too.
\end{multicols}

\begin{multicols}{2}[\section{Applications of the $\Wone$ convergence}]
	\label{section_examples_convergence}
The main practical interest of the proposition \ref{convergence_indicators} is that it justifies that the Lorenz curves, the Gini and the Hoover are consistent with small perturpations an approximations. In the following section, we will show that the $\Wone$ convergence occurs in several cases, allowing us to apply proposition \ref{convergence_indicators}.

Let $\mu \in \M$. With increasing complexity, we deal with the following situations:
\begin{enumerate}
	\item $\mu$ is perturbated by a noise, which vanishes.
	\item $\mu$ is approximated by a random sample, and the size $n$ of the sample increases.
	\item $\mu$ is approximated by a discrete distribution corresponding to $\ell$ regularly-chosen quantiles, and $\ell$ grows.
	\item $\mu$ is approximated by $\ell$ regularily chosen quantiles of a sample of size $n$, as $\ell$ and $n$ grow. This corresponds to a more realistic situation: in practice, the public offices for statistics do only publish the quantiles of the population’s revenue, based on a sample.
	\item $\mu$ is approximated with the kernel density estimate of a sample of size $n$ with window $\eps$, as $n$ grows and $\eps$ shrinks.
\end{enumerate}

\paragraph{Notation.} In what follows, the notation $\otimes$ refers to the tensor product of measures. Recall that if $X \sim \mu$ and $Y \sim \nu$ are independent, then $(X,Y) \sim \mu \otimes \nu$. Likewise, $(X_n)_{n \in \N} \sim \mu^{\otimes\N}$ iff the $X_n$’s are i.i.d. random variables with distribution $\mu$, etc.

\subsection{Vanishing noise}

Consider a probability space $(\Omega, \mathscr F, \prob)$.

Let $X$ a nonnegative $\Lone$ random variable. Let $(Y_n)_{n \in \N}$ a sequence of real, variables converging to $0$ $\prob$-almost surely, such that $(\Exp[|Y_n|])_{n \in \N}$ is bounded, and let for $n \in \N$:
\[ Z_n \eqdef \max(X + Y_n, 0). \]

By dominated convergence theorem, ${\Exp[|Z_n-X|] \xrightarrow[n \to \infty]{} 0}$. Since $\Lone$ convergence of random variables implies $\Wone$ convergence of the underlying measures (see proposition~\ref{scheffe_measures} in appendix), one can apply proposition \ref{convergence_indicators}. We have $L_{Z_n} \xrightarrow[n \to \infty]{} L_X$ uniformly, $G(Z_n) \xrightarrow[n \to \infty]{} G(X)$ and $H(Z_n) \xrightarrow[n \to \infty]{} H(X)$.

In particular, this is the case if $Y_n$ represents a “noise” that decreases. (We cut in 0 to avoid dealing with negative values.)

For instance:

\begin{application}\label{convergence_corollary_noise} Let $(\Omega, \mathscr F, \prob)$ a probability space. Let $X$  a nonnegative random variable with $X \sim \mu \in \M$.
	
	Let $Y \in L^1(\Omega, \mathscr F, \prob)$ and $(\eps_n)_{n \in \N}$ a series with limit $0$. Let:
	\[ Z_n \eqdef \max(X + \eps_n Y, 0) \]
	and $\nu_n$ the distribution of $Z_n$. Then:\nobreakpar
	\begin{enumerate}
		\item $\nu_n \xrightarrow[n \to \infty]{\Wone} \mu$.
		\item $L_{\nu_n} \xrightarrow[n \to \infty]{\|.\|_\infty} L_{\mu}$.
		\item $G(\nu_n) \xrightarrow[n \to \infty]{} G(\mu)$.
		\item $H(\nu_n) \xrightarrow[n \to \infty]{} H(\mu)$.
	\end{enumerate}
\end{application}

\subsection{Sampling}

Recall the following the following, fundamental theorem, proven in \cite[theorem~20.6]{billingsley1995}.

\begin{theorem}[Glivenko--Cantelli]\label{glivenko_cantelli} ~
	Let ${\mu \in \mathscr M_1(\R_+, \Borel)}$ and $(X_n)_{n \in \N} \sim \mu^{\otimes \N}$. Let $\hat\mu_n$ be the empirical measure of $(X_1, \dots, X_n)$. Then, almost surely, $\|F_{\hat\mu_n} - F_\mu\|_\infty \xrightarrow[n \to \infty]{} 0$, where $F_\bullet$ denotes the cumulative distribution functions associated to measures.
\end{theorem}

\begin{application}
	\label{convergence_corollary_sampling}
	Let ${\mu \in \M}$. Let $X_1$, \dots, $X_n$, … i.i.d with distribution $\mu$. Let $\hat\mu_n$ the empirical measure of the subsample $(X_1, \dots, X_n)$. Then, almost surely,\nobreakpar
	\begin{enumerate}
		\item $\hat\mu_n \xrightarrow[n \to \infty]{\Wone} \mu$.
		\item $L_{\hat\mu_n} \xrightarrow[n \to \infty]{\|.\|_\infty} L_{\mu}$.
		\item $G(\hat\mu_n) \xrightarrow[n \to \infty]{} G(\mu)$.
		\item $H(\hat\mu_n) \xrightarrow[n \to \infty]{} H(\mu)$.
	\end{enumerate}
\end{application}

\begin{proof}
	From Glivenko--Cantelli, we infer that $\hat\mu_n \xrightarrow[n \to \infty]{} \mu$ weakly. Furthermore, by the strong law of large numbers, $m_{\hat\mu_n} \xrightarrow[n \to \infty]{} m_\mu$. Then, by Scheffé’s lemma, $\hat\mu_n \xrightarrow[n \to \infty]{\Wone} \mu$. The three other assertions follow from proposition~\ref{convergence_indicators}.
\end{proof}

The results can be reworded in terms of consistency of estimators.
\begin{corollary}
	Let ${\mu \in \M}$. Let $X_1$, \dots, $X_n$, … i.i.d with distribution $\mu$. Let $\hat\mu_n$ the empirical measure of the subsample $(X_1, \dots, X_n)$. Then, almost surely,\nobreakpar
	\begin{enumhypos}
		\item The following estimator of $G_\mu$ is strongly consistent:
		\[ \widehat{G_n} \eqdef \frac{\displaystyle \sum_{i=1}^n \sum_{j=1}^n |X_i - X_j| }{2n \displaystyle \sum_{i=1}^n X_i }. \]
		\item The following estimator of $H_\mu$ is strongly consistent:
		\[ \widehat{H_n} \eqdef \frac{\displaystyle \sum_{i=1}^n \left|X_i - \textstyle\frac{1}{n} \sum_{j=1}^n X_j\right| }{2\displaystyle \sum_{i=1}^n X_i }. \]
		\item The estimators $L^{(x)}_n$ of $L_\mu(x)$ are strongly consistent, uniformly in $x \in [0,1]$, where
		\[\widehat {L^{(x)}_n} \eqdef \frac{\sum_{i=1}^{nx} \left(X_{(1:n)}^\uparrow\right)_i}{\displaystyle \sum_{i=1}^n X_i}, \]
		with the following notations:
		\begin{itemize}
			\item $\left(X_{(1:n)}^\uparrow\right)_i$ is the $i$-th term of the subsequence $(X_1, \dots, X_n)$, reordered increasingly;
			\item $\sum_{i=1}^0 a_i = 0$;
			\item if $x = k + f$ with $k \in \N$ and ${0 < f < 1}$, then \[\sum_{i=1}^x a_i \eqdef f\sum_{i=1}^{k+1} a_i + (1-f)\sum_{i=1}^k a_i.\]
		\end{itemize}
	\end{enumhypos}
\end{corollary}

\paragraph{Uniform integrability of empirical measures.} The following result can be directly deduced from the $\Wone$ convergence of $(\hat\mu_n)_{n \in \N}$.

\begin{corollary}
	\label{unif_int_sampling}
	Let ${\mu \in \M}$ and ${(X_n)_{n \in \N^*}}$ i.i.d. with distribution $\mu$. Let $\hat\mu_n$ the empirical measure of $(X_1, \dots, X_n)$. Then, almost surely, the collection of measures $(\hat\mu_n)_{n \in \N^*}$ is uniformly integrable.
\end{corollary}

An alternate proof is given in appendix~\ref{appendix_alternate_unif_int_sampling}.

\subsection{Quantile approximation}
For $\ell \in \N^*$, the $\ell$-quantile approximation of $\mu$ is a discretization $\tilde\mu_\ell$ that has $\ell$ atoms in the $0$th, $1$st, …, $(\ell-1)$th $\ell$-quantiles of $\mu$, each with probability $\frac 1 \ell$.

\begin{lemma} \label{lemma_inequality_c.d.f._quantiles_approx}
	Let $\mu \in \M$ and $\ell \in \N^*$. Consider the $\ell$-quantile approximation of $\mu$:
	\[ \tilde \mu_\ell \eqdef \frac1\ell \sum_{k=0}^{\ell-1} \delta_{Q_\mu\left(\frac k \ell\right)}. \]
	
	For all $x \in \R_+$,
	\[ F_\mu(x) \le F_{\tilde\mu_\ell}(x) \le F_\mu(x) + \frac1\ell. \]

In other words, $\mu$ dominates $\hat\mu_\ell$ at first order and $\| F_{\hat\mu_\ell} - F_\mu \|_\infty \le \frac 1 \ell$.
\end{lemma}

\begin{proof}
	Let $x \in \R_+$. Consider the unique integer $k_0 \in \lb0, \ell\rb$ such that ${ \frac{k_0}\ell \le F_\mu(x) < \frac{k_0 + 1}\ell}$. We have:
	\begin{align*}
		F_{\tilde\mu_\ell}(x) &= \tilde\mu_\ell([0,x])\\
		&= \frac1\ell \sum_{k=0}^{\ell-1} \delta_{Q_\mu\left(\frac k \ell\right)}([0,x])\\
		&= \frac1\ell \sum_{k=0}^{\ell-1} \1_{ Q_\mu\left(\frac k \ell\right) \le x }\\
		F_{\tilde\mu_\ell}(x) &= \frac1\ell \sum_{k=0}^{\ell-1} \1_{\frac k \ell \le F_\mu(x)}
	\end{align*}

	The inequality $\frac k \ell \le F_\mu(x)$ is true for ${k=0, \dots, k_0}$, and false otherwise. Hence,
	\begin{align*}F_{\tilde\mu_\ell}(x) &= \frac{k_0 + 1}{\ell}.\qedhere
	\end{align*}
\end{proof}

\begin{application}
	\label{convergence_corollary_quantiles}
	Let $\mu \in \M$. for every $\ell \in \N^*$, let
	\[ \tilde \mu_\ell \eqdef \frac1\ell \sum_{k=1}^\ell \delta_{Q_\mu\left(\frac k \ell\right)}. \]
	
	Then $\tilde\mu_\ell \xrightarrow[\ell \to \infty]{\Wone} \mu$, and the conclusions of proposition~\ref{convergence_indicators} stand.
\end{application}

\begin{proof}
	Lemma \ref{lemma_inequality_c.d.f._quantiles_approx} ensures that ${F_{\tilde \mu_\ell} \xrightarrow[\ell \to \infty]{} F_\mu}$ pointwise. Thus, $\tilde\mu_\ell \xrightarrow[\ell \to \infty]{} \mu$ weakly.

	Furthermore, by lemma~\ref{lemma_inequality_c.d.f._quantiles_approx}, for every ${\ell \in \N^*}$, $F_{\mu} \le F_{\tilde \mu_\ell}$. It follows that $Q_{\tilde\mu_\ell} \le Q_\mu$ (see proposition~\ref{definitions_FSD} in appendix for details).
		
	$Q_\mu$ is integrable over $[0,1)$ and has finite integral $m_\mu$ (proposition~\ref{lemma_mean_quantile}). Thus, the collection $(Q_{\tilde\mu_\ell})_{\ell \in \N^*}$ is uniformly dominated by a $\Lone$ random variable; hence it is uniformly integrable.
	
	By Scheffé--Lebesgue $\tilde\mu_\ell \xrightarrow[n \to \infty]{\Wone} \mu$.
\end{proof}

\subsection{Quantile-of-sample approximation}

In the application \ref{convergence_corollary_quantiles}, the uniform integrability followed from the facts that all $Q_{\tilde\mu_\ell}$ were dominated by the same $\Lone$ random variable. In fact, this result can be easily generalized: a collection of random random variables $(X_j)_{j \in J}$ is u.i. as soon as there exists a collection of u.i. variables $(Y_i)_{i \in I}$ and for every $j \in J$, there exists an $i$ such that $X_j \le Y_i$ a.s. (just write it --- or see proposition~\ref{unif_int_FSD} in appendix for detailed proof).

\paragraph{Notation.} If $(u_{n,m})$ is a double-indexed sequence of numbers, we write:
\[ u_{n, m} \xrightarrow[n \to \infty ~\indep~ m \to \infty]{} u_\infty \]
if for every mappings $\varphi, \psi: \N \longrightarrow \N$ with limit $\infty$ we have:
\[ u_{\varphi(i), \psi(i)} \xrightarrow[i \to \infty]{} u_\infty. \]

\emph{Mutatis mutandis} for other modes of convergence (weak or $\Wone$ convergence of measure, uniform or pointwise convergence of functions, etc.). \emph{Mutatis mutandis} for convergence in the neighbourhood of a real number, etc.

\begin{application}[quantiles of sample]
	\label{convergence_corollary_quantiles_and_sampling}
	Let $\mu \in \M$ and $(X_n)_{n \in \N} \sim \mu^{\otimes \N}$ random variables defined on a probability space $(\Omega, \mathscr F, \prob)$. Let $\hat\mu_n$ the empirical measure of $(X_1, \dots, X_n)$. For every $\ell \in \N^*$, let:
	\[ \mathring\mu_{n,\ell} = \frac{1}{\ell} \sum_{k=0}^{\ell-1} \delta_{Q_{\hat\mu_n}\left(\frac k \ell\right)}. \]
	
	Then, $\prob$-almost surely,
	\[
		\mathring\mu_{n,\ell} \xrightarrow[n \to \infty ~\indep~ \ell \to \infty]{\Wone} \mu
	\]
	and the conclusions of proposition~\ref{convergence_indicators} stands.
\end{application}

Notice that $\mathring\mu_{n,\ell}$ is the quantile approximation (with $\ell$ quantiles) of the empirical distribution $\hat\mu_n$.

\begin{proof}
	We note $F_\bullet$ the cumulative distribution functions of the considered measures.
	
	By lemma~\ref{lemma_inequality_c.d.f._quantiles_approx} applied to measure $\hat\mu_n$, for each ${n \in \N^*}$, $\left\|F_{\mathring \mu_{n, \ell}} - F_{\hat\mu_n} \right\| \le \frac1\ell$. By Glivenko--Cantelli, $\left\|F_{\hat\mu_n} - F_\mu \right\|_\infty$ converges to 0 $\prob$-a.s. as ${n \to \infty}$. Hence $\prob$-a.s., $\| F_{\hat\mu_{n,\ell}} - F_\mu\|_\infty \xrightarrow[]{} 0$, i.e. $\mathring\mu_{n,\ell} \xrightarrow[n \to \infty ~\indep~ \ell \to \infty]{\weak} \mu$.
	
	Furthermore, by lemma \ref{lemma_inequality_c.d.f._quantiles_approx}, for each $n \in \N^*$, $\ell \in \N^*$, the measure  $\mathring\mu_{n, \ell}$ is stochastically dominated by $\hat\mu_n$. As we proved that $(\hat\mu_n)_{n \in \N}$ is u.i. (corollary~\ref{unif_int_sampling}), so is $(\mathring\mu_{n, \ell})_{n \in \N, \ell \in \N^*}$. Hence, by theorem~\ref{scheffe_lebesgue}, the $\prob$-a.s. convergence of $\mathring\mu_{n,\ell}$ is $\Wone$.
\end{proof}

\subsection{Kernel density estimation of a sample}

The kernel density estimate (KDE), or Parzen-Rosenblatt estimate, is classical way of approximating a distribution using a sample, by convolving the empirical distribution thereof with a continuous random variable (the “density kernel”) with a scale parameter $h$ (the bandwidth). 

In what follows, we will prove that the Lorenz curve, Gini and Hoover indexes are consistent with the KDE under the following loose hypotheses:\nobreakpar
\begin{itemize}
	\item the kernel density is $\Lone$;
	\item the bandwidth $h$ holds to $0$ and the sample size $n$ hords to $+\infty$, independently.
\end{itemize}

Note that to avoid tricks with negative values, we “cut” the kernel density estimate in zero. More precisely, we define:
\begin{definition}[cut-in-zero KDE]
	Let $\mu$ a measure over $(\R, \Borel)$, $K$ a density function over $\R$, and $h > 0$. Let $(\Omega, \mathscr F, \prob)$ a probability space.
	
	Let $(X, Y)$ a couple of random variables over $(\Omega, \mathscr F, \prob)$ with distribution $\mu \otimes K\Leb$ (i.e., $X \sim \mu$, $Y$ has density $K$ with respect to the Lebesgue measure and $X \indep Y$).
	
	We note $\KDE(\mu, K, h)$ 	the distribution of the random variable ${\max(X + h Y, 0)}$.
\end{definition}

In the most common case, when knowing a sample $(X_1, \dots, X_n)$, the original distribution $\mu$ is approximated as
$ \KDE(\hat \mu_n, K, h)$
where $\hat \mu_n$ is the empirical measure associated with the sample.

We first need two lemmas involving KDEs.

\begin{lemma}\label{lemma_expression_cdf_kde}
	Let ${\mu \in \mathscr M_1(\R_+, \Borel)}$, ${h > 0}$, $K$ a density function over $\R$ and ${G:x\longmapsto\int_{-\infty}^x K(t) \dd t}$.
	
	Then, for all $t \in \R_+$, the c.d.f. of $\KDE(\mu, K, h)$ can be written:
	\[ F_{\KDE(\mu, K, h)}(t) = \int_{\R_+} G\left(\frac{t-x}{h}\right)\dd\mu(x). \] 
\end{lemma}

\begin{proof}
	We have for all $t \in \R_+$:
	\begin{align*}
		& F_{\KDE(\mu, K, h)}(t) \\
		= & \iint_{\subalign{x \in & \R_+ \\ y \in& \R}} \1_{(\max(x + h y, 0) \le t)} \dd\mu(x) K(y) \dd y \\
		= & \iint_{\subalign{x \in & \R_+ \\ y \in& \R}} \1_{(x + h y \le t)} \dd\mu(x) K(y) \dd y \\
		= & \int_{\R_+} \left(\int_\R \1_{(x + h y \le t)} K(y) \dd y\right) \dd\mu(x).
	\end{align*}
	
	Let $Y \sim K\Leb$ on a probability space $(\Omega, \mathscr F, \prob)$. Then:
	\begin{align*}
		\int_0^\infty \1_{(x + h y \le t)} K(y) \dd y
		&= \prob(x + h Y \le t)\\
		&= \prob\left(Y \le \frac{t-x}{h}\right)\\
		\int_0^\infty \1_{(x + h y \le t)} K(y) \dd y &= G\left(\frac{t-x}{h}\right).\qedhere
	\end{align*}
\end{proof}

\begin{lemma}\label{lemma_bounding_kde}
	Let $K : \R \longrightarrow \R_+$ a density. For all $\eps > 0$, there exists $M_{K, \eps} > 0$ such that for all $t \in \R_+$, $h \in \R$ and $\mu \in \mathscr M_1(\R_+, \Borel)$:
	\begin{align*}
		& \left|F_{\KDE(\mu, K, h)}(t) - F_\mu(t)\right| \\*
		< \quad & 2\eps + 2\mu\bigl(\left(t-M_{K,\eps}h,~~ t+M_{K,\eps}h\right]\bigr)
	\end{align*}
\end{lemma}

\begin{proof}
	Let $M \in \R_+$ such that ${\int_{-M}^M K(u) \dd u > 1 - \eps}$. We fix $t$, $h$ and $\mu$ for the rest of the proof of the lemma.
	
	We note $\nu \eqdef \KDE(\mu, K, h)$. Let $G$ the c.d.f. associated with density $K$.
	
	By lemma \ref{lemma_expression_cdf_kde}, write ${F_\nu(t) = A + B + C}$ where:
	\begin{align*}
		A \eqdef & \int_0^{(t-Mh)^+} G\left(\frac{t-x}{h}\right) \dd\mu(x);\\
		B \eqdef & \int_{(t-Mh)^-}^{(t+Mh)^+} G\left(\frac{t-x}{h}\right) \dd\mu(x);\\
		C \eqdef & \int_{(t+Mh)^-}^\infty G\left(\frac{t-x}{h}\right) \dd\mu(x).
	\end{align*}
	
	$G$ is nondecreasing, with values in $[0,1]$, and we have ${G(-M) < \eps}$ and ${G(M) > 1-\eps}$. Hence, we can bound $A$, $B$ and $C$:\nobreakpar
	\begin{itemize}
		\item $(1-\eps) F_\mu(t-Mh) \le A \le F_\mu(t-Mh)$, hence:
		\begin{align*}
			-\eps & \le -\eps ~ F_\mu(t-Mh) \\*
			& \le A - F_\mu(t-Mh) \le 0;
		\end{align*}
		\item $0 \le B \le \mu\bigl(\left(t-Mh, t+Mh\right]\bigr)$;
		\item $0 \le C \le \eps ~ \mu\bigl(\left(t+Mh, \infty\right]\bigr) \le \eps.$
	\end{itemize}
	
	Thus:
	\begin{align*}
		& \left|F_\nu(t) - F_\mu(t)\right| \\
		\le ~ & \left|F_\nu(t)-F_\mu(t-Mh)\right| 
			\\* & \quad + \left|F_\mu(t-Mh)-F_\mu(t)\right|\\
		\le ~ & \left| A + B + C - F_\mu(t - Mh) \right|
			\\* & \quad + \mu\bigl(\left(t-Mh, t\right]\bigr)\\
		\le ~ & \left|A-F_\mu(t-Mh)\right| + |B| + |C|
			\\* & \quad + \mu\bigl(\left(t-Mh, t\right]\bigr)\\
		\le ~ & \eps + \mu\left(\left(t-Mh,t+Mh\right]\right) + \eps 
			\\* & \quad + \mu\bigl(\left(t-Mh, t\right]\bigr)\\
		\le ~ & 2\eps + 2\mu\bigl(\left(t-Mh,t+Mh\right]\bigr).
	\end{align*}
	Hence, the lemma stands with $M_{K,\eps} \eqdef M$.
\end{proof}

\begin{application}\label{convergence_corollary_sampling_kernel}
	Let $\mu \in \mathbf M$ and $K$ a probability density over $\R$ which is $\Lone$.
	
	Let $(\Omega, \mathscr F, \prob)$ a probability space, ${(X_n)_{n \in \N} \sim \mu^{\otimes \N}}$, and $\hat\mu_n$ the empirical measure of $(X_1, \dots, X_n)$.
	
	Then, $\prob$-almost surely:
	\[\KDE(\hat\mu_n, K, h) \xrightarrow[n \to \infty ~\indep~ h \to 0^+]{\Wone} \mu\]
	and the conclusions of proposition~\ref{convergence_indicators} stand.
\end{application}

\begin{proof} $F_\bullet$ denotes the cumulative distribution functions associated to measures. $\prob$-almost surely, $F_{\hat\mu_n} \xrightarrow[n \to \infty]{} F_\mu$ uniformly (Glivenko-Cantelli). We saw (corollary~\ref{unif_int_sampling}) that, $\prob$-almost surely, $(\hat\mu_n)_{n \in \N}$ is uniformly integrable.

Fix $\omega \in \Omega$ such that both previous assertion stand.
	
	\paragraph{Weak convergence.} We prove that the distributions $\KDE(\hat\mu_n, K, h)$ weakly converge to $\mu$ as $h \to 0^+$ and $n \to \infty$ independently.
	
	Let $t \in \R_+$ such that $F_\mu$ is continuous at $t$. Let $\eps > 0$. We chose $N_{\eps} > 0$ and $H_{\eps,t} > 0$ as follows:\nobreakpar
	\begin{itemize}
		\item $N_{\eps}$ is given by Glivenko-Cantelli, such that for every $n \ge N_\eps$, $\|F_{\hat\mu_n} - F_\mu\|_\infty < \eps$.
		
		\item Let $U_{\eps, t} > 0$, given by continuity of $F_\mu$ at $t$, such that for all $u$ with $|u| < U_{\eps, t}$, $|F(t+i) - F(t)| < \eps$. 
		
		Thanks to lemma~\ref{lemma_bounding_kde}, there exists ${M_\eps > 0}$ such that for all ${\nu \in \mathscr M_1(\R_+,\Borel)}$ and for all $r > 0$,
		\begin{align*}
			& \left|F_{\KDE(\nu, K, r)}(t) - F_\nu(t)\right| \\*
			<~ & 2\eps + 2\nu\bigl(\left(t-Mr, t+Mr\right]\bigr).
		\end{align*}
		
		We set: $H_{\eps,t} \eqdef U_{\eps, t} / M_\eps$.
	\end{itemize}
	
	Now, let $n \ge N_{\eps}$. We have:
	\begin{align*}
		& |F_{\KDE(\hat\mu_n, K, h)}(t) - F_\mu(t)| \\
		\le~ & |F_{\KDE(\hat\mu_n, K, h)}(t) - F_{\hat\mu_n}(t)|\!+\!|F_{\hat\mu_n}(t) - F_\mu(t)|\\
		\le~ & |F_{\KDE(\hat\mu_n, K, h)}(t) - F_{\hat\mu_n}(t)|\!+\!\eps.
	\end{align*}
	
	Assume $0 < h < H_{\eps,t}$. By definition of $H_{\eps, t}$:
	\begin{align*}
		& |F_{\KDE(\hat\mu_n, K, h)}(t) - F_\mu(t)| \\*
		\le~ & 3\eps + 2 \hat\mu_n((t-Mh, t+Mh]) \\
		=~ & 3\eps + 2F_{\hat\mu_n}(t+Mh) - 2F_{\hat\mu_n}(t-Mh)\\
		\le~ & 7 \eps + 2F_\mu(t+Mh) - 2F_\mu(t-Mh)\\
		\le~& 11\eps.
	\end{align*}
	which proves the weak convergence.
	
	\paragraph{Uniform integrability.} Let $X_n \sim \hat\mu_n$ (${n \in \N}$) and $Y \sim K~\Leb$ defined on a probability space $(\Omega', \mathscr F', \prob')$. 
	
	The random variable $Y$ is integrable. Hence, the collection ${\{ hY : 0 < h \le 1 \}}$ is u.i. Furthermore, the collection $(\hat\mu_n)_{n \in \N^*}$ is u.i. Hence, by sum, the collection ${\{X_n + h Y : n \in \N^*, 0 < h \le 1\}}$ is u.i. Eventually, the collection ${\{\max(X_n + h Y,0) : {n \in \N^*}, {0 < h \le 1}\}}$ is u.i. too. This means that the collections of measures ${\{ \KDE(\hat\mu_n, K, h)~: {n \in \N^*}, {0 < h \le 1} \}}$ is u.i. (See appendix~\ref{operations_ui} for details about operations on u.i. collections.)
	
	\paragraph{Conclusion.} By theorem~\ref{scheffe_lebesgue}, the $\Wone$ convergence holds.
\end{proof}
\end{multicols}

\begin{multicols}{2}[\section{Weaker asumptions}]
	\label{section_weaker_asumptions_cv}
By Scheffé--Lebesgue, $\Wone$ convergence is equivalent to weak convergence plus (either uniform integrability or convergence of means). The goal of this section is to analyse what happens if only one of these hypotheses stands.

\subsection{Weak convergence without $\Wone$ convergence}
\label{subsection_weak_convergence}

\subsubsection{Weak convergence is not enough for convergence of $L$, $G$ and $H$}

We first give two counter-examples to illustrate the fact that, if $\mu_n \xrightarrow[n \to \infty]{} \mu_\infty$ weakly but not in $\Wone$, almost anything can happen to the Lorenz curves, Gini and Hoover indexes.

\begin{counterex}
	Let $X_n$ random variables such that $\prob(X_n = 1) = 1 - \frac{1}{n^2}$ and $\prob(X_n = n^2) = \frac{1}{n^2}$. Let $\mu_n$ the distribution of $X_n$.
	
	Then, Borel--Cantelli’s lemma states that thet set of $n$’s such that $X_n \neq 1$ is finite, i.e. $(X_n)_n$ holds to the deterministic random variable $\1$ $\prob$-a.s. Thus, $\mu_n \xrightarrow[n \to \infty]{\weak} \delta_1$ weakly. We have ${H(\delta_1) = G(\delta_1) = 0}$ and $L_{\delta_1} = \id$.
	
	for every $n \in \N$, the expectation of $X_n$ is $m_{\mu_n} = \Exp[X_n] = 2 - \frac{1}{n^2}$. Hence, the convergence of $\mu_n$’s is not $\Wone$.
	
	Direct computations show that ${\Exp[|X_n - m_{\mu_n}|] \xrightarrow[n \to \infty]{} 2}$. Hence, $H(\mu_n) \xrightarrow[n \to \infty]{} 1$ and $G(\mu_n) \xrightarrow[n \to \infty]{} 1$. Furthermore, the pointwise limit of $L_{\mu_n}$'s is a discontinuous function $\ell$ such that $\ell(t) = 0$ for $t<1$ and $\ell(1)=1$.
\end{counterex}

\begin{counterex}
	Consider the probability space $(\Omega, \mathscr F, \prob) \eqdef \bigl([0,1), \Borel, \Leb\bigr)$.
	
	For every $n \in \N^*$ and $\omega \in [0, 1)$, let:
	\[ Y_n(\omega) = \left\{ \begin{array}{ll}
		0 & \text{~if~} \omega \le 0.5;\\
		1 & \text{~if~} 0.5 < \omega \le 1 - \frac{1}{n^2};\\
		n^2 & \text{~otherwise.}
	\end{array} \right. \]
	
	The random variables $Y_n$ converge $\prob$-a.s. to a random variable $Y_\infty$ such that $Y_\infty(\omega) = 0$ if $\omega \le 0.5$, $Y_\infty(\omega) = 1$ otherwise. Call $\nu_n$ and $\nu_\infty$ the underlying distributions.
	
	We have: $m_{\nu_\infty} = \Exp[Y_\infty] = 0.5$. However, $m_{\nu_n} = \Exp[Y_n] = \frac{3}{2} - \frac{1}{n^2}$. Hence, $\nu_n \xrightarrow[n\to\infty]{\weak} \nu_\infty$ but not $\Wone$.
	
	$G\left(\nu_\infty\right) = H\left(\nu_\infty\right) = 0.5$, $L_{\nu_\infty}(t) = 0$ if $t \le 0.5$ and $L_{\nu_\infty}(t) = 2t-1$ if $t \ge 0.5$.
	
	However, direct computations show that $G(\nu_n) \xrightarrow[n \to \infty]{} \frac{5}{6}$, $H(\nu_n) \xrightarrow[n \to \infty]{} \frac{2}{3}$ and that $L_{\nu_n}$ converges pointwise to a limit $f$ such that $f(t) = 0$ if $t \ge 0.5$, $f(t) = \frac23 t - \frac13$ if $0.5 \le t < 1$ and $f(1) = 1$.
\end{counterex}

	\subsubsection{Topological properties of ${\Phi : (\M, \weak) \longrightarrow \Lorspace \times \R_+^*}$}
	Now consider the bijection $\Phi : \M \longrightarrow \Lorspace \times \R_+^*$ of proposition~\ref{bijection_M_Lorspace_R}, defined by $\Phi(\mu) = (L_\mu, m_\mu)$. Contrary to theorem~\ref{characterization_lorenz_convergence_topological}, we embed $\M$ with the topology of weak convergence $\weak$. $\weak$ is at least as coarse than the topology induced by metric $\Wone$. Hence, $\Phi^{-1}$ is continuous. However, from the previous counterexamples, follows that $\Phi$ is not continuous.
	
	What can we say about the results of proposition~\ref{homeos_M_Lorspace} if $\M$, its subspaces and quotients are embedded with $\weak$ rather than with $\Wone$?
	
	\paragraph{Restrictions.} It is easy to see that some restrictions of $\Phi$ are continuous, because $\Wone$ and $\weak$ induce the same topologies. For instance:
	\begin{itemize}
		\item Let $\alpha > 0$. Due to Scheffé’s lemma the weak convergence of $(\mu_n)_{n \in \N} \in \M_\alpha^\N$ to some $\mu_\infty \in \M_\alpha$ is equivalent to $\Wone$ convergence. 
		\item Let $\mu \in \M_1$. Let $(\nu_n)_{n \in \N} \in [\mu]^{\N}$ weakly converging to some $\nu_\infty \in [\mu]$. We have $Q_{\nu_n}(t) \xrightarrow[n \to \infty]{} Q_{\nu_\infty}(t)$ almost everywhere. For every $n \in \N \cup \{\infty\}$, $Q_{\nu_n} = m_{\nu_n} Q_{\mu_n}$. As $\{t \in [0,1) : Q_\mu(t) \neq 0\}$ has strictly positive measure, it follows that ${m_{\nu_n} \xrightarrow[n \to \infty]{} m_{\nu_\infty}}$. Finally, due to Scheffé’s lemma, ${\nu_n \xrightarrow[n \to \infty]{\Wone} \nu_\infty}$. 
	\end{itemize}
	
	Since $\Wone$ convergence always implies $\weak$ convergence and both topologies are metrizable, it follows that they induce the same topologies on $\M_\alpha$ and on $[\mu]$.
	
	\paragraph{Quotient.} The quotient space $\M / \R_+^*$ can be embedded with the quotient topology induced by $\weak$, i.e. the thinest topology making the mapping $\mu \longmapsto [\mu]$ continuous whence $\M$ is embedded with $\weak$.
	
	This topology is coarser than the quotient topology induced by $\Wone$. Hence, if $\tilde \Psi$ is the quotient mapping:
	\[ \begin{array}{rcrcl}\tilde \Psi & : & (\M/ \R_+^*,~\weak) & \longrightarrow & \Lorspace \\ && [\mu] & \longmapsto & L_\mu\end{array} \]
	then $\tilde\Psi^{-1}$ is continuous. However, $\tilde\Psi$ is not continuous. This proves that the quotient topology induced by $\weak$ is \ul{strictly} coarser than the quotient topology induced by $\Wone$.

	\subsection{Convergence of means without weak convergence}
	In fact, in both previous counterexamples, the Lorenz curves converge in $[0,1]^{[0,1]}$, but the limit is not a Lorenz curve itself, as it is not continuous at 1. On $[0,1)$, they converge to a “shrinked” Lorenz curve. The following results explore what we may say about the limit behaviour of Lorenz curves in the neighbourhood of 1, with no asumption about weak convergence.
	
	First, recall that if $\ell_n$'s are Lorenz curves and if they converge \ul{pointwise} to any function $\ell_\infty : [0,1] \longrightarrow \R$, then $\ell_\infty$ is nondecreasing, convex, continuous on $[0,1)$ with $\ell_\infty(0) = 0$ and $\ell_\infty(1) = 1$. Furthermore, $\ell_\infty$ is continuous at 1 if and only if $\ell_\infty \in \Lorspace$, if and only if the convergence is uniform.
	
	\subsubsection{Pointwise convergence of Lorenz function implies weak convergences of the measures}
	\label{subsection_convergence_means}
	The following proposition shows what may happen in the convergence in distribution holds.

	\begin{proposition}\label{proposition_pointwize_cv_lorenz_implies_weak_cv}
		Let $(\mu_n)_{n \in \N} \in \M^\N$. Let $m_{\mu_n}$ the mean of $\mu_n$. Assume that $(L_{\mu_n})$ converges pointwise to a function $\ell : [0,1] \longrightarrow \R$ and $(m_{\mu_n})$ converges to some limit $\alpha \in \R_+$. Then:\nobreakpar
		\begin{enumerate}
			\item $(\mu_n)_{n \in \N}$ has a weak limit $\mu_\infty$ in ${\mathscr M_1(\R_+, \Borel)}$.
			\item The mean of $\mu_\infty$ satisfies:
			\[ m_{\mu_\infty} = \ell(1^-) \cdot \alpha \le \alpha < \infty. \]
			\item If $m_{\mu_\infty} > 0$, i.e. ${\mu_\infty \in \M}$, then for all $x \in [0,1)$:
			\[ \ell(x) = \ell(1^-) L_{\mu_\infty}(x). \]
		\end{enumerate}
	\end{proposition}
	
	\begin{proof} Let $m_n$, $Q_n$ and $L_n$ the mean, quantile function and Lorenz function of $\mu_n$.
		\paragraph{Case $\alpha \cdot \ell(1^-) = 0$.} First, assume that either $\alpha = 0$ or $\ell(1^-) = 0$. In both cases,
		\[ \int_0^{1^-} Q_n(p) \dd p \xrightarrow[n \to \infty]{} 0\]
		hence for all $x \in [0,1)$,
		\[ \int_0^x Q_n(p) \dd p \xrightarrow[n \to \infty]{} 0. \]
		
		By Dini (lemma~\ref{dini}), the convergence is uniform over every interval of the form $[0,b]$, $b \in [0,1)$. By lemma~\ref{lemma_unif_convergence_of_convex_implies_convergence_of_derivatives}, $Q_n \xrightarrow[n \to \infty]{} 0$ pointwise over $(0,b)$. Hence, $Q_n \xrightarrow[n \to \infty]{} 0$ pointwise over $(0,1)$. The convergence also stands in $0$. Thus the $\mu_n$’s weakly converge to the Dirac mass $\delta_0$, and the proposition stands.
		
		\paragraph{General case.} Now consider the function:
		\[\begin{array}{rcrcl}
			\tilde \ell &:& [0,1] &\longrightarrow& [0,1] \\
			&& x &\longmapsto & \begin{cases}
				\frac{\ell(x)}{\ell(1^-)} & \text{if~} x < 1\\
				1 & \text{if~} x = 1.\\
			\end{cases}
		\end{array}\]
		
		$\tilde \ell$ is a convex, continuous function taking values $0$ in $0$ and $1$ in $1$. By proposition~\ref{bijection_M_Lorspace_R}, there exists an unique measure $\mu_\infty$ with mean $\alpha \cdot \ell(1^-)$ such that $\tilde\ell$ is the Lorenz function of $\mu_\infty$. It suffices to prove that $(\mu_n)_{n \in \N}$ weakly converges to $\mu_\infty$. Let $Q_\infty$ the quantile function of $\mu_\infty$.
		
		For all $x \in [0,1)$, $\frac{1}{\ell(1^-)} L_n(x) \xrightarrow[n \to \infty]{} \tilde\ell(x)$. By Dini, the convergence is uniform over every compact of form $[0,p]$. Hence, by lemma~\ref{lemma_unif_convergence_of_convex_implies_convergence_of_derivatives}, for every $p \in (0,1)$ such that $Q_\infty = \partial_- \tilde\ell$ is continuous at $p$,
		\begin{align*}
			\frac{1}{\ell(1^-)} \partial_- L_n(p) &\xrightarrow[n \to \infty]{} \partial_- \tilde\ell(p) \\
			\frac{1}{\ell(1^-)} \frac{Q_n(p)}{m_n} & \xrightarrow[n \to \infty]{} \frac{Q_\infty(p)}{\ell(1^-) \cdot \alpha}\\
			Q_n(p) & \xrightarrow[n\to \infty]{} Q_\infty(p).
		\end{align*}
		The convergence also stands for $p = 0$.
		
		Finally, $(\mu_n)_{n \in \N}$ weakly converges to $\mu_\infty$, which concludes the proof.
	\end{proof}
	
	\subsubsection{Uniform integrability and the infimum of the Lorenz curves}
	\label{subsection_ui}
	
	Now we do not necessarily consider sequences of measures, but any collection. We have the following characterization:
	\begin{proposition}\label{equivalences_ui_lim_liminf}
		Let $(\mu_i)_{i \in I}$ a collection of measures in $\M$, with means $m_i$ and Lorenz functions $L_i$.\nobreakpar
		\begin{enumerate}
			\item If the pointwise $\liminf_{i \in I} L_i$ is continuous at 1 and $\sup_{i \in I} m_i < \infty$, then $(\mu_i)_{i \in I}$ is  uniformly integrable.
			\item If $(\mu_i)_{i \in I}$ is uniformly integrable and $\inf_{i \in I} m_i > 0$, then the pointwise $\inf_{i \in I} L_i$ is continuous at 1.
		\end{enumerate}
	\end{proposition}
	
	Notice that to define the pointwise $\liminf$ in point 2, $I$ need not be countable. In all generality, $\liminf L_i(x)$ is defined as the lowest accumulation point of the set $\{ L_i(x) : i \in I \}$.
	
	\begin{proof} Recall that both $\liminf_{i \in I} L_i$ and $\inf_{i \in I} L_i$ take value $1$ in $1$.
				
		\paragraph{Continuity of  pointwise liminf implies uniform integrability.} Let~:
		\[A \eqdef \sup_{i \in I} m_i  < \infty.\]
		
		Let $L \eqdef \liminf_{i \in I} L_i$. Let ${x < y}$. There exists $(i_n)_{n \in \N}$ such that ${\lim_{n \in \N} L_{i_n} (y) = L(y)}$. Yet for every ${n \in \N}$, ${L_{i_n}(x) \le L_{i_n}(y)}$. Thus,
		\[L(x) = \liminf_{i \in I} L_i(x) \le \liminf_{n \in \N} L_{i_n}(x) \le L(y).\]
		Hence, $L$ is nondecreasing.
		
		Consider the probability space $\bigl([0,1), \Borel, \Leb\bigr)$. Let $x \in [0,1)$. For all $i \in I$, let $Q_i$ the quantile function of $\mu_i$. We have:
		\begin{align*}
			(1-x) ~ Q_i(x) &\le \int_x^1 Q_i(x) \dd x\\
			& \le \int_0^1 Q_i(x) \dd x \le A.
		\end{align*}
		Hence, we have for all $u \in [0,1]$:
		\begin{align} u \le x \implies Q_i(u) \le Q_i(x) \le \frac{A}{1-x}. \label{proof_ui_liminf:majoration_Qi}\end{align}
		
		Now let $M > A$. Taking the contraposition of \eqref{proof_ui_liminf:majoration_Qi} and letting $x = 1 - \frac{A}{M}$, we have:
		\[ Q(u) > M \implies u > 1 - \frac{A}{M}. \]
		
		Thus, for all $i \in I$:
		\begin{align*}
			\Exp\left[Q_i ~ \1_{ Q_i > M }\right] &= \int_0^1 Q_i(u) ~ \1_{Q_i(u) > M} \dd u\\
			&\le \int_0^1 Q_i(u) \1_{\left(u > 1 - \frac{A}{M}\right)} \dd u\\
			&= \int_{1-\frac 1 M}^1 Q_i(u) \dd u\\
			&= m_i \left[L_i(1) - L_i\left(1-\frac{A}{M}\right)\right]\\
			\Exp[Q_i \1_{Q_i > M}] &\le A \left[1 - L_i\left(1-\frac{A}{M}\right)\right].
		\end{align*}
		
		Hence, to prove that $(\mu_i)_i$ is u.i., it suffices to prove that the rightmost side is arbitrarly small for $M$ big enough, uniformly in $i$. Let $\eps > 0$.
		
		Since $L$ is continuous at 1, there exists $M > 0$ such that $L\left(1-\frac A M\right) > 1 - \frac{\eps}{A}$. Furthermore, by definition of lim inf, there exists $J \subseteq I$ such that $I\setminus J$ is finite and or all $j \in J$,
		\[ L_j\left(1-\frac A M\right) > L\left(1-\frac A M\right) - \frac{\eps}{A}. \]
		Then in particular, for all $x > 1-\frac A M$, for all $j \in J$, $L_j(x) > 1 - 2\frac{\eps}{A}$.
		
		Hence, for all $j \in J$, $\Exp\left[Q_j \1_{Q_j > M}\right] < 2\eps$, and $I \setminus J$ is finite. This is enough to show that the $\mu_i$’s are u.i. (see lemma~\ref{characterization_ui_subsets} in appendix).
		
		\paragraph{Uniform integrability implies continuity of the pointwise infimum.} Let ${A \eqdef \sup_{i \in I} \frac{1}{m_i} < \infty}$.
		
		Let $\eps > 0$. Notice that for all $M > 0$, $x \in [0,1)$ and $i \in I$, we have:
		\begin{align*}
			&\int_x^1 Q_i(t) \dd t\\
			= & \int_x^1 Q_i(t) \1_{Q_i(t) > M} \dd t + \int_x^1 Q_i(t) \1_{Q_i(t) \le M} \dd t\\
			\le & \int_x^1 Q_i(t) \1_{Q_i(t) > M} \dd t + (1-x)M
		\end{align*}
		
		As the $Q_i$’s are uniformly integrable, there exists $M$ such that the first term is less than $\eps$. Let $\eta = \frac{\eps}{M}$. If $x > 1 - \eta$, then we have:
		\[ \int_x^1 Q_i(t) \dd t \le 2\eps. \]
		
		Furthermore, for all $i \in I$,
		\[ 1 - L_i(x) = \frac{1}{m_i}\int_x^1 Q_i(t) \dd t \le 2A\eps. \]
		
		Finally, $1 - \inf_{i \in I} L_i(x) \le 2A\eps$. This proves that $\inf_{i \in I}L_i$ is continuous at $1$.
	\end{proof}
	
	Notice that for all $x \in [0,1]$,
	\[\inf_{i \in I} L_i(x) \le \liminf_{i \in I} L_i(x) \le 1. \]
	Hence, if $\inf L_i$ is continuous at 1, so is $\liminf L_i$. In other words, we have the following (less strong, but more elegant) characterization:
	\begin{corollary}
		Let $(\mu_i)_{i \in I}$ a collection of measures in $\M$ with means $m_i$ and Lorenz functions $L_i$. Assume that both $m_i$’s and $1 / m_i$’s are bounded. The following are equivalent~:
		\begin{enumhypos}
			\item $(\mu_i)_{i \in I}$ is uniformly integrable;
			\item The pointwise infimum function $\inf_{i \in I} L_i$ is continuous at 1;
			\item The pointwise liminf function $\liminf_{i \in I} L_i$ is continuous at 1.
		\end{enumhypos}
	\end{corollary}
	
	\subsubsection{Conclusion}
	The following corollary can be directly deduced from each of the propositions~\ref{proposition_pointwize_cv_lorenz_implies_weak_cv} and~\ref{equivalences_ui_lim_liminf} and the use of Scheffé--Lebesgue and proposition~\ref{bijection_M_Lorspace_R}.
	\begin{corollary}
		Assume ${(\mu_n)_{n \in \N} \in \M^\N}$ is uniformly integrable, $(L_{\mu_n})_{n \in \N}$ converges pointwise to some ${\ell \in [0,1]^{[0,1]}}$ and $(m_{\mu_n})_{n \in \N}$ converges to some $\alpha \in \R_+^*$.
		
		Then $\ell \in \Lorspace$ and $\mu_n \xrightarrow[n \to \infty]{\Wone} \mu_\infty$, where $\mu_\infty = \Phi^{-1}(\ell, \alpha)$ is the only distribution with Lorenz curve $\ell$ and mean $\alpha$.
	\end{corollary}
\end{multicols}

\section*{Aknowledgements}
	The author thanks Hippolyte d’Albis, Guillaume Conchon-Kerjean, David Leturcq and Anne Perrot, whose advices during the writing of this article were precious.

	\section*{References}
	\addcontentsline{toc}{section}{References}
	\printbibliography[heading=none]

\newpage

\appendix
\addcontentsline{toc}{part}{Appendices}
\part*{Appendices}

\begin{multicols}{2}[\section{Details of notations and basic definitions}]
	\label{appendix_notations}
	The notations of this article are chosen to be consistent with \cite{legall2022}.
	
	Let $\R_+$ the set of nonnegative real numbers, $\R_+^*$ the set of strictly positive numbers, and $\Rbarplus$ the set $\R_+ \cup \{\infty\}$.
	
	\paragraph{Elementary notations.}
	If $S$ is a set, we note $\1_S$ the indicator function of $S$. $\id_S$ is the identity function of $S$. In the absence of ambiguity, we simply write $\id$.
	
	Let $A$ an assertion. We note $\1_{(A)} = 1$ if $A$ is true, $0$ otherwise. If no ambiguity, we simply write $\1_A$.
	
	If $f: A\longrightarrow B$ is a mapping, $\alpha \subseteq A$ and $\beta \subseteq B$, we note $f\langle\alpha\rangle$ the direct image of $\alpha$ and $f^{-1}\langle\beta\rangle$ the inverse image of $\beta$ through $f$.
	
	If $f$ is a real-valued function defined over a real interval $I$ and $x \in I$, we note $f(x^-)$ (resp. $f(x^+)$) the left-limit (resp. right-limit) of $f$ in $x$ (if exists). We note $\partial_- f(x)$ (resp. $\partial_+ f(x)$) the left derivative (resp. right derivative) of $f$ in $x$ (if exists).
	
	\paragraph{Lebesgue integrals.}
	If $\mathcal X$, is a topological space, we note note $\Borel(\mathcal X)$ the $\sigma$-algebra of Borel sets of $I$ (i.e. the $\sigma$-algebra generated by the open sets of $\mathcal X$). If $I$ is a subset of $\R$, we embed $I$ with the trace topology induced by the standard topology of $\R$; then $\Borel(I)$ is the trace $\sigma$-algebra of $\Borel(\R)$ on $I$. We call $\Leb_I : \Borel(I) \longrightarrow \R_+$ the Lebesgue measure over the measurable space $(I, \Borel(I))$. If no ambiguity, we simply note $\Borel$ and $\Leb$.
	
	For any Borel measure $\nu$ over $I$ and any measurable function $f:I \longrightarrow \R_+$, if ${a, b \in I \cup \{-\infty, +\infty\}}$ with $a \le b$, we denote:
	\[ \int_a^{b^-} f(x) \dd \nu(x) \eqdef \int_{[a,b)} f\dd\nu = \int_{I} \1_{[a,b)}~f \dd\nu \]
	\emph{mutatis mutandis} for intervals of form $(a,b]$, $[a,b]$ and $(a,b)$. If $f$ is defined on the interval $[a,b)$ but not in $b$ and if no confusion is possible, we allow to write:
	\[ \int_a^b f(x) \dd\nu(x) \eqdef \int_{[a,b)} f \dd\nu.\]
	
	When the integral is computed with respect to the Lebesgue measure, we drop the symbol $\Leb$, i.e. we simply write:
	\[ \int_a^b f(x) \dd x \eqdef \int_{[a,b]} f \dd\Leb. \]

	\paragraph{Measures and probability distributions.}
	
	Let $(\Omega, \mathscr F)$ a measurable space. We note $\mathscr M(\Omega, \mathscr F)$ the set of measures over it, and $\mathscr M_1(\Omega, \mathscr F)$ the subset of probability measures, i.e. the measures $\mu$ such that $\mu(\Omega) = 1$. If no ambiguity on the measurable space, we may simply write $\mathscr M$ and $\mathscr M_1$.
	
	Let $(\Omega', \mathscr F')$ another measurable space, ${f: (\Omega, \mathscr F) \longrightarrow (\Omega', \mathscr F')}$ a measurable function, ${\nu \in \mathscr M(\Omega, \mathscr F)}$. Then we note $f_\sharp(\nu)$ the \emph{pushforward of $\nu$ through $f$}, i.e. the measure over $(\Omega', \mathscr F')$ such that for all $A \in \mathscr F'$, ${f_\sharp(\nu)(A) = \nu(f^{-1}\langle A\rangle)}$.
	
	Let $\mu \in \mathscr M(\Omega, \mathscr F)$. A point $\omega \in \Omega$ is called an \emph{atom} of $\mu$ if $\mu(\{\omega\}) > 0$.  A measure $\mu$ is called \emph{nonatomic} or \emph{diffuse} if it has no atom.
	
	If $\omega \in \Omega$, we note $\delta_\omega$ the \emph{Dirac mass in $\omega$}, i.e. the measure such that for any $A \in \mathscr F$, $\delta_\omega(A) \eqdef \1_A(\omega)$.
	
	Let $(\Omega, \mathscr F, \prob)$ a probability space and $(\mathcal X, d)$ a metric space. Fix any $x_0 \in \mathcal X$. For $p > 0$, a random variable $\Omega \longrightarrow \mathcal X$ is said to be $L^p$ if $\int_\Omega d(x_0, X)^p \dd\prob < \infty$ (this does not depend on the choice of $x_0$).
	
	Let $X : (\Omega, \mathscr F) \longrightarrow (A, \mathscr G)$ a random variable. If $\mu \in \mathscr M_1(A, \mathscr G)$, then we note $X \sim \mu$ when the distribution of $X$ is $\mu$, i.e. when $X_\sharp(\prob) = \mu$. Recall that in such case, the \emph{law of the unconscious statistician} \cite[proposition~8.5]{legall2022} states that for every measurable (resp. $\mu$-integrable) function ${f : A \longrightarrow \Rbarplus}$ (resp. ${f : A \longrightarrow \R}$):
	\[ \Exp[f(X)] = \int_\Omega f(X(\omega)) \dd\prob(\omega) = \int_A f \dd \mu. \]

	\paragraph{Cumulative distribution functions.} Let $\mu$ a probability distribution over $(\R_+, \Borel)$. We note $F_\mu : \R_+ \longrightarrow [0,1)$ the \emph{cumulative distribution function} (\emph{c.d.f.}) of $\mu$, i.e. the function such that for all $x \in \R_+$,
	\[ F_\mu(x) \eqdef \mu([0,x]). \]
	
	Recall this function is nondecreasing, right-continuous. Its left-discontinuities correspond to the atoms of $\mu$. It admits a limit equal to $1$ in $+\infty$.
	
	If $\mu$ and $\nu$ are two measures over $(\R_+, \Borel)$ such that $F_\mu = F_\nu$, then they match over all the sets of form $[0,x]$. Yet, the class of sets of form $[0,x]$ is closed under finite intersections and generates $\Borel(\R_+)$ as a $\sigma$-algebra. Furthermore, the class of the sets $A$ such that $\mu(A) = \nu(A)$ is a monotone class. Hence, the monotone class theorem \cite[theorem~1.18, p.~13]{legall2022} ensures that $\mu$ and $\nu$ match on \emph{every} set. In other words, $F_\mu$ fully characterizes $\mu$.
	
	\paragraph{First-order stochastic domination.} Let $\mu, \nu \in \mathscr M_1(\R_+, \Borel)$. We say that $\mu$ \emph{stochastically dominates $\nu$ at first-order} (FSD) if for all $x \in \R_+$, $F_\nu(x) \ge F_\mu(x)$.
	
	Notice that most authors add an hypothesis that $F_\nu$ and $F_\mu$ differ in at least one point. For the whole article, we choose to keep this non-strict definition.
	
	\paragraph{Mean and expectation.}
	If $X$ and $Y$ are real-valued, nonnegative random variables over the same probability space $(\Omega, \mathscr F, \prob)$, we note $\Exp[X]$ the expectation of $X$ and $\Exp[X \mid Y]$ the expectation of $X$ conditional to $Y$. We allow them to be equal to $+\infty$.
	
	If $\mu$ is a probability measure on $(\R_+, \Borel)$, we note:
	\[m_\mu \eqdef \int_0^\infty x \dd\mu(x) \in [0, \infty ]\]
	the \emph{average} or \emph{mean} of $\mu$. Notice that if $X \sim \mu$, then $\Exp[X] = m_\mu$.
	
	\begin{lemma} \label{lemma_mean_survival}
		Let $\mu \in \mathscr M_1(\R_+, \Borel)$, $m_\mu$ its mean, $F_\mu$ its c.d.f. The following equality holds in $\Rbarplus$:
		\[ m_\mu = \int_0^\infty (1-F_\mu(x))\dd x. \]
	\end{lemma}
	
	\begin{proof}
		We use Fubini’s theorem so we can swap two integrals:
		\begin{align*}
			\int_0^\infty (1-F_\mu(x)) \dd x
			&= \int_0^\infty \mu\bigl(x,\infty)\bigr) \dd x\\
			&= \int_{x=0}^\infty \int_{u=x^+}^\infty \dd\mu(u) \dd x\\
			&= \int_{u=0}^\infty \int_{x=0}^{u^-} \dd x \dd\mu(u)\\
			&= \int_{0}^\infty u \dd\mu(u)\\
			\int_0^\infty (1-F(x_\mu)) \dd x &= m_\mu.\qedhere
		\end{align*}
	\end{proof}
	
	\paragraph{Measures on product spaces.} Let $(X, \mathscr \Sigma, \mu)$ and $(X', \Sigma', \mu')$ two measure spaces. We note $\Sigma \otimes \Sigma'$ the product $\sigma$-algebra on $X \times X'$, and ${\mu \otimes \mu' \in \mathscr M({X \times X'}, {\Sigma \otimes \Sigma'})}$ the tensor product of measures $\mu$ and $\mu'$, \emph{mutatis mutandis} for products indexed by collections.
	
	Notice that $\mu\otimes \mu'$ is the distribution of a couple of independent variables of distributions $\mu$ and $\mu'$; $\mu^{\otimes \N}$ the distribution of a countable collection of i.i.d. variables, each one having distribution $\mu$; etc.
	
	\paragraph{Rescalings.} For all $\alpha > 0$, let ${S_\alpha : x \in \R_+ \longmapsto \alpha\cdot x}$. If $\mu \in \M_1(\R_+, \Borel)$ and $X$ is a random variable with distribution $\alpha$ on some probability space, $(S_\alpha)_\sharp(\mu)$ is the distribution of $\alpha\cdot X$. We say that $(S_\alpha)_\sharp(\mu)$ is the \emph{rescaling of $\mu$ of factor $\alpha$}.
	
	The relation:
	\[ \mu \equiv \nu \iff \exists \alpha > 0,~ \mu = (S_\alpha)_\sharp(\nu) \]
	(“$\mu$ and $\nu$ are equal up to a rescaling”) is an equivalence relation on $\M$. Since every equivalence class $[\mu]$ is one-to-one with $\R_+$ (take $\nu \longmapsto m_\nu$), we write $\M / \R_+^*$ for the quotient $\M / \equiv$.
\end{multicols}

\begin{multicols}{2}[\section{Some elementary results and prerequisites}]
	We assume all results about elementary theory of measure (e.g. chapters~1--3 and~5 in \cite{legall2022}) and the foundations of probability theory (chapter~8 in \cite{legall2022}). The following appendix recalls and gives a proof of the most important results in probability theory needed for the present article.
	
	\subsection{Quantile functions}\label{appendix_quantile_functions}
	\subsubsection{Definition and elementary properties}
	Let $\mu \in \mathscr M_1(\R_+, \Borel)$ and $F_\mu$ its c.d.f. For all $p \in [0,1)$, the set
	\[ S_p = \{q \in \R_+ : F_\mu(q) \ge p\} \]
	is nonempty, because $F$ has a limit equal to $1$ in $+\infty$. Hence, it admits a greatest lower bound $\inf S_x$. By right-continuity of $F$, $\inf S_x \in S_x$. Hence, we can define the \emph{quantile function} of $\mu$ as follows:
	\begin{definition} Let $\mu \in \mathscr M_1 (\R_+, \Borel)$, $F_\mu$ its cumulative distribution function. The \emph{quantile function of $\mu$} is the function ${Q_\mu : [0,1) \longrightarrow \R_+}$ such that for all $p \in [0,1)$,
		\[
		Q_\mu(p) \eqdef \min~\{ q \in \R_+  : F_\mu(q) \ge p \}. \\
		\]
	\end{definition}
	
	We immediately have:\nobreakpar
	\begin{itemize}
		\item $Q_\mu(0) = 0$.
		\item $Q_\mu$ is nondecreasing.
	\end{itemize}
	
	Furthermore, we can show that:
	\begin{itemize}
		\item If $F_\mu$ is a surjection, then $Q_\mu$ is strictly increasing.
		\item For all $p \in [0,1)$, $F_\mu (Q_\mu (p)) \ge p$, with equality if and only if $p \in F_\mu\langle\R_+\rangle$.
	\end{itemize}
	
	\subsubsection{Galois inequalities and immediate applications}
	\label{galois_properties}
	
	The definition of the quantile function immediately leads to the following equivalence, known as the “Galois inequalities”:
	\begin{gather*}
		\text{for all $p \in [0,1)$, $q \in \R_+$,}\\*
		Q_\mu(p) \le q \iff p \le F_\mu(q),\\*
		q < Q_\mu(p) \iff F_\mu(q) < p.
	\end{gather*}
	
	We directly deduce from the Galois inequalities that for all ${\mu \in \mathscr M_1(\R_+, \Borel)}$:\nobreakpar
	\begin{itemize}
		\item For all $x \in \R_+$, $Q_\mu(F_\mu(x)) \le x$, and the inequality is strict if and only if there exists $\eps > 0$ such that $F(x-\eps) = F(x)$.
		\item $Q_\mu \circ F_\mu \circ Q_\mu = Q_\mu$ and $F_\mu \circ Q_\mu \circ F_\mu = Q_\mu$.
		\item $Q_\mu$ is constant on all interval of form $[p, F_\mu(Q_\mu(p))]$ (which may be a singleton).
		\item If $Q_\mu$ is strictly increasing, then $F_\mu\langle\R_+\rangle$ contains $[0,1)$.  
	\end{itemize}
	
	An important proposition is the following.
	
	\begin{proposition}[continuity of $Q_\mu$]\label{quantile_continuity}
		Let $\mu \in \mathscr M_1(\R_+, \Borel)$.
		\begin{enumerate}
			\item $Q_\mu$ is left-continuous at every point of $(0,1)$.
			\item Furthermore, $Q_\mu$ is right-continuous at $p \in [0,1)$ if and only if $F_{\mu}^{-1}\langle\{p\}\rangle$ containst at most one element.
		\end{enumerate}
	\end{proposition}
	\begin{proof}
		\begin{enumerate}
		\item Let $p \in (0,1)$.
		As $Q_\mu$ is nondecreasing, it admits a left-limit $Q(p^-)$ in $p$. We immediately have $Q(p^-) \le Q(p)$.
		
		Furthermore, for all $\pi < p$, $Q(p^-) \ge Q(\pi)$. By Galois inequalities, $F(Q(\pi^-)) \ge \pi$. Taking the lowest upper bound for $\pi < p$, it follows that ${F(Q(p^-)) \ge p}$, so ${Q(\pi^-) \ge Q(p)}$.
		
		Hence, ${Q(p^-) = Q(p)}$.
		
		\item 
		\begin{itemize}
			\item Assume that $F_{\mu}^{-1}\langle\{p\}\rangle$ is empty. Then $F_\mu(Q_\mu(p)) > p$ and $Q_\mu(F_\mu(Q_\mu(p))) = Q_\mu(p)$. Hence, $Q_\mu$ is constant between $p$ and $F_\mu(Q_\mu(p))$, thus in particular right-continuous at $p$.
			\item Suppose that $F_{\mu}^{-1}\langle\{p\}\rangle$ is a singleton $\{x\}$. Then $F_\mu(x) = p$. By right-continuity of $F_\mu$, for every ${n \in \N}$, there exists $\eps_n$ such that $p < F_\mu(x + \eps_n) < p + 2^{-n}$.  Then, $Q_\mu(p^+) = \lim_{n \to \infty} Q_\mu(F_\mu(x + \eps_n))$. But $Q_\mu(F_\mu(x+\eps_n)) \le x+\eps_n$. Since $\eps_n \xrightarrow[n \to \infty]{} 0$, it implies that ${Q(p^+) \le Q(p)}$. Hence, $Q_\mu$ is right-continuous at $p$.
			\item Assume $F^{-1}_\mu\langle\{p\}\rangle$ contains at least two elements $x < y$. Then $Q_\mu(p) \le x$ and for all $\eps > 0$, $Q_\mu(p+\eps) > y$, so $Q_\mu(p^+) \ge y$. Hence, $Q_\mu$ is not continuous at $p$.\qedhere
		\end{itemize}
		\end{enumerate}
	\end{proof}
	
	\paragraph{Inverse Galois.} \label{inverse_galois} The Galois inequalities are invalid if the sign is changed. However, for $q \in \R_+$, $p \in [0,1)$, one has:
	\begin{align*}
		Q_\mu(p) < q & \iff \exists \eps > 0, Q_\mu(p) \le q - \eps\\
		& \iff \exists \eps > 0, p \le F_\mu(q-\eps)\\
		Q_\mu(p) < q & \iff p \le F_\mu(q^-).
	\end{align*}
	Notice that since $F_\mu$ must be constant in the neighbourhood of $q$, the latter inequality may be an equality.
	
	Hence, taking the negation, one has:
	\[ q \le Q_\mu(p) \iff F_\mu(q^-) < p.\]
	
	By the same reasoning,
	\[ F_\mu(q) \le p \iff q < Q_\mu(p^+). \]
	
	\subsubsection{Characterizing measures from quantile function}
	Let $\mu \in \mathscr M_1(\R_+, \Borel)$. Since $Q_n$ is nondecreasing, it is a measurable function:
	\[ Q_\mu : \bigl([0,1), \Borel\bigr)  \longrightarrow (\R_+, \Borel). \] 
	
	\begin{lemma}[$Q_\mu$ characterizes $\mu$]
		\label{characterization_by_quantile}
		Let ${\mu \in \mathscr M_1(\R_+, \Borel)}$ and $Q_\mu$ its quantile function. Then, $ \mu = (Q_\mu)_\sharp\left(\Leb_{[0,1)}\right)$.
		
		In other words, $Q_\mu$ is a random variable on the space $\bigl([0,1), \Borel, \Leb\bigr)$, and we have: $Q_\mu \sim \mu$.
	\end{lemma}
	
	This means, in particular, that a probability measure over $(\R_+, \Borel)$ is fully characterized by its quantile function.
	
	\begin{proof}
		Let us work in the probability space $\bigl([0,1), \Borel, \Leb\bigr)$. By definition, the random variable $Q_\mu$ has distribution $\nu \eqdef (Q_\mu)_\sharp(\Leb_{[0,1)})$. It suffices to prove that it also has distribution $\mu$.
		
		Thanks to Galois inequalities, $F_\nu$ is such that for all ${x \in \R_+}$:
		\begin{align*}
			F_{\nu}(x) & \eqdef \nu([0,x]) = (Q_\mu)_\sharp \Leb([0,x])\\
			&= \Leb(\{\omega \in [0,1) : Q_\mu(\omega) \le x\})\\
			&= \Leb(\{\omega \in [0,1) : \omega \le F_\mu(x)]\})\\
			&= \Leb([0, F_\mu(x)])\\
			F_{\nu}(x) & = F_\mu(x).
		\end{align*}
		
		Since they have same c.d.f, $\mu$ and $\nu$ are equal.
	\end{proof}
	
	Thanks lemma to this lemma and LOTUS, the integrals with respect to $\mu$ can be expressed as integrals with respect $\Leb$ implying the function $Q_\mu$:
	\begin{proposition}[LOTUS with $Q_\mu$]\label{pushforward_quantile}
		Let $\mu \in \mathscr M_1(\R_+, \Borel)$, $Q_\mu$ its quantile function. Let $f : \R_+ \longrightarrow \R_+$ is a measurable function. Then the following equality holds in $\Rbarplus$:
		\[ \int_0^1 f(Q_\mu(p)) \dd p = \int_0^\infty f(x) \dd\mu(x). \]
	\end{proposition}
	
	\subsubsection{Several properties that can be characterized through the quantile function}
	
	The three following results consist in characterizing some properties of measures through their quantile functions.
	
	\begin{proposition} \label{lemma_mean_quantile}
		Let $\mu \in \mathscr M_1(\R_+, \Borel)$, $m_\mu$ its mean, and $Q_\mu$ its quantile function. We have, in $\Rbarplus$:
		\[ m_\mu = \int_0^1 Q_\mu(u) \dd u. \]
	\end{proposition}
	
	\begin{proof}
		This is a direct application of proposition~\ref{pushforward_quantile} where $f$ is the identity function.
	\end{proof}
	
	\begin{proposition}[rescaling]\label{rescale_quantile}
		Let $\mu$ and $\nu$ elements of $\mathscr M_1(\R_+, \Borel)$, $\alpha \in \R_+^*$. $\nu$ is a rescaling of $\mu$ of factor $\alpha$ if, and only if, $Q_\nu = \alpha \cdot Q_\mu$.
	\end{proposition}
	\begin{proof}
		Assume $Q_\nu = \alpha \cdot Q_\mu$. Thanks to lemma~\ref{characterization_by_quantile}, we know that $Q_\nu$ and $Q_\mu$ can be seen as random variables over the probability space $\bigl([0,1), \Borel, \Leb\bigr)$ such that $Q_\nu \sim \nu$ and $Q_\mu \sim \mu$. Thus, $\nu$ is a rescaling of $\mu$ with a factor $\alpha$.
		
		Conversely, assume that there exists a random variable $X$ on any probability space $(\Omega, \mathscr F, \prob)$ such that $X \sim \mu$ and $\alpha X \sim \nu$. Let $F_\mu$ and $F_\nu$ be the c.d.f. of those distributions. We have for any $t \in [0,1)$:
		\begin{align*}
			Q_\nu(t) &= \min \{ q : F_\nu(q) \ge p \}\\
			&= \min \{ q : \prob(\alpha X \le q) \ge p\}\\
			&= \alpha \min \{ q : \prob(\alpha X \le \alpha q) \ge p \}\\
			&= \alpha \min \{ q : \prob(X \le q) \ge p)\}\\
			&= \alpha \min \{ q : F_\mu(q) \ge p \}\\
			Q_\nu(t) &= \alpha\; Q_\mu(t).\qedhere
		\end{align*}
	\end{proof}
	
	\begin{proposition}[FSD]\label{definitions_FSD}
		Let $\mu$ and $\nu$ elements of $\mathscr M_1(\R_+, \Borel)$, $F_\mu$ and $F_\nu$ their c.d.f., and $Q_\mu$ and $Q_\nu$ their quantile functions. The three following statements are equivalent:
		\begin{enumhypos}
			\item For all $x \in \R_+$, $F_\mu(x) \le F_\nu(x)$, i.e. $\mu$ stochastically dominates $\nu$ at first order. \label{definitions_FSD:p1}
			\item For all $p \in [0,1)$, $Q_\mu(p) \ge Q_\nu(p)$. \label{definitions_FSD:p2}
			\item There exists a probability space $(\Omega, \mathscr F, \prob)$ and two random variables $X$ and $Y$ over it such that $X \sim \mu$, $Y \sim \nu$ and $X \ge Y$ $\prob$-almost surely.  \label{definitions_FSD:p3}
		\end{enumhypos}
	\end{proposition}
	\begin{proof} ~\nobreakpar
		\begin{itemize}
			\item Assume \ref{definitions_FSD:p1}. Let $p \in [0,1)$. Setting $x \eqdef Q_\mu(p)$, we have,
			\[F_\mu(Q_\mu(p)) \le F_\nu(Q_\mu(p)).\]
			By Galois inequalities, we get
			\[Q_\nu(F_\mu(Q_\mu(p))) \le Q_\mu(p).\]
			Furthermore, $p \le F_\mu(Q_\mu(p))$. $Q_\nu$ being nondecreasing, we get:
			\[Q_\nu(p) \le Q_\nu(F_\mu(Q_\mu(p))). \]
			Hence, \ref{definitions_FSD:p2} stands.
			
			\item Assume \ref{definitions_FSD:p2}. Consider the probability space $\bigl([0,1), \Borel, \Leb\bigr)$. Let $X = Q_\mu$ and $Y = Q_\nu$. According to lemma \ref{characterization_by_quantile}, they match the conditions of \ref{definitions_FSD:p3}.
			
			\item Assume \ref{definitions_FSD:p3}. Let $t \in \R_+$. Then with probability 1, ${\1_{(X \le t)} \ge \1_{(Y \le t)}}$. Taking the expectation, we get ${\prob(X \le t) \ge \prob(Y \le t)}$, which proves~\ref{definitions_FSD:p1}.\qedhere
		\end{itemize}
	\end{proof}

	\subsubsection{Extensions to $[0,1]$ and $[0, \infty]$}
	We rigorously defined $F$ and $Q$ as mappings between the sets $[0,1)$ and $[0,\infty)$.
	
	In fact, it can be sensible to “extend by left-continuity” the function $Q$ to $1$, by letting $Q(1) \eqdef Q(1^-)$. This notation is consistent with the definition, since we have:
	\[Q(1) = \min\{ q \in \Rbarplus : F(q) \ge 1 \} \]
	if we set $F(\infty) \eqdef \mu([0,\infty)) = 1$.
	
	With these notations, $F$ and $Q$ are extended to mappings between the sets $[0,1]$ and $[0, \infty]$. However, until the end of the article, to avoid confusions, we chose to keep the rigorous definitions of $F$ and $Q$, and we will restrict ourselves to the right-open intervals $[0,1)$ and $\R_+$.	

	\subsection{Weak convergence of measures}
	\label{appendix_weak_cv}
	\subsubsection{Generalities}
	
	For this subsection, fix $(\mathcal X, d)$ a separable, complete metric space. We call $C_b(\mathcal X)$ the set of continuous, bounded functions $\mathcal X \longrightarrow \R$. If no ambiguity, we simply note $\mathscr M_1$ for the set $\mathscr M_1(\mathcal X, \Borel(\mathcal X))$ of probability measures on $(\mathcal X, \Borel(\mathcal X))$.
	
	\begin{definition}\begin{enumerate}
		\item The \emph{weak topology} on $\mathscr M_1$ is the topology $\weak$ generated by the elementary balls of form
		\[ \left\{ \nu \in \mathscr M_1 : \left|\int_{\mathcal X} f \dd\nu - \int_{\mathcal X} f \dd\mu \right| < \eps \right\} \]
		for given $f \in C_b(\mathcal X)$, $\mu \in \mathscr M_1$ and $\eps > 0$. In other words, the open sets are arbitrary unions of finite intersections of elementary balls.
		
		\item We call \emph{weak convergence} the convergence with respect to the topology $\weak$. We write $\mu_n \xrightarrow[n \to \infty]{\weak} \mu_\infty$ if $(\mu_n)_{n \in \N}$ weakly converges to $\mu_\infty$. It is immediate to check that this is equivalent to having
		\[ \int_{\mathcal X} f \dd \mu_n \xrightarrow[n \to \infty]{} \int_{\mathcal X} f \dd \mu_\infty \]
		for all $f \in C_b(\mathcal X)$.
		
		\item Let $X_1, \dots, X_n, \dots, X_\infty$ random variables taking values in $\mathcal X$, each being defined on a probability space $(\Omega_n, \mathscr F_n, \prob_n)$. Let $\mu_1, \dots, \mu_n, \dots, \mu_\infty$ their distributions. We say that $(X_n)_{n \in \N}$ \emph{converges in distribution} to $X_\infty$ if $\mu_n \xrightarrow[n \to \infty]{\weak} \mu_\infty$. By LOTUS, this is equivalent to having, for all $f \in C_b(\mathcal X)$,
		\[ \int_{\Omega_n} f(X_n) \dd\prob_n \xrightarrow[n \to \infty]{} \int_{\Omega_\infty} f(X_\infty) \dd\prob_\infty. \]
	\end{enumerate}\end{definition}

	Most authors only define the weak convergence for sequences. However, \emph{a priori}, there is no reason for $(\mathscr M_1, \weak)$ to be sequential, i.e. there is no guarantee that convergent sequences fully characterize the topology.
	
	We immediately have the following fact:
	\begin{proposition}\label{as_cv_implies_weak_cv}
		Let $(\Omega, \mathscr F, \prob)$ a probability space and $X_1, \dots, X_n, \dots, X_\infty$ random variables such that $X_n \xrightarrow[n \to \infty]{} X_\infty$, $\prob$-almost surely. Then $X_n \xrightarrow[n \to \infty]{} X_\infty$ in distribution.
	\end{proposition}
	\begin{proof}
		Let ${f \in C_b(\mathcal X)}$. By continuity of $f$, $f(X_n) \xrightarrow[n \to \infty]{} f(X_\infty)$ $\prob$-a.s. Furthermore, $(f(X_n))_{n \in \N}$ is uniformly bounded by $\max f$. Hence, by dominated convergence theorem,
		\[ \int_\Omega f(X_n) \dd\prob  \xrightarrow[n \to \infty]{} \int_\Omega f(X_\infty) \dd\prob.\qedhere \]
	\end{proof}
	
	\paragraph{Portmanteau’s theorem.} The following strong version of \emph{Portmanteau’s theorem} gives other characterizations of the weak topology on $\mathscr M_1$. 

	\begin{theorem}[Portmanteau]\label{portmanteau}
		The following collections of subsets of $\mathscr M_1(\mathcal, \Borel)$ each generate $\weak$ as a subbase:
		\begin{enumhypos}
			\item $\left\{ \nu \in \mathscr M_1 : \left|\int_{\mathcal X} f \dd\nu - \int_{\mathcal X} f \dd\mu \right| < \eps \right\}$
			for $f \in C_b(\mathcal X)$ \ul{uniformly continuous}, $\mu \in \mathscr M_1$ and $\eps > 0$.\label{portmanteau:unif_cont}
			\item $\{ \nu \in \mathscr M_1~: \nu(F) < \mu(F) + \eps \}$ for $\mu \in \mathscr M_1$, $F$ closed subset of $\mathcal X$ and $\eps > 0$.
			\item $\{ \nu \in \mathscr M_1~: \nu(G) > \nu(G) - \eps \}$ for $\mu \in \mathscr M_1$, $G$ open subset of $\mathcal X$ and $\eps > 0$.
			\item $\{ \nu \in \mathscr M_1~: |\nu(A) - \mu(A)| < \eps \}$ for $\mu \in \mathscr M_1$, $A \in \Borel(\mathcal X)$ such that $\mu(\partial B) = 0$ (where $\partial B$ is the topological boundary of $B$) and $\eps > 0$. \label{portmanteau:continuity_borel}
		\end{enumhypos}
	\end{theorem}
	
	Each characterization of the topology immediately gives a characterization of the weakly convergent sequences of $\mathscr M_1$. For instance, from point~\ref{portmanteau:continuity_borel}, follows that $\mu_n \xrightarrow[n \to \infty]{\weak} \mu_\infty$ if and only if for all $A \in \Borel(\mathcal X)$ such that $\mu_\infty(\partial A) = 0$, $\mu_n(A) \xrightarrow[n \to \infty]{} \mu_\infty(A)$.
	
	For a proof, see \cite[appendix~III, theorem~3]{billingsley1968}. The book does not explicitely state point~\ref{portmanteau:unif_cont}; however, it is a free consequence of Billingsley’s proof.
	
	From~\ref{portmanteau:unif_cont}, follows this implication:\nobreakpar
	\begin{proposition}\label{L1_cv_implies_weak_cv}
		Let $(\Omega, \mathscr F, \prob)$ a probability space and $X_1, \dots, X_n, \dots, X_\infty$ random variables with values in $\mathcal X$. Let $\mu_1, \dots, \mu_\infty$ their distributions.
		
		If $\Exp[d(X_n, X_\infty)] \xrightarrow[n \to \infty]{} 0$, then ${\mu_n \xrightarrow[n \to \infty]{\weak} \mu_\infty}$.
	\end{proposition}
	\begin{proof}
		Let $f : \mathcal X \longrightarrow \R$ bounded, uniformly continuous. Let $M \eqdef \sup f$. Chose $\eps > 0$.
		
		By uniform continuity of $f$, there exists $\delta > 0$ such that for all $x, y \in \mathcal X$, \[d(x, y) < \delta \implies |f(x) - f(y)| \le \eps.\]
		
		For each $n \in \N$, let ${A_n \eqdef \{ d(X_n, X_\infty) < \delta \}}$. We have:
		\begin{align*}
			&~\bigl|\Exp[f(X_n)] - \Exp[f(X_\infty)]\bigr|\\
			\le & ~ \Exp[|f(X_n) - f(X_\infty)|]\\
			= &~\int_{A_n} |f(X_n) - f(X_\infty)|\dd\prob
				\\* &\quad +  \int_{\Omega \setminus A_n} |f(X_n) - f(X_\infty)| \\
			\le &~\eps \prob(A_n) + 2M\prob(\Omega\setminus A_n).
		\end{align*}
		
		Yet $\prob(A_n) \le 1$, and by Markov’s inequality, \[\prob(\Omega \setminus A_n) \le \frac{\Exp[d(X_n, X_\infty)]}{\delta}.\]
		
		Hence, if $N$ is chosen such that for every $n \ge \N$, $\Exp[d(X_n, X_\infty)] < \frac{\delta \eps}{2M}$, we have:
		\[ \bigl|\Exp[f(X_n)] - \Exp[f(X_\infty)]\bigr| \le 2\eps.  \]
		
		Thus, $\Exp[f(X_n)] \xrightarrow[n \to \infty]{} \Exp[f(X_\infty)]$. By Portmanteau’s theorem, this is enough to state the convergence in distribution.
	\end{proof}
	
	\paragraph{Metric inducing $\weak$.} A nice property of $\weak$ is that it is metrizable. For all $A \in \Borel(\mathcal X)$ and $\eps > 0$, let:
	\[ A^{\eps} \eqdef \{ x \in \mathcal X~: \exists a \in A, d(x,a) < \eps \} \]
	which is immediately an open set.
	
	\begin{definition}[Prokhorov metric]
		For all ${\mu, \nu \in \mathscr M_1}$, we call $d_P(\mu, \nu)$ the infimum of positive $\eps$ such that the inequalities ${\mu(A) < \nu(A^\eps) + \eps}$ and ${\nu(A) < \mu(A^\eps) + \eps}$ stand for all $A \in \Borel(\mathcal X)$.
		
		We note $d_P : \mathscr M_1(\mathcal X)^2 \longrightarrow \R_+$ the Prohkorov metric.
	\end{definition}
	
	It is known that $d_P$ is a metric, and that $\mu_n \xrightarrow[n \to \infty]{\weak} \mu_\infty$ if and only if $d_P(\mu_n, \mu_\infty)$ (see for instance \cite[p.~72, remark~(i)]{billingsley1999}). However, \ul{this is not enough to prove that $d_P$ induces the topology $\weak$}. This more difficult result is proven in \cite[appendix~III, theorem~5]{billingsley1968}\footnote{The proof has been removed in the Second Edition of the book and can only be found in the 1968 edition. \cite{billingsley1968} generalizes the result by releasing the separability hypothesis.}.
	
	\ul{Now}, we know that $\weak$ is metrizable, hence sequential.
	
	\subsubsection{The real line}
	
	Now we restrict on real numbers. For this article, we only need to consider $\mathcal X = \R_+$, but all the following results can be extended to $\mathcal X = \R$.
	
	We give a few characterization of weak convergence that will freely be used in the article.
	
	\begin{theorem}\label{characterization_weak_convergence}
		Let $\mu_1, \dots, \mu_n, \dots, \mu_\infty$ distributions on $(\R_+, \Borel)$, $F_{\mu_\bullet}$ their cdf and $Q_{\mu_\bullet}$ their quantile functions. The following are equivalent:
		\begin{enumhypos}
			\item $\mu_n \xrightarrow[n \to \infty]{\weak} \mu_\infty$;\label{characterization_weak_convergence:W}
			\item ${F_{\mu_n}(x) \xrightarrow[n \to \infty]{} F_{\mu_\infty}(x)}$ for every $x \in \R_+$ where $F_{\mu_\infty}$ is continuous;\label{characterization_weak_convergence:F_cont}
			\item ${F_{\mu_n}(x) \xrightarrow[n \to \infty]{} F_{\mu_\infty}(x)}$ for $\Leb$-almost all ${x \in \R_+}$;\label{characterization_weak_convergence:F_ae}
			\item ${Q_{\mu_n}(p) \xrightarrow[n \to \infty]{} Q_{\mu_\infty}(p)}$ for every ${p \in [0,1)}$ where $Q_{\mu_\infty}$ is continuous;\label{characterization_weak_convergence:Q_cont}
			\item ${Q_{\mu_n}(p) \xrightarrow[n \to \infty]{} Q_{\mu_\infty}(x)}$ for $\Leb$-almost all ${p \in [0,1)}$;\label{characterization_weak_convergence:Q_as}
			\item There exist a probability space $(\Omega, \mathscr F, \prob)$ and nonnegative random variables $X_1, \dots, X_n, \dots, X_\infty$ such that ${X_n \xrightarrow[n \to \infty]{} X_\infty}$ $\prob$-almost surely, $X_n \sim  \mu_n$ for every $n \in \N^*$ and $X_\infty \sim \mu_\infty$.\label{characterization_weak_convergence:X}
		\end{enumhypos}
	\end{theorem}
	
	\begin{proof}
		\begin{itemize}
			\item $\ref{characterization_weak_convergence:W} \implies \ref{characterization_weak_convergence:F_cont}$ is a direct consequence of Portmanteau's theorem  (point~\ref{portmanteau:continuity_borel} of theorem~\ref{portmanteau}), taking ${\mu = \mu_\infty}$ and $A = [0, x]$, thus ${\mu_\infty(\partial A) = \mu_\infty(\{x\}) = 0}$ by continuity of $F_{\mu_\infty}$. (Notice that $0$ is not in $\partial A$, since $[0,x)$ is open in $\R_+$.)
			
			\item Assume \ref{characterization_weak_convergence:F_cont}. Since $F_{\mu_\infty}$ is nondecreasing, the set of its discontinuity points is at most countable. Hence, it has Lebegue measure 0, i.e. \ref{characterization_weak_convergence:F_ae} stands.
			
			\item Assume \ref{characterization_weak_convergence:F_ae}. Consider the probability space $(\R_+, \Borel, \prob)$ where ${\prob(\rm d t) = e^{-t} \Leb(\rm d t)}$ ($\prob$ is the measure with density ${t \longmapsto e^{-t}}$ with respect to $\Leb$). The $F_{\mu_\bullet}$’s can be seen as random variables on $(\R_+, \Borel, \prob)$.
			
			Let $C \subseteq \R_+$ the set of $\omega \in \R_+$ such that ${F_{\mu_n}(\omega) \xrightarrow[n \to \infty]{} F_{\mu_\infty}(\omega)}$. We know that ${\Leb(C^c) = 0}$, hence ${\prob(C^c) = 0}$, thus $\prob(C) = 1$. 
			
			Consider $x \in \R_+$ such that $Q_{\mu_\infty}$ is continuous at $x$. Let $D$ the set of $\omega \in \R_+$ such that $F_{\mu_\infty}(\omega) = x$. By proposition~\ref{quantile_continuity}, $D$ is either a singleton or empty. Hence, $\Leb(L_x) = 0$, thus $\prob(L_x) = 0$.
			
			Now, for all $\omega \in C \setminus D$, we have $F_{\mu_n}(\omega) \xrightarrow[n\to\infty]{} F_{\mu_\infty}(\omega) \neq x$. Since the function $\1_{[0,x)}$ is continuous everywhere but in $x$,
			\[\1_{[0,x)} (F_{\mu_n}(\omega)) \xrightarrow[n \to \infty]{} \1_{[0,x)} (F_{\mu_\infty}(\omega)).\]
			
			Since $\prob(C\setminus D) = 1$ and $\1_{[0,x)}(F_{\mu_n})$ is bounded by 1, by dominated convergence theorem, we have:
			\begin{align*}
				\Exp\left[\1_{[0,x)}(F_{\mu_n})\right] &\xrightarrow[n \to \infty]{} \Exp\left[\1_{[0,x)}(F_{\mu_\infty})\right] \\
				\prob(F_{\mu_n} < x)  &\xrightarrow[n \to \infty]{} \prob(F_{\mu_\infty} < x).
			\end{align*}
			
			However, for $n \in \N \cup \{\infty\}$, $\omega \in \R_+$, $F_{\mu_n}(\omega) < x \iff \omega < Q_{\mu_n}(x)$. Hence,
			\[ \prob\bigl([0, Q_{\mu_n}(x))\bigr) \xrightarrow[n \to \infty]{} \prob\bigl([0, Q_{\mu_\infty}(x))\bigr) \]

			The function $t \longmapsto \prob([0, t)) = 1-e^{-t}$ has inverse ${p \longmapsto -\ln(1-p)}$, which is continuous over $\R_+$. Hence,
			\[ Q_{\mu_n(x)} \xrightarrow[n \to \infty]{} Q_{\mu_\infty}(x), \] i.e. \ref{characterization_weak_convergence:Q_cont} is true.

			\item $\ref{characterization_weak_convergence:Q_cont} \implies \ref{characterization_weak_convergence:Q_as}$ is proven exactly the same way as $\ref{characterization_weak_convergence:F_cont} \implies \ref{characterization_weak_convergence:F_ae}$.
			
			\item $\ref{characterization_weak_convergence:Q_as} \implies \ref{characterization_weak_convergence:X}$ is immediate since for every $n \in \N^* \cup \{\infty\}$, $Q_{\mu_n}$ is a random variable on $\bigl([0,1), \Borel, \Leb\bigr)$ with distribution $\mu_n$ (see lemma~\ref{characterization_by_quantile}).
			
			\item $\ref{characterization_weak_convergence:X} \implies \ref{characterization_weak_convergence:W}$ is a direct consequence of proposition~\ref{as_cv_implies_weak_cv}.\qedhere
		\end{itemize}
	\end{proof}
	
	\subsection{$\Wone$ metric and $\Wone$ convergence}
	\label{appendix_W1_cv}
	\subsubsection{The $\Wone$ metric}
	For this subsection, we work on the set $\M'$ of probability measures on $(\R_+, \Borel)$ with finite expectations, i.e. $\M' = \M \cup \{\delta_0\}$. In the article, the following results will be restricted to $\M$, since we exclude the distribution $\delta_0$.
	\begin{definition}[Wasserstein-1 metric]
		Let $\mu, \nu \in \M'$. We let $\Wone(\mu, \nu) \eqdef \left\|Q_\mu - Q_\nu\right\|_1$, i.e.
		\[ \Wone(\mu, \nu) = \int_0^1 |Q_\mu(t) - Q_\nu(t)| \dd t. \]
	\end{definition}
	
	It is immediate that $\Wone$ is a pseudometric. Now 	assume $\Wone(\mu, \nu) = 0$. This implies that $Q_\mu -Q_\nu = 0$ almost everywhere. However, $Q_\mu - Q_\nu$ is left-continuous. Hence $Q_\mu - Q_\nu = 0$ everywhere. By proposition~\ref{converse_quantile}, $\mu = \nu$. Hence, $\Wone$ is a metric.
	
	The name of \emph{Wassertein-$p$ metric} more usualy refers to the optimal cost of transportation for $\mathrm L^p$ cost in the Kantorovitch optimal transportation problem framework: for $(\mathcal X, d)$ a Polish space and $\mu$, $\nu$ measures on $(\mathcal X, \Borel)$, $\mathrm W_p(\mu, \nu)$ is equal to:
	\[ \min_{\pi \in \Pi(\mu, \nu)} \left(\int_{(x,y) \in \mathscr X^2} d(x, y)^p \dd\pi(x,y)\right)^{\frac 1 p}  \]
	where $\Pi(\mu,\nu)$ is the set of distributions on $\mathcal X^2$ with marginals $\mu$ and $\nu$. 
	Happily those definitions are consistant. For a proof that for the minimum has value $\left(\int_0^1 |Q_\mu(p) - Q_\nu(p)|^p \dd p\right)^{\frac 1 p}$, see for instance \cite[theorem~2.18 and remarks~2.19]{villani2003}.

	Another way to express the $\Wone$ metric is the following:
	\begin{proposition} Let $\mu, \nu \in \M'$, $F_\mu$ and $F_\nu$ their cdf. Then $\Wone(\mu, \nu) = \left\|F_\mu - F_\nu\right\|_1$.
	\end{proposition}
	
	\begin{proof} We have:
	\begin{align*}
		& \int_0^\infty |F_\mu(x) - F_\nu(x)| \dd x\\
		=& \int_{x=0}^\infty \left(\int_{y=F_\mu(x)}^{F_\nu(x)} 1 \dd y + \int_{y = F_\nu(x)}^{F_\mu(x)} 1 \dd y\right) \dd x\\
		=& \int_{x=0}^\infty \left(\int_{y=F_\mu(x)^+}^{F_\nu(x)} 1 \dd y + \int_{y = F_\nu(x)^+}^{F_\mu(x)} 1 \dd y\right) \dd x\\
		=& \int_{x=0}^\infty \int_{y=0}^{1^-} \left(\begin{array}{rl} & \displaystyle \1_{F_\mu(x) < y \le F_\nu(x)} \\ + & \displaystyle \1_{F_\nu(x) < y \le F_\mu(x)}\end{array}\right) \dd y \dd x\\
		=& \int_{x=0}^\infty \int_{y=0}^{1^-} \left(\begin{array}{rl} & \displaystyle \1_{Q_\nu(y) \le x < Q_\mu(x)} \\ + & \displaystyle \1_{Q_\mu(y) \le x < Q_\nu(x)}\end{array}\right) \dd y \dd x\\
		=& \int_{y=0}^{1^-} \int_{x=0}^\infty \left(\begin{array}{rl} & \displaystyle \1_{Q_\nu(y) \le x < Q_\mu(x)} \\ + & \displaystyle \1_{Q_\mu(y) \le x < Q_\nu(x)}\end{array}\right) \dd x \dd y\\
		=& \int_0^{1^-} |Q_\mu(y) - Q_\nu(y)| \dd y.
	\end{align*}
	using, in order, the nonatomicity of Lebesgue measure, the Galois inequalities, Fubini’s theorem, and making the same computations backwards.
	\end{proof}
	
	\subsubsection{The $\Wone$ convergence and its characterization}
	
	By definition, the $\Wone$ convergence of a sequence $(\mu_n)_{n \in \N}$ to a limit $\mu_\infty$ is equivalent to the $\Lone$ convergence of the random variables $(Q_{\mu_n})_{n \in \N}$ to $Q_{\mu_\infty}$ in the probability space $\bigl([0,1), \Borel, \Leb]\bigr)$. We give some other characterizations. 
	
	\begin{lemma}[Scheffé]
		Let $(\mathcal X, \mathscr F, \mu)$ a probability space and $f_1, \dots, f_n, \dots, f_\infty$ measurable functions ${\mathcal X \longrightarrow \R_+}$.
		
		If $f_n \xrightarrow[n \to \infty]{} f_\infty$ $\mu$-almost everywhere and $\int_{\mathcal X} f_n \dd\mu \xrightarrow[n \to \infty]{} \int_{\mathcal X} f_\infty \dd\mu < \infty$, then:
		\[ \int_{\mathcal X} |f_n - f_\infty| \dd\mu \xrightarrow[n \to \infty]{} 0. \]
	\end{lemma}
	
	\begin{proof}
		For $g : \mathcal X \longrightarrow \R$, call $g_+ = \max(g, 0)$ and $g_- = \max(-g, 0)$. We have ${g = g_+ - g_-}$ and ${|g| = g_+ + g_-}$. Hence, $|g| = 2 g_+ - g$. It follows that:
		\begin{align}
			\|g\|_1 = \int_{\mathcal X} |g| \dd\mu = 2\int_{\mathcal X} g_+ \dd\mu- \int_{\mathcal X} g \dd\mu.\label{scheffe_proof:eq1}
		\end{align}
		
		Since $f_n \xrightarrow[n \to \infty]{} f_\infty$ $\mu$-a.e. and $x \longmapsto x_+$ is continuous, it follows that $(f_\infty - f_n)_+ \xrightarrow[n \to \infty]{} 0$, $\mu$-a.e. Furthermore,
		\[ 0 \le (f_\infty-f_n)_+ \le f_\infty. \]
		
		Since $f_\infty$ is integrable, by dominated convergence theorem,
		\[ \int_{\mathcal X} (f_\infty- f_n)_+ \dd\mu \xrightarrow[n \to \infty]{} 0. \]
		
		Furthermore, by hypothesis:
		\[ \int_{\mathcal X} (f_\infty - f_n) \dd\mu \xrightarrow[n \to \infty]{} 0.\]
		
		Injecting this in \eqref{scheffe_proof:eq1} with $g = f_\infty - f_n$, we get $\|f_\infty-f_n\|_{1} \xrightarrow[n \to \infty]{} 0.$
	\end{proof}

	Scheffé’s lemma gives us an important characterization of $\Wone$-convergence:\nobreakpar
	\begin{proposition} \label{scheffe_measures}
		Let $(\mu_n)_{n \in \N} \in \M^\N$, ${\mu_\infty \in \M}$ and $m_{\mu_1}, \dotsc,\allowbreak m_{\mu_\infty}$ their means. The following assertions are equivalent:\nobreakpar
		\begin{enumhypos}
			\item $\mu_n \xrightarrow[n \to \infty]{\Wone} \mu_\infty$.\nobreakpar\label{scheffe_measures:W1}
			\item $\mu_n \xrightarrow[n \to \infty]{\weak} \mu_\infty$ \ul{and} $m_{\mu_n} \xrightarrow[n \to \infty]{} m_{\mu_\infty}$.\nobreakpar\label{scheffe_measures:weak}
			\item There exists a probability space $(\Omega, \mathscr F, \prob)$ and nonegative random variables $X_1, \dots, X_n, \dots, X_\infty$ such that ${X_n \sim \mu_n}$ for every ${n \in \N \cup \{\infty\}}$ and ${\Exp[|X_n - X_\infty|] \xrightarrow[n \to \infty]{} 0}$.\label{scheffe_measures:X}
		\end{enumhypos}
	\end{proposition}
	
	\begin{proof} \begin{itemize}
			\item $\ref{scheffe_measures:W1} \implies \ref{scheffe_measures:X}$~: take $\Omega = [0,1)$, $\mathscr F = \Borel$ and $\prob = \Leb$ and for all $x \in \N \cup \{\infty\}$, let $X_n = Q_{\mu_n}$.
			\item Assume $\ref{scheffe_measures:X}$. By triangle inequality, $m_{\mu_n} = \Exp[X_n] \xrightarrow[n \to \infty]{} \Exp[X_\infty] = m_{\mu_\infty}$. Furthermore, since $\Lone$ convergence implies weak convergence (proposition~\ref{L1_cv_implies_weak_cv}), $\mu_n \xrightarrow[n \to\infty]{\weak} \mu_\infty$. Hence,~$\ref{scheffe_measures:weak}$ stands.
			\item $\ref{scheffe_measures:weak} \implies \ref{scheffe_measures:W1}$ lies in applying Scheffé’s lemma to the $Q_{\mu_n}$’s.\qedhere
		\end{itemize}
	\end{proof}

	\subsection{Uniform integrability}
	\label{appendix_unif_int}
	\subsubsection{Definition}
	
	Recall the definition of uniform integrability of random variables. We restrict ourselves to nonnegative-valued random variables. In order to simplify the following redaction, we also consider “uniformly integrable measures”, i.e. measures which are the distributions of a uniformly integrable collection of random variables --- this concept is nonstandard. 
	
	\begin{definition}[u.i. for random vars]
		Let $(\Omega_i, \mathscr F_i, \prob_i)_{i \in I}$ a collection of probability spaces and for all $i \in I$, $U_i : \Omega_i \longrightarrow \R$ a random variable.
		
		The collection $(U_i)_{i \in I}$ is said \emph{uniformly integrable (u.i.)} if
		\[ \sup_{i \in I} \int_{\Omega_i} |U_i|~ \1_{|U_i| > \alpha} \dd\prob_i \xrightarrow[\alpha \to +\infty]{} 0. \]
	\end{definition}
	
	\begin{definition}[u.i. for measures]\label{def_ui}
		Let $I$ a set and $(\mu_i)_{i \in I} \in \mathscr M_1(\R_+, \Borel)^I$. We say that $(\mu_i)_{i \in I}$ is \emph{uniformly integrable} if
		\[ \sup_{i \in I} \int_{\R_+} x~\1_{x > \alpha} \dd\mu_i(x) \xrightarrow[\alpha \to +\infty]{} 0. \]
	\end{definition}
	
	These definitions are coherent, thanks to the following proposition:
	\begin{proposition}
		Let $I$ a set and ${(\mu_i)_{i \in I} \in \mathscr M_1(\R_+, \Borel)^I}$. The following are equivalent:
		\begin{enumhypos}
			\item $(\mu_i)_{i \in I}$ is uniformly integrable.
			\item There exists probability spaces $(\Omega_i, \mathscr F_i, \prob_i)_{i \in I}$ and random variables $U_i : \Omega_i \longrightarrow \R_+$ such that for all $i \in I$, $U_i \sim \mu_i$ and  $(U_i)_{i \in I}$ is uniformly integrable.
			\item For all probability spaces $(\Omega_i, \mathscr F_i, \prob_i)_{i \in I}$ and random variables $U_i : \Omega_i \longrightarrow \R_+$ such that $U_i \sim \mu_i$,  $(U_i)_{i \in I}$ is uniformly integrable.
		\end{enumhypos}
	\end{proposition}
	\begin{proof}
		Immediate consequence of LOTUS and the fact that $Q_\mu \sim \mu$.
	\end{proof}
	
	In particular, a collection $(\mu_i)_{i \in I}$ of probability measures is u.i. if, and only if, the collection of random variables $\left(Q_{\mu_i}\right)_{i \in I}$ defined on the probability space $\bigl([0,1), \Borel, \Leb\bigr)$ is u.i.
	
	We will freely use the following technical lemma.
	
	\begin{lemma}\label{lemma_charac_ui_subsets}
			Let $(U_i)_{i \in I}$ real, integrable random variables on probability spaces $(\Omega_i, \mathscr F_i, \prob_i)$. $(U_i)_{i \in I}$ is uniformly integrable if and only if for all $\eps > 0$, there exists $\alpha_\eps > 0$ and $J_\eps > 0$ such that $I \setminus J_\eps$ is finite and
			\begin{align}
				\sup_{i \in J_\eps} \Exp_i\bigl[|U_i|~\1_{|U_i| > \alpha_\eps} \bigr] < \eps. \label{characterization_ui_subsets}
			\end{align}
	\end{lemma}
	\begin{proof}
		Assume $(U_i)_{i \in I}$ is u.i. Let $\eps > 0$. By definition, there exists $\alpha > 0$ such that:
		\[ \sup_{i \in I} \Exp_i\bigl[|U_i| ~ \1_{|U_i| > \alpha_\eps} \bigr] < \alpha \]
		and \eqref{characterization_ui_subsets} holds for $\beta_\eps = \alpha$ and $J_\eps \eqdef I$.
		
		Conversely, assume that for all $\eps > 0$, there exists $J_\eps$ and $\beta_\eps$ such that $I \subseteq J_\eps$ is finite and \eqref{characterization_ui_subsets} holds.
		
		Let $j \in I\setminus J_\eps$. Since $|U_j|$ is integrable, by dominated convergence theorem,
		\[ \Exp\bigl[|U_j| \1_{|U_j| > \alpha}\bigr] \xrightarrow[\alpha \to +\infty]{} 0 \]
		hence there exists $\alpha_j$ such that
		\[ \Exp\bigl[|U_j| \1_{|U_j| > \alpha_j}\bigr] < \eps. \]
		
		Now if we take $A = \max\left(\beta_\eps, \max_{j \in I\setminus J_\eps} \alpha_j\right)$, for all $\alpha > A$ and $i \in I$:
		\begin{itemize}
			\item either $i \in J_\eps$ and
			\begin{align*}
				&\Exp\bigl[|U_i|~\1_{|U_i| > \alpha}\bigr] \\
				< & \sup_{i \in J_\eps} \Exp_i\bigl[|U_i|~\1_{|U_i| > \alpha_\eps} \bigr] < \eps;
			\end{align*}
			\item or $i \in I \setminus J \eps$, and
			\[ \Exp\bigl[|U_i|~\1_{|U_i| > \alpha}\bigr] < \Exp\bigl[|U_j| \1_{|U_j| > \alpha_j}\bigr] < \eps. \]
		\end{itemize}
		Hence, $(U_i)_{i \in I}$ is u.i.
	\end{proof}
	
	One can directly check that if a collection $(U_i)_{i \in I}$ of real random variables is u.i., then for all $i \in I$, $\Exp_i[|U_i|] < \infty$. Hence, only need to deal with integrable random variables and measures in $\M$. In fact, uniform integrability even implies that $\Exp_i[|U_i|]$ are \emph{uniformly} bounded. Indeed, take $\alpha > 0$ such that $\sup_{i \in I} \Exp\bigl[|U_i| \1_{|U_i| > \alpha}\bigr] < 1$; we have for all $i \in I$, $\Exp[|U_i|] < \alpha \prob_i(\{ |U_i| \le \alpha \}) + 1 \le \alpha+1$. The converse is false.
	
	\subsubsection{Some sufficient conditions for having uniform integrability}
	
	\paragraph{Uniform $\mathrm L^p$ bound.} On the other hand, for $p > 1$, if there exists $M > 0$ such that $\Exp_i[|X_i|^p] < M$ for all $i \in I$, then $(X_i)_{i \in I}$ is u.i, the converse being false. Indeed,
	\[\1_{(|X_i| > \alpha)} |X_i| = \1_{\bigl(|X_i|^{p-1} > \alpha^{p-1}\bigr)} |X_i| \le \frac{|X_i|^p}{\alpha^{p-1}}\]
	hence \[\Exp_i\bigl[\1_{(|X_i| > \alpha)} |X_i|\bigr] \le \frac{M}{\alpha^{p-1}} \xrightarrow[n \to \infty]{} 0.\]
	
	\paragraph{$\Lone$ singleton.} If $X$ is an integrable random variable on any probability space, then the collection $\{X\}$ is uniformly integrable. Indeed, $|X| \1_{|X| > n} \xrightarrow[n \to +\infty]{} 0$ pointwise and those functions are dominated by $|X|$, which is integrable. Hence, the dominated convergence theorem states that \[\Exp\left[\1_{X > n} X\right] \xrightarrow[\alpha \to \infty]{} \Exp[0] = 0.\]
	
	\subsubsection{Operations on uniformly integrable collections}
	\label{operations_ui}
	
	\paragraph{Multiplication by bounded scalars.} Assume $(X_i)_{i \in I}$ is u.i. and $(a_j)_{j \in I}$ is bounded. Immediately, $(a_j X_i)_{i \in I, j \in J}$ is u.i.
	
	\paragraph{Sum.} Let $X_i, i \in I$ and $Y_j, j \in J$ real random variables defined on the same probability space $(\Omega, \mathscr F, \prob)$. Assume $(X_i)_{i \in I}$ is u.i. and $(Y_j)_{j \in J}$ is u.i. Then $(X_i + Y_j)_{i \in I, j \in J}$ is u.i.
	
	Indeed, for all $(a, b) \in \R$ and $c \in \R_+$, the following statement holds:
	\[ |a+b| \1_{(|a+b| > 2c)} \le 2|a| \1_{(|a| > c)} + 2|b| \1_{(|b| > c)} \]
	(this is trivial if $|a + b| \le 2c$; otherwise one can directly check it in cases $|a| \ge |b|$ and $|a| \le |b|$).
		
	Thus:
	\begin{align*}
		&~ \Exp\bigl[|X_i + Y_j| \1_{|X_i + Y_j| > 2\alpha}\bigr]\\
		\le &~ 2\Exp\bigl[|X_i|\1_{|X_i| > \alpha}\bigr] + 2 \Exp\bigl[|Y_j| \1_{|Y_j| > \alpha}\bigr]\\
		\le &~ 4\eps
	\end{align*}
	for $\alpha$ large enough.
	
	\begin{proposition}[u.i. and FSD]\label{unif_int_FSD}
		Let $(\mu_i)_{i \in I}$ and $(\nu_j)_{j \in J}$ collections of measures in $\M$. Assume that:
		\begin{enumhypos}
			\item The collection $(\mu_i)_{i \in I}$ is uniformly integrable;
			\item For all $j \in J$, there exists a $\iota(j) \in I$ such that $\nu_j$ is stochastically dominated at first order by $\mu_{\iota(j)}$.
		\end{enumhypos}
		
		Then $(\nu_j)_{j \in J}$ is uniformly integrable.
	\end{proposition}

	\begin{proof}
		Consider the probability space $\bigl([0,1), \Borel, \Leb\bigr)$. For every $i \in I$, $j \in J$, let the random variables $A_i = Q_{\mu_i}$ and $B_j = Q_{\nu_j}$.
		
		The proposition~\ref{definitions_FSD} and the hypothesis of stochastic dominance ensure that for all $j \in J$, $B_{j} \le A_{\iota(j)}$.
		
		Fix $\alpha \in \R_+$ and $j \in J$. $\Leb$-almost surely,
		\begin{align*}
			\1_{\left(B_j > \alpha\right)} & \le \1_{\left(A_{\iota(j)} > \alpha\right)}\\
			B_j \1_{\left(B_j > \alpha\right)} & \le A_{\iota(j)}\1_{\left(A_{\iota(j)} > \alpha\right)}.
		\end{align*}
		
		Hence,
		\[\Exp\left[B_j \1_{\left(B_j > \alpha\right)}\right] \le \Exp\left[ A_{\iota(j)}\1_{\left(A_{\iota(j)} > \alpha\right)}\right]. \]
		
		Taking the lowest upper bound over $j \in J$, we get:
		\[ \sup_{j \in J}\Exp\left[B_j \1_{\left(B_j > \alpha\right)}\right] \le \sup_{i \in I} \Exp\left[ A_{i}\1_{\left(A_{i} > \alpha\right)}\right] \]
		
		The hypothesis of uniform integrability of collection $(\mu_i)_{i \in I}$ ensures that the right-hand side converges to 0 as $\alpha \to \infty$.
	\end{proof}
	
	\subsubsection{Uniform integrability and $\Wone$ convergence}

	The main result that justifies to use uniform integrability in this article is the following, which is a measure-based adaptation of \cite[theorem~16.14]{billingsley1995}.
	\begin{proposition}\label{weak_vitali}
		Let $(\mu_n)_{n \in \N} \in \M^\N$ and $\mu_\infty$ in $\M$. Assume that $\mu_n \xrightarrow[n \to \infty]{\weak} \mu_\infty$. The following assertions are equivalent:\nobreakpar
		\begin{enumhypos}
			\item $m_{\mu_n} \xrightarrow[n \to \infty]{} m_{\mu_\infty}.$\label{weak_vitali:1}
			\item $(\mu_n)_{n \in \N}$ is uniformly integrable.\label{weak_vitali:2}
		\end{enumhypos}
	\end{proposition}
	
	\begin{proof} Consider the probability space $\bigl([0,1), \Borel, \Leb\bigr)$. For every $n \in \N$, let $Q_n \eqdef Q_{\mu_n}$ and $Q \eqdef Q_{\mu_\infty}$.
		
		For $\alpha > 0$, let $f_\alpha : x \longmapsto x \1_{x \le \alpha}$ and $g_\alpha : g \longmapsto x \1_{x > \alpha}$.
		
		For every $n \in \N$, by lemma~\ref{lemma_mean_quantile},
		\begin{align}\label{weak_vitali:proof:m_f_g}
			m_{\mu_n} & = \Exp[f_\alpha(Q_n)] + \Exp[g_\alpha(Q_n)]~; \nonumber \\
			m_{\mu_\infty} & = \Exp[f_\alpha(Q)] +  \Exp[g_\alpha(Q)].
		\end{align}
		
		Furthermore, by proposition~\ref{characterization_by_quantile}, ${Q_n \xrightarrow[n\to\infty]{} Q}$ $\Leb$-a.s. Hence, if $\mu(\{\alpha\}) = 0$, then:
		\[ f_\alpha(Q_n) \xrightarrow[n \to \infty]{} f_\alpha(Q)~; \quad g_\alpha(Q_n) \xrightarrow[n \to \infty]{} g_\alpha(Q). \]
		
		Since $\Exp[f_\alpha(Q_n)] \le \alpha$ for every $n \in \N$, by dominated convergence theorem,
		\begin{align}
			\Exp[f_\alpha(Q_n)] \xrightarrow[n \to \infty]{} \Exp[f_\alpha(Q)].\label{weak_vitali:proof:cv_f}
		\end{align}
		
		Now we can prove both implications.
		
		\paragraph{$\ref{weak_vitali:1} \implies \ref{weak_vitali:2}.$} For every $n\in \N$, by \eqref{weak_vitali:proof:m_f_g}:
		\begin{align*}
			\Exp[g_\alpha(Q_n)] \le& ~\Exp[g_\alpha(Q)] + | m_{\mu_\infty} - m_{\mu_n} |
				\\*	&\quad + \bigl|\Exp[f_\alpha(Q_n)] - \Exp[f_\alpha(Q)] \bigr|.
		\end{align*}
		
		Fix $\eps > 0$. Since $\{Q\}$ is u.i., there exists $\alpha_1 > 0$ such that for all $\alpha > \alpha_1$, $\Exp[g_{\alpha}(Q)] < 0$. Since $\mu$ has at most countably many atoms, fix $\alpha_2 > \alpha_1$ such that $\mu(\{\alpha_2\}) = 0$. By \eqref{weak_vitali:proof:cv_f} applied to $\alpha \eqdef \alpha_2$ and $\ref{weak_vitali:1}$, there exists $N \in \N$ such that for every $n \ge N$,
		\[ \Exp[g_\alpha(Q_n)] < 3\eps. \]
		
		Hence, by lemma~\ref{lemma_charac_ui_subsets}, \ref{weak_vitali:2} holds.
		
		\paragraph{$\ref{weak_vitali:2} \implies \ref{weak_vitali:1}.$} For every $n \in \N$, by \eqref{weak_vitali:proof:m_f_g}:
		\begin{align*}
			|m_{\mu_\infty} - m_{\mu_n}| \le &~ \bigl|\Exp[f_\alpha(Q_n)] - \Exp[f_\alpha(Q)]\bigr|
				\\* & \quad + \Exp[g_\alpha(Q_n)] + \Exp[g_\alpha(Q)].
		\end{align*}
		
		By the same reasoning, for all $\eps > 0$, there exists some $\alpha_1$ such that for all $\alpha > \alpha_1$, $\Exp[g_\alpha(Q)] < \eps$. Assume \ref{weak_vitali:2}; there exists $\alpha_2$ such that for every $n \in \N$ and $\alpha >\alpha_2$, $\Exp[g_\alpha(Q_n)] < \eps$.
		
		Hence, take $\alpha_3 > \max(\alpha_1, \alpha_2)$ such that $\mu(\{\alpha_2\}) = 0$. There exists $N \in \N$ such that for every $n \ge N$, $\bigl|\Exp\left[f_{\alpha_3}(Q_n)\right] - \Exp\left[f_{\alpha_3}(Q)\right]\bigr| < \eps$. Finally,
		\[ |m_{\mu_n} - m_{\mu_\infty}| < 3\eps. \]
		
		Thus, \ref{weak_vitali:1} holds.
	\end{proof}
\end{multicols}

\begin{multicols}{2}[\section{Motivation of the alternate definitions of the Gini index (or: Why Is The Section~\ref{section_alternate_def_gini} Useful?)}\label{appendix_need_gini_hoover}]
	The Gini inequality index is one of the best-known mathematical objects used to measure economic inequalities. This index was first approached by Corrado Gini in his 1912 book \emph{Variabilità e Mutabilità} \cite{variabilita1912}\footnote{See \cite{ceriani2011} for a partial translation and comment in English.}. In this paper, Gini considers $n$ nonnegative, not all zero quantities $(a_i)_{1 \le i \le n}$ (for instance, the incomes of $n$ individuals) and gives several formula to express the \emph{mean difference between the $n$ quantities}, i.e. the value
	\[ \frac{1}{n^2} \sum_{1 \le i, j \le n} |a_i - a_j|. \]
	
	Normalizing this mean difference of a nonnegative series by its arithmetic mean and by a factor of two gives a number between 0 (included) and 1 (excluded). This value is precisely the \emph{Gini index} as defined, without ambiguity, by economists and sociologists:
	\[ G((a_i)_{1 \le i \le n}) = \frac{1}{n} \frac{\sum_{1 \le i, j \le n} |a_i - a_j|}{2\sum_{1 \le i \le n} a_i}.\]
	The value $0$ is reached if and only if the series $(a_i)_i$ is constant, and $1$ is the limit of more and more \emph{concentrated} distributions.
	
	If $A$ is a random variable which follows the empirical distribution associated to\footnote{I.e. if the probability of the event $(A = x)$ is the number of occurrences of $x$ in series $(a_i)_{1 \le i \le n}$, divided by $n$.} the sample $(a_i)_{1 \le i \le n}$, then:
	\[ G((a_i)_i) = \frac{\Exp[|A-A'|]}{2 \Exp[A]}, \]
	where $\Exp$ denotes the expectation, and $A'$ is an independent copy of $A$.
	
	Modelization leads us to consider the case of variables with densities. For instance, the income of households could be, very roughly, modelized as a random variable following a gamma or a lognormal distribution\footnote{See \cite{chotikapanich2008} for a collection of articles about probability distributions for modelling household incomes.}. Thus, one can extend the definition of $G$, and set:
	\begin{align} G(A) &= \frac{\Exp[|A-A'|]}{2 \Exp[A]} \label{gini_firstdef} \end{align}
	for any nonnegative random variable for which this quantity is defined — the latter is equivalent to having $0 < \Exp[A] < \infty$. Notice that $G$ does only depend of the distribution of $A$; thus we can write $G(\mu)$ for a probability measure $\mu$.
	
	\centerstar 
	
	It is well known that the Gini index has a strong relation with the \emph{Lorenz curve}. This one is, roughly, defined as follows: “$L(u)$ is the share of the total income earned by the $100u~\%$ of the total population” --- this definition is challenged in section~\ref{section_alt_lorenz}.
	
	Gini himself noticed in \cite{gini1914}\footnote{See \cite{gini1914translated} for a translation in English. The relation between the Gini index (noted $\Delta$) and the area between the Lorenz curve and the diagonal (noted $R$) is derived in \cite[section~9, pp.~27--29]{gini1914translated}.} that the inequality index, for a sample $(a_i)_i$, is equal (up to a multiplicative factor of two) to the area between the Lorenz curve associated to this sample and the one corresponding to a situation of strict equality, i.e. the diagonal line~$y = x$. In other words, one has:
	\begin{align}G &= 1 - \frac12 \int_0^1 L(u) \dd u. \label{gini_lorenzdef} \end{align}
	
	However, Gini did only define this index for discrete distributions. Hence, he proved that $\eqref{gini_firstdef} = \eqref{gini_lorenzdef}$ only whene $A$ is discrete. Some proofs were early proposed when $A$ has a density (understand: with respect to the Lebesgue measure), for instance \cite[section~2.33]{kendall1945}.
	
	From this and on, it seems that most authors would use freely the result that $\eqref{gini_firstdef} = \eqref{gini_lorenzdef}$. For instance, \cite{atkinson1970} uses \eqref{gini_lorenzdef} with continuous variables as a definition; \cite{newbery1970} uses \eqref{gini_lorenzdef} as a definition for every kinds of variables; \cite{sheshinski1972} and \cite{dasgupta1973} use definition \eqref{gini_firstdef} with discrete variables; \cite{sen1973} and \cite{fei1978} use both definitions in the discrete case. \cite{gastwirth1972} states that the equality holds as soon as $F$ is increasing on its support and cites \cite{kendall1945} for a proof, but this latter proof does only stand if $F$ has a density. Furthermore, while most of these authors only deal with discrete variables, the definition of the Lorenz curve they use does not correctly deal with the case of atoms, as shown in lemma \ref{relation_lambda_l}.
	
	Hence, two issues appear when using the equality $\eqref{gini_firstdef} = \eqref{gini_lorenzdef}$:\nobreakpar
	\begin{enumerate}
		\item Limiting to either distributions with density with respect to Lebesgue measure or empirical distributions of samples seems artificially restrictive. One may want to consider more complex distributions, such as mixture models of a sample and a continuous random variable. For instance, it is sensible to model the distribution of the gross income as a mixture of a Dirac in $0$ (people with no income) and a random variable with a density such as a gamma or a lognormal.
		\item Having two completely different proofs for the same equality depending if we use discrete or continuous distributions is somehow unsatisfying.
	\end{enumerate}
	
	The most general proof of ${\eqref{gini_firstdef} = \eqref{gini_lorenzdef}}$ we found is \cite{dorfman1979}. In this paper, Dorfman proves, in substance, that both quantities $\Exp[|X-X'|]$ and $\int_0^1 L(u) \dd u$ are related to $\int_0^\infty (1-F(t))^2 \dd t$, where $F$ is the c.d.f. of the considered variable; he then proposes to consider
	\begin{align}
		G &= 1 - \frac{1}{\Exp[X]} \int_0^\infty (1-F(t))^2 \dd t \label{gini_dorfmandef}
	\end{align}
	as a formula for the Gini coefficient, either for discrete and for continuous variables. Dorfman first defines $G$ with equation \eqref{gini_lorenzdef} and derives formula \eqref{gini_dorfmandef} — which is a rather complex proof. Then, at the end of the article, a bit \emph{out of the blue}, he reminds that $G$ can also be defined through \eqref{gini_firstdef} and derives \eqref{gini_dorfmandef} with a really straightforward proof. Such a redacting is somehow surprising, and tends to suggest that the author assumes the equality $\eqref{gini_firstdef} = \eqref{gini_lorenzdef}$ was already proven in \emph{every case}. Furthermore, Dorfman needs to keep a few restrictions\footnote{The set of its discontinuities needs have no point of accumulation, and $F$ needs be differentiable between its discontinuities. It is true that most \emph{nice} probability distributions used in economics match these constraints; however there is no \emph{a priori} reason to be restrictive.} on $F$.
	
	\centerstar
	
	Hence, we propose two proofs of ${\eqref{gini_firstdef} = \eqref{gini_lorenzdef}}$ that deal with the widest possible spectrum of nonnegative distributions, i.e. every distributions on $\R_+$ with finite, nonzero mean.
	
	To achieve the first proof, we begin with deriving $\eqref{gini_firstdef} = \eqref{gini_dorfmandef}$ using the straightforward argument by Dorfman. Then, we propose a different, more straightforward derivation of $\eqref{gini_dorfmandef} = \eqref{gini_lorenzdef}$.
	
	We also propose a second proof of the equality, more technical, based on an idea of David Leturcq.
	
	In order to carry this out, we need to use a convenient definition of Lorentz curve, proposed by \cite{gastwirth1971}.
\end{multicols}

\begin{multicols}{2}[\section{Alternate proofs of some results}]
	\subsection{Two alternate proofs of the relation between Gini index and Lorenz curve}
	\subsubsection{Direct computation dealing with atoms}\label{appendix_proof_gini_lorenz_atomic_complete}
	\begin{proof}
		We adapt the proof of theorem~\ref{theorem_alternate_def_gini} presented in section~\ref{second_proof_equivalence_gini_simplified} in the case where $\mu$ may have atoms. It becomes necessary to make some “splits” to isolate the effects of the atoms, that will, at the end, sum up to 0.
		
		\paragraph{First split.}
		Let $I \eqdef m\int_0^1 L(p)\dd p$. As $\Leb$ is diffuse, we have:
		\[I = m\int_0^{1^-} L(p) \dd p.\]
		
		Let $\Lambda$ the pseudo-Lorenz function of $\mu$. Using lemma \ref{relation_lambda_l}, we write $I = J - K$, with:
		\begin{align*}
			J &\eqdef \int_0^{1^-} m\cdot \Lambda(p)\dd p~;\\
			K &\eqdef \int_0^{1^-} Q(p)\cdot [F(Q(p)) - p]\dd p.
		\end{align*}
		
		~
		
		We first deal with the integral $K$. As $F$ is nondecreasing, the set of its discontinuities is at most countable. Let $\left\{x_i: i \in \mathscr A\right\}$ an enumeration of this set. We note for all $i \in \mathscr A$, $\ell_i \eqdef F(x_i^-)$ and $r_i \eqdef F\left(x_i\right)$. Then:
		\[ [0,1) = F\langle\R_+\rangle~\sqcup~\coprod_{i \in \mathscr A} (\ell_i, r_i)  \]
		where $\sqcup$ and $\coprod$ denote disjoint unions.
		
		Recall that if $p \in F\langle\R_+\rangle$, then $p=F(Q(p))$. Hence, using the $\sigma$-additivity of the integral and the elementary properties of $F$ and $Q$, one can write:
		\begin{align*}
			K &= \sum_{i\in \mathscr A} \int_{\ell_i}^{r_i} Q(p)\cdot [F(Q(p)) - p]\dd p\\
			&= \sum_{i\in \mathscr A} \int_{\ell_i}^{r_i} x_i\cdot (r_i - p)\dd p\\
			&= \sum_{i\in \mathscr A} x_i\cdot \frac{(r_i - \ell_i)^2}{2}\\
			K &= \sum_{i\in \mathscr A} \frac{x_i \cdot \mu(\{x_i\})^2}{2}.
		\end{align*}
		
		~
		
		\paragraph{Second split.} Now, let us consider the integral 
		\[ J = \int_{p=0}^{1^-} \int_{u=0}^{Q(p)} u \dd \mu(u) \dd p. \]
		
		We first isolate the upper bound of the inner integral, and we write $J = J_1 + J_2$, where:
		\begin{align*}
			J_1 &\eqdef \int_0^{1^-} \int_{u=0}^{Q(p)^-} u \dd\mu(u) \dd p~;\\
			J_2 &\eqdef \int_0^{1^-} Q(p) \cdot \mu(\{Q(p)\}) \dd p.
		\end{align*}
		
		The value of $\mu(\{Q(p)\})$ is nonzero if and only if $p$ belongs to one of the semi-open intervals $(\ell_i, r_i]$. If so, $Q(p) = x_i$. Thus, we have:
		\begin{align*}
		J_2
		=& \sum_{i \in \mathscr A} \int_{\ell_i}^{r_i} x_i\cdot \mu(\{x_i\})\dd p\\
		=& \sum_{i \in \mathscr A} x_i~\mu(\{x_i\})~(r_i - \ell_i)\\
		=& \sum_{i \in \mathscr A} x_i~\mu(\{x_i\})^2\\
		J_2=& 2K.
		\end{align*}
		
		~
		
		Let us deal with $J_1$. To do so, we make the same computations than in the “simplified” proof (equations \eqref{proof_eq_gini_simplified:eq1} to  \eqref{proof_eq_gini_simplified:eq3}). We need to be careful about the endpoints of the integrals.
		
		\begin{align}
			J_1 &= \int_{p=0}^{1^-} \int_{u=0}^\infty u \1_{u < Q(p)} \dd \mu(u) \dd p \nonumber \\
			&= \int_{p=0}^{1^-} \int_{u=0}^\infty u \1_{F(u) < p} \dd\mu(u) \dd p\nonumber\\
			&= \int_0^\infty \int_{p=F(u)^+}^{1} u \dd p \dd\mu(u)\nonumber\\
			J_1 &= \int_0^\infty u\cdot(1-F(u)) \dd\mu(u) \label{gini_J1_1}\\
			&= \int_0^\infty u \int_{s=u^+}^\infty \dd\mu(s) \dd\mu(u) \nonumber\\
			J_1 &= \iint_{0 \le u < s} u \dd\mu(u)\dd\mu(s) \label{gini_J1_2}.
		\end{align}
		
		From \eqref{gini_J1_1}, we also deduce that:
		\begin{align*}
			J_1 &= \int_0^\infty u\dd\mu(u) - \int_0^\infty \int_{s=0}^u u\dd\mu(s) \dd\mu(u)\\
			J_1 &= m - \iint_{0 \le s \le u} u \dd\mu(s)\dd\mu(u).
		\end{align*}
		
		\paragraph{Third split.} We write ${J_1 = m - J_{11} - J_{12}}$, with:
		\begin{align*}
			J_{11} &\eqdef \iint_{0 \le s < u} u \dd\mu(s) \dd\mu(u)~;\\
			J_{12} &\eqdef \iint_{0 \le s = u} u \dd\mu(s) \dd\mu(u).
		\end{align*}
		
		Using Fubini’s theorem, we directly find 
		
		\begin{align*}
			J_{12} &= \int_{0}^\infty u \int_0^\infty \1_{\{u\}}(s) \dd \mu(s) \dd \mu(u)\\
			&= \int_0^\infty u \mu(\{u\}) \dd\mu(u)\\
			&= \sum_{i\in\mathscr A} x_i \mu(\{x_i\})^2\\
			J_{12} &= 2K.
		\end{align*}
		Then, performing the permutation of variables $u \leftrightarrow s$, we write:
		\[J_{11} = \int_{0 \le u < s} s \dd\mu(u) \dd\mu(s).\]
		
		We get:
		\begin{equation}
			J_1 = m - \iint_{0\le u < s} s \dd\mu(s)\dd\mu(u) - 2K. \label{gini_J1_4}
		\end{equation}
		
		\paragraph{Summation.} Then, summing up the equations \eqref{gini_J1_2} and \eqref{gini_J1_4} gives us:
		\[ J_1 = \frac{m}{2} - \frac12 \iint_{0 \le u < s} (s-u) \dd\mu(s)\dd\mu(u) - K. \]
		
		The last thing to notice is that, by symmetry, we have:
		\begin{align*}
			m\cdot G(\mu) &= \frac12 \iint_{\R_+^2} |s-u| \dd\mu(u)\dd\mu(s)\\
			m\cdot G(\mu) &= \iint_{0\le u<s} (s-u)\dd\mu(s) \dd\mu(u).
		\end{align*}
		
		Thus,
		\[J_1 = \frac m 2 - \frac m 2 \cdot G(\mu) - K.\]
		
		Hence,
		\[I = J_1 + J_2 - K = \frac m 2 - \frac m 2\cdot G(\mu)\]
		which concludes the proof in the general case.
	\end{proof}
	
	\subsubsection{Approximating atoms with diffuse measures~\ref{theorem_alternate_def_gini}}
	Corollary \ref{convergence_indicators} gives another path for proving theorem~\ref{theorem_alternate_def_gini}.
	
	Indeed, the proof of theorem~\ref{theorem_alternate_def_gini} could be simplified in the case where $\mu$ is nonatomic (section~\ref{second_proof_equivalence_gini_simplified}). The general case can then be proven using approximations with nonatomic measures.
	
	We give the sketch of an alternative proof based on this idea. For what follows, we let $G$ the Gini index as we defined with the expectation ($\Exp[|X - X'|] / 2m$), and $\Gamma$ the alternative definition using Lorenz curve ($1 - 2\int_0^1 L(p) \dd p$):\nobreakpar
	\begin{enumerate}
		\item Prove that $G(\mu) = \Gamma(\mu)$ as soon as $\mu$ has no atom. This can be done either with the simplified proof of section~\ref{second_proof_equivalence_gini_simplified} or using \cite{kendall1945}'s proof.
		
		\item Now take $\mu \in \M$, with no asumption. We approximate $\mu$ with nonatomic measures as follows. Consider a probability space $(\Omega, \mathscr F, \prob)$ with random variables $X, X' \sim \mu$ and $\eps_n, \eps'_n \sim \mathscr U([0, 2^{-n}])$ (uniform over $[0, 2^{-n}]$), all of them being mutually independent. Let $\mu_n$ the distribution of $X + \eps_n$ and $\nu_n$ the distribution of $\eps_n$.
		
		\item For all $a \in \R_+$, by Fubini, we have:
		\begin{align*}
			&\mu_n(\{a\})\\
			=& \int_{\Omega} \1_{(X(\omega) + \eps_n(\omega) = a)} \dd\prob(\omega)\\
			=& \iint_{\R_+ \times \R_+} \1_{(x + \eps = a)} \dd(\mu\otimes \nu_n)(x,\eps)\\
			=& \int_{x \in \R_+} \int_{\eps \in \R_+} \1_{(x+\eps=a)} \dd\nu_n(\eps) \dd\mu(x)\\
			=& \int_{x \in \R_+} \int_{(\eps = a-x)} \dd\nu_n(\eps) \dd\mu(x)\\
			=& \int_{\R_+} 0 \dd\mu(x)\\
			=& 0.
		\end{align*}
		
		Hence, $\mu_n$ is atomless. Hence, $G(\mu_n) = \Gamma(\nu_n)$.
		
		\item $(X + Y_n)_{n \in \N}$ converges $\prob$-almost surely to $X$; so does it in distribution, i.e. $(\mu_n)_{n \in \N}$ weakly converges to $\mu$. Furthermore, $m_{\mu_n} = m_\mu + 2^{-n-1} \xrightarrow[n \to \infty]{} m_\mu$. Hence, $\mu_n \xrightarrow[n \to \infty]{\Wone} \mu$.
		
		The proof of corollary~\ref{convergence_indicators} shows that $\Gamma(\mu_n) \xrightarrow[n \to \infty]{} \Gamma(\mu)$.
 
		\item Consider the random variables:
		\[ Z_n \eqdef |X + Y_n - X' - Y'_n|. \]
		One has: $Z_n \xrightarrow[n \to \infty]{} |X - X'|$, $\prob$-a.s. Furthermore, for every $n \in \N$, $0 \le Z_n \le X + 2$. As $\mu \in \mathbf M$, $X$ has finite expectation; so does $X+2$. Hence, the dominated convergence theorem states that $(Z_n)_{n \in \N}$ $\Exp[Z_n] \xrightarrow[n\to\infty]{} \Exp[|X - X'|]$. We have then proven that $G(\mu_n) \xrightarrow[n \to \infty]{} G(\mu)$.
		
		\item Hence, $G(\mu) = \Gamma(\mu)$.\hfill $\scriptscriptstyle\;\blacksquare$
	\end{enumerate}
	
	The same reasoning exactly can be performed for the equivalence of both definitions of Hoover index (proposition \ref{theorem_alternate_def_hoover_mean}). Indeed, as the tricky part of the proof of this theorem is to deal with a possible atom of $\mu$ in $\{m_\mu\}$, one can first prove the proposition assuming $\mu$ is diffuse, and then approximating it by the same $\mu_n$ we used for Gini.
	
	\subsection{Convergence of the Hoover index of a sample}
	
	We provide a different, less straightforward but more elementary proof of point~3 of application~\ref{convergence_corollary_sampling}, that does not need the framework of Lorenz functions nor theorem~\ref{convergence_indicators} . This proof is courtesy of  Guillaume Conchon--Kerjean.
	
	\begin{proof}
		Let $n \in \N^*$. We have:
		\[ H_{\hat\mu_n} = \frac{\sum_{i=1}^n |X_i - m_{\hat \mu_n}|}{2n m_{\hat \mu_n}} \]
		
		The inverse triangular inequality states that for every $i \in \{1, \dots, n\}$,
		\[ \bigl||X_i - m_{\hat \mu_n}| - |X_i - m_\mu|\bigr| \le | m_{\hat \mu_n} - m_\mu|. \]
		
		Thus, summing these inequalities and applying the triangular inequality,
		
		\begin{align*}
			~ &\left|\frac1n \sum_{i=1}^n |X_i - m_{\hat\mu_n}| - \frac1n \sum_{i=1}^n |X_i - m_\mu| \right|\\
			\le ~ & \frac 1 n \sum_{i=1}^n \bigl| \left|X_i - m_{\hat \mu_n}\right| - \left|X_i - m_\mu\right| \bigr| \\
			\le ~ & | m_{\hat \mu_n} - m_\mu|.
		\end{align*}
		
		On the one hand, the law of large numbers applied to random variables $(X_n)_{n \in \N^*}$ states that almost surely,
		\[ | m_{\hat \mu_n} - m_\mu| \xrightarrow[n \to \infty]{} 0. \]
		
		On the other hand, the law of large numbers applied to random variables $(|X_n - m_\mu|)_{n \in \N^*}$ states that a.s.,
		\[ \frac1n \sum_{i=1}^n |X_i - m_\mu| \xrightarrow[n \to \infty]{} \Exp[|X_1 - m_\mu| ]. \]
		
		Hence, a.s.,
		\[ \frac 1 n \sum_{i=1}^n |X_i - m_{\hat \mu_n}| \xrightarrow[n\to \infty]{} \Exp[|X_1 - m_\mu|].\]
		
		Thus, a.s. $H(\hat\mu_n) \xrightarrow[n \to \infty]{} H(\mu)$.
	\end{proof}
	
	\subsection{Miscellanous tool theorems}
	\subsubsection{Uniform integrability of samples}
	\label{appendix_alternate_unif_int_sampling}
	
	We give a different proof of the corollary~\ref{unif_int_sampling}, not using either Glivenko--Cantelli’s theorem nor Scheffé’s lemma.
	
	\begin{proof}
		Let $(\Omega, \mathscr F, \prob)$ the departure space of the random variables $X_i$’s. For every $\omega \in \Omega$ and $n \in \N^*$, we call $\hat\mu_{n, \omega}$ the empirical measure over $(\R_+, \Borel)$ associated with the sample $X_1(\omega), \dots, X_n(\omega)$, i.e.:
		\[ \hat\mu_{n, \omega} = \frac1n \sum_{i=1}^n \delta_{X_i(\omega)}. \]
		
		Let $A$ the set of $\omega \in \Omega$ such that the collection $(\hat\mu_{n,\omega})_{n \in \N^*}$ is u.i. It suffices to find a measurable subset of $A$ of probability 1.
		
		For $\omega \in \Omega$ fixed, let:
		\[ \begin{array}{rcrcl}
			f_{n, \omega} &:& \R_+ & \longrightarrow & \R_+\\
			&& x & \longmapsto & \displaystyle \frac1n\sum_{k=1}^n X_i(\omega) \1_{(X_i(\omega) > x)}
		\end{array} \]
		
		In other words, $f_{n,\omega}(x) = m_{}$ is the expectation of $Y\mathbf 1_{\{Y > x\}}$, where $Y$ is any random variable of distribution $\hat\mu_{n,\omega}$. (Where $Y$ is defined on a probability space which has nothing to do with $(\Omega, \mathscr F, \prob)$.) Thus, by definition,
		\[ A = \left\{ \omega \in \Omega : \sup_{n \in \N} f_{n, \omega}(x) \xrightarrow[x \to +\infty]{} 0 \right\}. \]
		
		By lemma~\ref{lemma_charac_ui_subsets},
		\[ A = \left\{\omega \in \Omega: \begin{array}{l} \forall \eps > 0, \exists N \in \N,\\ \forall n \ge n, f_{n, \omega}(x) < \eps\end{array} \right\}.\]
		
		Let:
		\[ \begin{array}{rcrcl}
			g &:& \R_+ & \longrightarrow & \R_+\\
			&& x & \longmapsto & \Exp\left[X_1 \1_{(X_1 > x)}\right],
		\end{array} \]
		and for all $x \in \R_+$,
		\[ B_x \eqdef \{ \omega \in \Omega : f_{n, \omega}(x) \xrightarrow[n \to \infty]{} g(x) \}. \]
		
		For all $x \in \R_+$, the law of large numbers applied to the i.i.d. random variables $\left(X_n \1_{(X_1 > n)}\right)_{n \in \N^*}$ ensures that $\prob(B_x) = 1$. Thus, the set \[C\eqdef \bigcap_{n \in \N} B_n\] has probability 1.
		
		Now, we fix $\omega \in C$, and it suffices to prove that $\omega \in A$. Let $\eps > 0$.
		
		As $X_i$'s have finite expectation, by dominated convergence theorem, ${g(a) \xrightarrow[a \to \infty]{} 0}$. There exists $K \in \N$ such that $g(K) < \eps$. $\omega \in B_{K}$, so there exists $N \in \N^*$ such that for every $n \ge N$, $|f_{n, \omega}(K) - g(K)| < \eps$.
		
		Then for every $n \ge N$, $f_{n, \omega}(K) < 2\eps$. Hence, $\omega \in A$.
	\end{proof}
	
	\subsubsection{$\Wone$ convergence implies weak convergence}
	We now give another proof of the following result, used in appendix~\ref{appendix_W1_cv}.
	
	\begin{proposition}
		Let $(\mu_n)_{n \in \N} \in \M^\N$ and $\mu_\infty \in \M$. If $\mu_n \xrightarrow[n \to \infty]{\Wone} \mu_\infty$, then $\mu_n \xrightarrow[n \to \infty]{\weak} \mu_\infty$.
	\end{proposition}
	
	\begin{proof}
		Apply the following lemma to the $Q_{\mu_n}$’s and the characterization of weak convergence by quantiles (proposition~\ref{characterization_weak_convergence}).
	\end{proof}
	
	\begin{lemma}\label{cv_L1_nondec_implies_cv_ae}
		Let $I \subseteq \R$ an interval and consider the measured space $(I, \Borel, \Leb)$. Let $E$ the set of nondecreasing, integrable functions $E \longrightarrow \R$.
		
		If $\|f_\infty - f_n\|_1 \xrightarrow[n \to \infty]{} 0$, then $f_n \xrightarrow[n \to \infty]{} f$ $\Leb$-almost everywhere.
	\end{lemma}
	
	\begin{proof} We shall first prove the following claim: for all $f \in E$, $x \in I$ such that $f$ is continuous at $x$  and $\eps > 0$, there exists $\delta > 0$ such that for all $g \in E$,
		\[ |g(x) - f(x)| > \eps \implies \|g-f\|_{1} > \delta. \]
		
		Indeed, assume that $g(x) > f(x) + \eps$. Then:
		\begin{itemize}
			\item Either $g(x) > f(x) + \eps$. If so, let $\eta > 0$ such that $f(x + \eta) < f(x) + \frac\eps2$. If $t \in [x, x + \eta)$, then 
			\[ g(t) - f(t) > g(x) - f(x) - \frac\eps2 > \frac \eps2. \] 
			Hence, $\|g-f\|_{\Lone} > \eta\frac{\eps}{2}$. By chosing $\delta \eqdef \eta\frac{\eps}{2}$, the assertion holds.
			\item Or $g(x) < f(x) - \eps$. Then, the same reasoning leads to the same conclusions.
		\end{itemize}
	
		Now, assume $\|f_\infty - f_n\|_{1} \xrightarrow[n \to \infty]{} 0$. By the claim, $f_n(x) \xrightarrow[n \to \infty]{} f_\infty(x)$ for all $x \in I$ where $f_\infty$ is continuous. Since $f_\infty$ is nondecreasing, this is the case $\Leb$-almost everywhere.
	\end{proof}	
\end{multicols}

\newpage

\tableofcontents

\end{document}